\documentclass[12pt,pstricks,usenames,dvipsnames]{amsart}
\usepackage{amsmath,amssymb,amsthm}
\usepackage{mathtools}
\usepackage{stmaryrd} 
\usepackage[inline]{enumitem}
\usepackage{graphicx,setspace}
\usepackage{tikz}
\usetikzlibrary{arrows}
\usepackage{wasysym}
\usepackage[makeroom]{cancel}
\usepackage{fullpage}
\usepackage{comment}
\usepackage{enumitem}
\usepackage{caption}
\usepackage{mathtools}
\usepackage{url}
\allowdisplaybreaks
\usepackage{amsmath,amsfonts,amsthm,verbatim, graphicx,tikz}
\usepackage{amssymb, mathrsfs, mathtools, centernot}
\usetikzlibrary{shapes.geometric,decorations.pathreplacing,angles,quotes}
\newtheorem{thm}{Theorem}[section]
\newtheorem{lem}[thm]{Lemma}
\newtheorem{lemma}[thm]{Lemma}

\newtheorem{cor}[thm]{Corollary}

\theoremstyle{definition}
\newtheorem{defn}{Definition}[section]

\newtheorem{ex}{Example}
\newcommand{\rank}{\text{rank}}
\newcommand{\nul}{\text{nul}}

\newcommand{\fcent}{\text{\normalfont cent}_{r,c}}
\definecolor{lightgray}{rgb}{0.75,0.75,0.75} 
\begin{document}
\title{The Sortability of Graphs and Matrices \\under Context Directed Swaps}

\author[C. Brown]{C.\ Brown}
\address{Department of Mathematics\\
University of Arizona\\ Tucson, AZ 85721, USA}
\email{\url{colbybrown@email.arizona.edu}}

\author[C.S. Carrillo Vazquez]{C.S.\ Carrillo Vazquez}
\address{Department of Computer Science and Department of Mathematics\\
University of Rochester\\ Rochester, NY 14627, USA}
\email{\url{ccarril2@u.rochester.edu}}

\author[R. Goswami]{R.\ Goswami}
\address{Department of Electrical Engineering and Computer Science\\
University of Michigan -- Ann Arbor\\ Ann Arbor, MI 48109, USA}
\email{\url{gorashmi@umich.edu}}

\author[S. Heil]{S.\ Heil}
\address{Department of Mathematics\\ 
Washington University in St. Louis\\ St. Louis, MO 63130, USA}
\email{\url{sheil@wustl.edu}}

\author[M. Scheepers]{M.\ Scheepers}
\address{Department of Mathematics\\ 
Boise State University\\ Boise, ID 83725, USA}
\email{\url{mscheepe@boisestate.edu}}

\thanks{Supported by the National Science Foundation under the Grant number DMS-1659872.}
\thanks{$^{\S}$ Corresponding Author: mscheepe@boisestate.edu}
% NEED TO CHANGE: 
%-------------------------------------------
\maketitle
% \textbf{cds}, when referring to it (in English)
% \textbf{cds}, when using it as a function
% $\textbf{cds}$, when using it in the game
%-------------------------------------------
%--- ABSTRACT ------------------------------
\begin{abstract} 
The study of sorting permutations by block interchanges has recently been stimulated by a phenomenon observed in the genome maintenance of certain ciliate species. The result was the identification of a block interchange operation that applies only under certain constraints. Interestingly, this constrained block interchange operation can be generalized naturally to simple graphs and to an operation on square matrices. This more general context provides numerous techniques applicable to the original context. In this paper we consider the more general context, and obtain an enumeration, in closed form, of all simple graphs on n vertices that are ``sortable" by the graph analogue of the constrained version of block interchanges. We also obtain asymptotic results on the proportion of graphs on $n$ vertices that are so sortable.
\end{abstract}
%-------------------------------------------
%\tableofcontents
%-------------------------------------------
\section{Introduction}
%-------------------------------------------
%-------------------------------------------

Sorting is one of the fundamental steps in the preparation of data for various data-mining operations. The requirement of efficiency of the sorting process has lead to the invention and study of several sorting methods. Among these is the block interchange technique: Abstractly described, a repetition-free list of the numbers $1$ through $n$ in some order is given. We shall call such a list a permutation of the numbers $1$ through $n$.The objective is to sort this list to result in the numbers $1$ through $n$ listed in their canonical order $1,\; 2,\; \cdots,\; n-1,\; n$. A block interchange of the original list consists of identifying two non-overlapping intervals of the list, and to then interchange these two intervals while retaining the relative order of items in each of the two intervals.

\begin{ex}\label{ex:blocswap} Consider the list $\lbrack 8,\; 3,\; 1,\; 4,\; 6,\; 5,\; 2,\; 9,\; 7\rbrack$. The result of interchanging the intervals $\lbrack 3,\; 1,\; 4\rbrack$ and $\lbrack 5,\; 2\rbrack$ is the new list $\lbrack 8,\; 5,\; 2,\;6,\;1,\;3,\;4,\;9,\;7\rbrack$.

\end{ex} 

Two fundamental questions regarding block interchanges are whether a permutation can be sorted to the canonical order by applying only block interchanges, and if so, to sort the permutation using the minimum possible number of block interchanges. In \cite{christie} Christie showed that every permutation can be sorted by a sequence of consecutive block interchanges, and identified a special type of block interchange that is critical to achieving the sorting in the minimum possible number of steps. The type of block interchange identified in \cite{christie} turns out to be a special case of a broader class of constrained block interchanges that were identified in a different context in modeling genome maintenance activity in certain species of a class of single cell organisms called ciliates \cite{per}.  As in prior work, these special block interchanges will be called context directed swaps in this paper, and will be defined below.

There is a significant body of work on sorting permutations by context directed swaps. %, including [\textcolor{blue}{References}]. 
Of the various approaches to this subject matter, translations to graph theoretic and linear algebraic methods have been especially useful. In this paper we indicate how to faithfully translate the context directed swap operation on permutations to corresponding operations on finite graphs and appropriate square matrices. It turns out that the corresponding operation on graphs is strongly related to the previously studied operation of edge complementation \cite{BH} was independently discovered in \cite{BHH}, where in Section 7 it is called edge reduction, and in \cite{p4}. Hints of the edge complementation operation also appear on p. 107 of \cite{Geelen}.The corresponding operation on square matrices in turn, is directly related to the elementary pivot operation (see page 106  of \cite{Geelen}) on a matrix and to the classical Schur complement \cite{BH, BHH, Geelen}.  We will exploit the body of knowledge in graph theory and linear algebra to enumerate, for each $n$, the set of simple finite graphs on $n$ vertices that are ``sortable" (to be defined later) by the graph theoretic analogue of context directed swaps. No corresponding closed formula is known for the  $n$-term permutations that are sorted by context directed swaps. When considering for each $n$ the proportion of simple graphs on $n$ vertices that are sortable by the graph theoretic analog of context directed swaps, an interesting phenomenon emerges: For even values of $n$ this proportion converges, and for odd values of $n$ the proportion also converges, but the limiting values for these two cases are different. A similar phenomenon has been observed for permutations sortable by context directed swaps, as the reader can verify from the data in the sequence A249165 available at  \cite{oeis}.

The rest of the paper is organized as follows: 
In section \ref{sec2}, some of the background definitions and notations will be given. The formal definitions of permutations and context directed swaps will be provided, as well as the description of the multiple representations of the permutations. In section \ref{sec3}, the results of applying \textbf{cds} to an overlap graph, \textbf{gcds}, will be presented. In particular, this section will discuss the properties of \textit{Eulerian graphs} after applying \textbf{gcds} on them.  In Section 6 we determine the complexity of a decision problem regarding graphs associated with permutations. In Section 7 we examine the sortabilty of finite simple graphs and square matrices.  The results of this section leads to the enumerative results of Section 8, and the asymptotic results of Section 9. %A simplified definition of applying \textbf{cds} on a matrix (abbreviated as \textbf{mcds}) will be given, in contrast to the description of a similar operation provided in \cite{p4}. %[\textcolor{blue}{Jansen, 2014}]. 
%Section \ref{sec4} will give sortability criteria for permutations being represented by graphs or matrices.

\section{Acknowledgements}
%-------------------------------------------
%-------------------------------------------
This research was supported by the National Science Foundation Grant No. DMS-1659872 and Boise State University. We also thank Professor Robert Brijder for a personal communication in which the connection with the Schur complement was pointed out.

%-------------------------------------------
%-------------------------------------------
\section{Preliminaries}\label{sec2}
%-------------------------------------------
There are two main types of graphs for a given permutation. One of these is the \textit{cycle graph}, which provides a sortability criterion through what is called the \textit{strategic pile}. The second graph is the \textit{overlap graph}, which is constructed from what is called the \textit{breakpoint graph}. The connections between these descriptions may be extended by the matrix representation of permutations. The matrix of a permutation that directly corresponds to the overlap is the \textit{adjacency matrix of the overlap graph}.

Even though \textbf{cds} originates from the field of biology, the results presented here are purely mathematical. Therefore, graphs and matrices that do not correspond to physically permissible permutations can be studied. This is because all graphs (and thus the matrices from the graphs) do not necessarily have a corresponding permutation.

%-------------------------------------------
\subsection{CDS and Associated Objects}
%-------------------------------------------
\subsubsection{Permutations, Pointers, and Context Directed Swaps}
%-------------------------------------------
\begin{defn}
A \textit{permutation}, $\pi = [a_1, ..., a_n] \in S_n$ is an array of the integers between $1$ and $n$ inclusive, in any order, such that there are no repeats among the integers. 
\end{defn}

Referring to the elements of permutations, each integer $i$ has a \emph{left pointer}, $(i-1,~i)$, and a \emph{right pointer}, $(i,~i+1)$. An integer $i$ along with its pointers is represented as $^{(i-1,i)}~i~^{(i,i+1)}$.
%-------------------------------------------
\begin{ex}
The permutation $\pi_1 = [3,2,5,1,4]$ can be expressed as
\begin{center}
$[$ $^{(2,3)}3^{(3,4)}$,~ $^{(1,2)}2^{(2,3)}$,~ $^{(4,5)}5^{(5,6)}$,~ $^{(0,1)}1^{(1,2)}$,~ $^{(3,4)}4^{(4,5)}]$.
\end{center}
\end{ex}
%-------------------------------------------
A permutation can be \textit{framed}, in which $0$ is added to the left of the open parenthesis and $(n+1)$ is added to the right of the close parenthesis. In a framed permutations where the pointers are explicitly given, the left pointer of $0$ and the right pointer of $n + 1$ are excluded. \\

The definition or \textit{context directed swap}, the main focus of this paper, may now be stated.
%-------------------------------------------
\begin{defn}
Given a permutation 
\begin{center}
$\pi = \Big[\big\{\alpha_{1}$ $^p\big\}$ $\big\{\alpha_{2}$ $^q\big\}$ $\big\{\alpha_{3}\big\}$ $\big\{$ $^p\alpha_{4}\big\}$ $\big\{$ $^q\alpha_{5}\big\}\Big]$
\end{center}
where each $\alpha_i$ is a subsection (or block) of the permutation and $p = (x,~x+1)$ and $q = (y,~y+1)$ are the pointers of arbitrary entries $x$ and $y$ in the permutation, we define the operation \textit{context directed swaps} on the permutation $\pi$ with context $p$ and $q$ as,
\begin{center}
$\text{\textbf{cds}}_{\{p,q\}}(\pi) = \Big[\big\{\alpha_{1}$ $^p$\big\} \big\{$^p$ $\alpha_{4}$\big\} \big\{$\alpha_{3}$\big\} \big\{$\alpha_{2}$ $^q$\big\} \big\{$^q$ $\alpha_{5}\big\}\Big]$.
\end{center}
\end{defn}
%-------------------------------------------
\begin{ex}
For example, for the permutation $\pi_1 = [3,2,5,1,4]$, we have that \[ \textbf{cds}_{\{(1,2),(4,5)\}} (\pi_1) = [3,4,5,1,2]. \]
\end{ex}
%-------------------------------------------
The operation \textbf{cds} interchanges the blocks $\alpha_{2}$ and $\alpha_{4}$, including the pointers that were part of each block. Notice that $\text{\textbf{cds}}_{\{p,q\}}$ can only be applied if the permutation has the pointers $p$ and $q$ in the following order: $``...^p...^q...^p...^q..."$. Also, the two $p$ pointers are adjacent to each other in the image of the permutation under \textbf{cds} (as are the two $q$ pointers). This paper will look at the different limitations and constraints that \textbf{cds} has on the various representations of permutations.\\
%-------------------------------------------
\subsubsection{Cycle Graph, Alternating Cycles, and Strategic Pile}
%-------------------------------------------
There are two distinct graph representations of a permutation, the first of which is the \textit{cycle graph}. 
%-------------------------------------------
\begin{defn}
The \textit{cycle graph} of a permutation $\pi = [a_1,...,a_n]$ is the directed graph $CG(\pi) = (V, E)$, where
\begin{align*}
	V &= \{0,~1,~...,~n,~(n+1)\}\text{, and}\\
E &= E_d \cup E_b\\
\end{align*}
with
\begin{align*}
	E_b &= \{\{i,i+1\}~|~ 0 \le i \le n \}\\
    E_d &= \{\{a_{i+1},a_i\}~|~ 1 \le i \le n \} \cup \{\{(n+1),a_n\},\{a_1,0\}\}
\end{align*}.
\end{defn}
%-------------------------------------------
The edges are divided into two sets, where $E_b$ is the set of \textit{bold} or \textit{black} edges and $E_d$ is the set of \textit{dotted} or \textit{gray} edges of the cycle graph. The bold edges are connected per each consecutive number (from smallest to largest), while the dotted edges are connected per each consecutive element in the permutation, starting at the end of the permutation and moving backward.
%-------------------------------------------
\begin{ex}
The cycle graph of the permutation $\pi_2 = [5,3,1,6,2,7,4]$ is shown in Figure 1 below.
\begin{center}
%\begin{figure}
\begin{tikzpicture}[scale=.6, every node/.style={scale=0.6}]
    %0452617389
    \draw[black] (-.625,-.3) node(0)[shape=circle,draw=black,inner sep=0.5pt][anchor=north,scale=1.5]{0};
    \filldraw[black] (.875,-.3) node(5)[shape=circle,draw=black,inner sep=0.5pt][anchor=north,scale=1.5]{5};
    \filldraw[black] (2.375,-.3) node(3)[shape=circle,draw=black,inner sep=0.5pt][anchor=north,scale=1.5]{3};
    \filldraw[black] (3.875,-.3) node(1)[shape=circle,draw=black,inner sep=0.5pt][anchor=north,scale=1.5]{1};
    \filldraw[black] (5.375,-.3) node(6)[shape=circle,draw=black,inner sep=0.5pt][anchor=north,scale=1.5]{6};
    \filldraw[black] (6.875,-.3) node(2)[shape=circle,draw=black,inner sep=0.5pt][anchor=north,scale=1.5]{2};
    \filldraw[black] (8.375,-.3) node(7)[shape=circle,draw=black,inner sep=0.5pt][anchor=north,scale=1.5]{7};
    \filldraw[black] (9.875,-.3) node(4)[shape=circle,draw=black,inner sep=0.5pt][anchor=north,scale=1.5]{4};
    \filldraw[black] (11.375,-.3) node(8)[shape=circle,draw=black,inner sep=0.5pt][anchor=north,scale=1.5]{8};
      
    \foreach \z in {0,...,7}
    	{\pgfmathtruncatemacro{\a}{\z+1}
    	\draw[black, thick, out=90, in=90, ->] (\z.east) to (\a.west);}
        
	\draw[dotted, thick, ->] (8) to (4);
    \draw[dotted, thick, ->] (4) to (7);
    \draw[dotted, thick, ->] (7) to (2);
    \draw[dotted, thick, ->] (2) to (6);
    \draw[dotted, thick, ->] (6) to (1);
    \draw[dotted, thick, ->] (1) to (3);
    \draw[dotted, thick, ->] (3) to (5);
    \draw[dotted, thick, ->] (5) to (0);
\end{tikzpicture}\\
	%\caption{The cycle graph of $\pi_2 = [5,3,1,6,2,7,4]$}
%\end{figure}
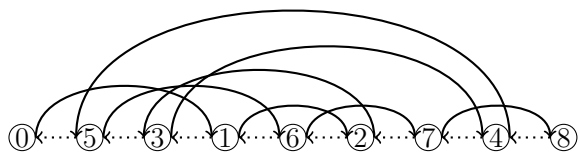
FIGURE 1. The Cycle Graph of $\pi_2 = [5,3,1,6,2,7,4]$
\end{center}
\end{ex}
%-------------------------------------------
If walks are made on the cycle graph where consecutive edges in the walks alternate between belonging to $E_b$ and $E_d$, then the disjoint sets of nodes in the walks form the \textit{alternating cycles} of the permutation.
\begin{center}
\begin{tikzpicture}
%(2*cos(90 - 45*z))+2, 2*sin(90 - 45*z))+2)
% Vertices
\draw[black] (2,4) node(0)[shape=circle,draw=black,inner sep=0.5pt][anchor=center]{0};
\draw[black] (3.4,3.4) node(1)[shape=circle,draw=black,inner sep=0.5pt][anchor=center]{1};
\draw[black] (4,2) node(2)[shape=circle,draw=black,inner sep=0.5pt][anchor=center]{3};
\draw[black] (3.4,0.6) node(3)[shape=circle,draw=black,inner sep=0.5pt][anchor=center]{4};
\draw[black] (2,0) node(4)[shape=circle,draw=black,inner sep=0.5pt][anchor=center]{7};
\draw[black] (0.6,0.6) node(5)[shape=circle,draw=black,inner sep=0.5pt][anchor=center]{8};
\draw[black] (0,2) node(6)[shape=circle,draw=black,inner sep=0.5pt][anchor=center]{4};
\draw[black] (0.6,3.4) node(7)[shape=circle,draw=black,inner sep=0.5pt][anchor=center]{5};
% Edges
\draw[black, ultra thick, shorten >=.2cm,shorten <=.2 cm, ->] (0.center) -- (1.center);
\draw[dashed, ultra thick, shorten >=.2cm,shorten <=.2 cm, ->] (1.center) -- (2.center);
\draw[black, ultra thick, shorten >=.2cm,shorten <=.2 cm, ->] (2.center) -- (3.center);
\draw[dashed, ultra thick, shorten >=.2cm,shorten <=.2 cm, ->] (3.center) -- (4.center);
\draw[black, ultra thick, shorten >=.2cm,shorten <=.2 cm, ->] (4.center) -- (5.center);
\draw[dashed, ultra thick, shorten >=.2cm,shorten <=.2 cm, ->] (5.center) -- (6.center);
\draw[black, ultra thick, shorten >=.2cm,shorten <=.2 cm, ->] (6.center) -- (7.center);
\draw[dashed, ultra thick, shorten >=.2cm,shorten <=.2 cm, ->] (7.center) -- (0.center);
% Vertices
\draw[black] (7,4) node(8)[shape=circle,draw=black,inner sep=0.5pt][anchor=center]{1};
\draw[black] (8.4,3.4) node(9)[shape=circle,draw=black,inner sep=0.5pt][anchor=center]{2};
\draw[black] (9,2) node(10)[shape=circle,draw=black,inner sep=0.5pt][anchor=center]{6};
\draw[black] (8.4,0.6) node(11)[shape=circle,draw=black,inner sep=0.5pt][anchor=center]{7};
\draw[black] (7,0) node(12)[shape=circle,draw=black,inner sep=0.5pt][anchor=center]{2};
\draw[black] (5.6,0.6) node(13)[shape=circle,draw=black,inner sep=0.5pt][anchor=center]{3};
\draw[black] (5,2) node(14)[shape=circle,draw=black,inner sep=0.5pt][anchor=center]{5};
\draw[black] (5.6,3.4) node(15)[shape=circle,draw=black,inner sep=0.5pt][anchor=center]{6};
% Edges
\draw[black, ultra thick, shorten >=.2cm,shorten <=.2 cm, ->] (8.center) -- (9.center);
\draw[dashed, ultra thick, shorten >=.2cm,shorten <=.2 cm, ->] (9.center) -- (10.center);
\draw[black, ultra thick, shorten >=.2cm,shorten <=.2 cm, ->] (10.center) -- (11.center);
\draw[dashed, ultra thick, shorten >=.2cm,shorten <=.2 cm, ->] (11.center) -- (12.center);
\draw[black, ultra thick, shorten >=.2cm,shorten <=.2 cm, ->] (12.center) -- (13.center);
\draw[dashed, ultra thick, shorten >=.2cm,shorten <=.2 cm, ->] (13.center) -- (14.center);
\draw[black, ultra thick, shorten >=.2cm,shorten <=.2 cm, ->] (14.center) -- (15.center);
\draw[dashed, ultra thick, shorten >=.2cm,shorten <=.2 cm, ->] (15.center) -- (8.center);
\end{tikzpicture}\\
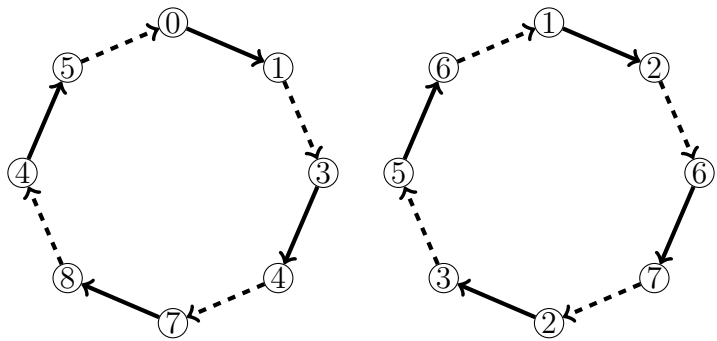
FIGURE 2. The Alternating Cycles of $\pi_2 = [5,3,1,6,2,7,4]$
\end{center}
%-------------------------------------------
The above suggests a representation of permutations as a product of cycles. The \textit{cycle notation} ($C_{\pi}$) of a permutation $\pi$ is the conjunction of the alternating cycles. Another way of representing $\pi = [a_1, \dots, a_n]$ is with $C_{\pi} = Y_{\pi}\circ X_n$ where $X_n = (0~1~2~\dots ~n)$ and $Y_{\pi} = (a_n~a_{n-1}~\dots ~a_1~0)$.
%-------------------------------------------
\begin{ex}
The cycle notation of $\pi_2 = [5,3,1,6,2,7,4]$ is $C_{\pi_2} = (4~7~2~6~1~3~5~0)(0~1~2~3~4~5~6~7) = (0~3~7~4)(1~6~2~5)$; in this case, we have that $Y_{\pi_2} = (4~7~2~6~1~3~5~0)$ and $X_7 = (0~1~2~3~4~5~6~7)$.
\end{ex}
%-------------------------------------------

The cycle notation and its related concepts, as used in this paper, are taken from \cite{p3}.  %\textcolor{blue}{[Adamyk, 2013]}.
%-------------------------------------------
\begin{defn}
If the cycle notation of a permutation $\pi$ is 
	$$C_{\pi} = (0~...~n~b_1~...~b_k)(...)...(...),$$ 
then the \textit{strategic pile} SP$(\pi)$ of the permutation is the set of elements in the same cycle of $n$ that follow $n$ in $C_{\pi}$; that is, SP$(\pi) = \{b_1,...,b_k\}$. If $0$ and $n$ appear in different cycles of $C_{\pi}$ or no elements follow $n$ in $C_\pi$, then $SP(\pi) = \emptyset$.
\end{defn}
%-------------------------------------------
The cycle notation of a permutation gives rise to a sortability criterion under \textbf{cds} for permutations.
%-------------------------------------------
\begin{thm}\cite{p3}%[Adamyk, 2013]
A permutation $\pi$ is \textbf{cds}-sortable if, and only if, $SP(\pi) = \emptyset$.
\end{thm}
%-------------------------------------------
The \textit{ordered strategic pile} of a permutation $\pi$, SP$^*(\pi)$, is the same as the strategic pile for $\pi$, where the order of the numbers that follow $n$ in the cycle notation is preserved. That is, $SP^*(\pi) = (b_1~...~b_k)$.\\
%-------------------------------------------
\begin{ex}
The ordered strategic pile of $\pi = [3,2,5,1,4]$, with cycle notation $C_\pi = (0,5,4,2)(1,3)$ is $SP^*(\pi) = (4,2)$. 
\end{ex}
%-------------------------------------------
\subsubsection{Breakpoint Graph and Overlap Graph}
%-------------------------------------------
The second graph representation of a permutation is the \textit{overlap graph}.
%-------------------------------------------
\begin{defn}
Given a permutation $\pi = [a_1,~...,~a_n]$, the \textit{breakpoint graph} of the permutation is the undirected graph BG$(\pi) = (V, E)$, where
\begin{align*}
V &= \{(i,i+1)_{dir} ~|~ 0\le i \le n\text{ and }dir \in \{ L, R \} \;\}\\
E &= \Big\{\{(i,i+1)_{L},(i,i+1)_{R}\} ~|~ 0\le i \le n \Big\}
\end{align*}
The vertices correspond to the pointers of the elements of the permutation,  where $(i,i+1)_{L}$ indicates a left pointer, and $(i,i+1)_{R}$ indicates a right pointer.

\end{defn}
%-------------------------------------------
\begin{ex} 
Let $\pi_3 = [4,5,2,6,1,7,3,8]$, and write the pointers 
{\footnotesize
\begin{center}
$0^{(0,1)}$ [ $^{(3,4)}$ $4$ $^{(3,4)},$ $^{(3,4)}$ $5$ $^{(3,4)},$ $^{(3,4)}$ $2$ $^{(3,4)},$ $^{(3,4)}$ $6$ $^{(3,4)},$ $^{(3,4)}$ $1$ $^{(3,4)},$ $^{(3,4)}$ $7$ $^{(3,4)},$ $^{(3,4)}$ $3$ $^{(3,4)},$ $^{(3,4)}$ $8$ $^{(3,4)}$ ] $^{(3,4)}9$.
\end{center}
}
Figure $3$ shows the breakpoint graph of $\pi_3$.
\end{ex}
%-------------------------------------------
\begin{center}
\begin{tikzpicture}[scale=.6, every node/.style={scale=0.6}]
    %0452617389
    \filldraw[black] (-.625,-.3) node(0)[anchor=north,scale=1.5]{0};
    \filldraw[black] (.875,-.3) node(0)[anchor=north,scale=1.5]{4};
    \filldraw[black] (2.375,-.3) node(0)[anchor=north,scale=1.5]{5};
    \filldraw[black] (3.875,-.3) node(0)[anchor=north,scale=1.5]{2};
    \filldraw[black] (5.375,-.3) node(0)[anchor=north,scale=1.5]{6};
    \filldraw[black] (6.875,-.3) node(0)[anchor=north,scale=1.5]{1};
    \filldraw[black] (8.375,-.3) node(0)[anchor=north,scale=1.5]{7};
    \filldraw[black] (9.875,-.3) node(0)[anchor=north,scale=1.5]{3};
    \filldraw[black] (11.375,-.3) node(0)[anchor=north,scale=1.5]{8};
    \filldraw[black] (12.875,-.3) node(0)[anchor=north,scale=1.5]{9}; 
   
    \filldraw[black] (0,0) circle (2pt) node(01r)[anchor=north east]{01}; 
    \filldraw[black] (.25,0) circle (2pt) node(34b)[anchor=north west]{34}; 
    \filldraw[black] (1.5,0) circle (2pt) node(45r)[anchor=north east]{45}; 
    \filldraw[black] (1.75,0) circle (2pt) node(45b)[anchor=north west]{45}; 
    \filldraw[black] (3,0) circle (2pt) node(56r)[anchor=north east]{56}; 
    \filldraw[black] (3.25,0) circle (2pt) node(12b)[anchor=north west]{12}; 
    \filldraw[black] (4.5,0) circle (2pt) node(23r)[anchor=north east]{23}; 
    \filldraw[black] (4.75,0) circle (2pt) node(56b)[anchor=north west]{56}; 
    \filldraw[black] (6,0) circle (2pt) node(67r)[anchor=north east]{67}; 
    \filldraw[black] (6.25,0) circle (2pt) node(01b)[anchor=north west]{01}; 
    \filldraw[black] (7.5,0) circle (2pt) node(12r)[anchor=north east]{12}; 
    \filldraw[black] (7.75,0) circle (2pt) node(67b)[anchor=north west]{67}; 
    \filldraw[black] (9,0) circle (2pt) node(78r)[anchor=north east]{78}; 
    \filldraw[black] (9.25,0) circle (2pt) node(23b)[anchor=north west]{23}; 
    \filldraw[black] (10.5,0) circle (2pt) node(34r)[anchor=north east]{34}; 
    \filldraw[black] (10.75,0) circle (2pt) node(78b)[anchor=north west]{78}; 
    \filldraw[black] (12,0) circle (2pt) node(89r)[anchor=north east]{89};  
    \filldraw[black] (12.25,0) circle (2pt) node(89b)[anchor=north west]{89};
    
    \foreach \z in {0,...,8}
             {\pgfmathtruncatemacro{\a}{\z+1}
              \draw[black, out=90, in=90] (\z\a r.north east) to (\z \a b.north west);}
\end{tikzpicture}\\
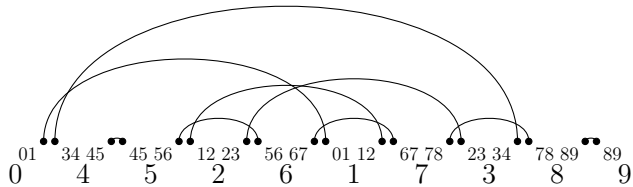
FIGURE 3. The Breakpoint Graph of $\pi_3 = [4,5,2,6,1,7,3,8]$
\end{center}
%-------------------------------------------
Using the graph above, we can construct the \textit{overlap graph} (referred as the interleaving graph in other sources [\textcolor{blue}{References}]) by taking the edges between each of the pointers of the breakpoint graph as the vertices, and constructing an edge in the overlap graph between these vertices for which edges intersect in the breakpoint graph. A formal definition of this notion is now given.
%-------------------------------------------
\begin{defn}
Given the breakpoint graph BG$(\pi)$ of a permutation $\pi$, the \textit{overlap graph} of the permutation $\pi$ is the undirected graph OG$(\pi) = (V, E)$ where
\begin{center}
	$V = \{(i,i+1) ~|~ 0\le i \le n\}$, ~and\\~\\
	$E = \{(i,i+1),(j,j+1)\}$ if there exists an intersection between $\{(i,i+1)_{L},(i,i+1)_{R}\}$ $\&$ $\{(j,j+1)_{L},(j,j+1)_{R}\}$ in BG$(\pi)$.
\end{center}
\end{defn}

\begin{defn}
Given a permutation $\pi \in S_n$, the \textit{root pointers} of $\pi$ are the two pointers $(0,1)$ and $(n,n+1)$ which cannot be used as $p$ or $q$ in a \textbf{cds} move $\textbf{cds}_{\{p,q\}}(\pi)$. Then, define the \textit{roots} of the overlap graph OG($\pi$) to be the vertices corresponding to the two root pointers in $\pi$. 
\end{defn}
%-------------------------------------------
An example of the overlap graph of a permutation can be seen in Figure 4.
\begin{center}
\begin{tikzpicture}
%(2*cos(90 - 40*z))+2, 2*sin(90 - 45*z))+2)
% Vertices
\draw[black] (4,6) node(0)[shape=diamond, draw=black, inner sep=0.5pt][anchor=center]{(0,1)};
\draw[black] (5.3,5.5) node(1)[shape=circle, draw=black, inner sep=0.5pt][anchor=center]{(1,2)};
\draw[black] (6,4.3) node(2)[shape=circle, draw=black, inner sep=0.5pt][anchor=center][anchor=center]{(2,3)};
\draw[black] (5.7,3) node(3)[shape=circle, draw=black, inner sep=0.5pt][anchor=center][anchor=center]{(3,4)};
\draw[black] (4.7,2.1) node(4)[shape=circle, draw=black, inner sep=0.5pt][anchor=center][anchor=center]{(4,5)};
\draw[black] (3.3,2.1) node(5)[shape=circle, draw=black, inner sep=0.5pt][anchor=center][anchor=center]{(5,6)};
\draw[black] (2.3,3) node(6)[shape=circle, draw=black, inner sep=0.5pt][anchor=center][anchor=center]{(6,7)};
\draw[black] (2,4.3) node(7)[shape=circle, draw=black, inner sep=0.5pt][anchor=center][anchor=center]{(7,8)};
\draw[black] (2.7,5.5) node(8)[shape=diamond, draw=black, inner sep=0.5pt][anchor=center]{(8,9)};
% Edges
\draw[black, thick, shorten >=.45cm, shorten <=.45cm] (0.center) -- (3.center);
\draw[black, thick, shorten >=.45cm, shorten <=.45cm] (0.center) -- (1.center);
\draw[black, thick, shorten >=.45cm, shorten <=.45cm] (0.center) -- (2.center);
\draw[black, thick, shorten >=.45cm, shorten <=.45cm] (0.center) -- (6.center);
\draw[black, thick, shorten >=.45cm, shorten <=.45cm] (1.center) -- (5.center);
\draw[black, thick, shorten >=.45cm, shorten <=.45cm] (1.center) -- (2.center);
\draw[black, thick, shorten >=.45cm, shorten <=.45cm] (1.center) -- (6.center);
\draw[black, thick, shorten >=.45cm, shorten <=.45cm] (2.center) -- (5.center);
\draw[black, thick, shorten >=.45cm, shorten <=.45cm] (2.center) -- (7.center);
\draw[black, thick, shorten >=.45cm, shorten <=.45cm] (3.center) -- (7.center);
\end{tikzpicture}\\
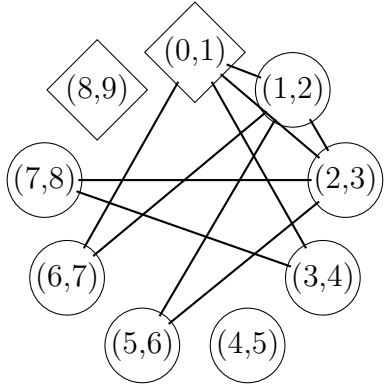
FIGURE 4. The Overlap Graph of $\pi_3 = [4,5,2,6,1,7,3,8]$ \\The root vertices are denoted with ``${\Large \lozenge}$" rather than with ``${\tiny \bigcirc}$"\\
% The "$\lozenge$ rather than with $\bigcirc$." was a suggestion given by Liljana.
%BEFORE: "using diamonds rather than circles."
\end{center}
%-------------------------------------------
\begin{defn}
Let $G_{n,r}$ denote the set of all $n$-vertex graphs such that $r$ of the $n$ vertices are roots. Thus, if $\pi \in S_n$, then $OG(\pi) \in G_{n+1,2}$, since all overlap graphs are two-rooted and the overlap graph of an $n$-element permutation must have $n+1$ vertices.
\end{defn}
\begin{defn}
A graph $G$ is \textit{Eulerian} if every vertex of $G$ has even degree.
\end{defn}
%-------------------------------------------
\begin{defn}
Given a simple graph $G = (V,E)$, let the \textit{neighborhood} of a vertex $v$ in the graph, denoted $N_G(v)$, be the set of all vertices $u \in V$ such that edge $\{ u,v \}$ is in $E$. 
\end{defn}
%-------------------------------------------
\begin{defn}
Given a simple graph $G = (V, E)$, a set $S \subseteq V$, and a vertex $v \in V$, let $\delta_S(v)$ denote the degree of $v$ with respect to $S$, or the number of edges between $v$ and vertices in $S$.
\end{defn}
%-------------------------------------------

Note that $v \not \in N_G(v)$ for all $v \in V$. \\

For a set $S$, denote the cardinality of $S$ by $|S|$. \\

%-------------------------------------------
\subsubsection{Adjacency and Precedence Matrices}
%-------------------------------------------
We will denote the entry in row $i$ and column $j$ of the matrix $A$ using the 
Let $A(i,j)$ be the entry in row $i$ and column $j$ of the matrix $A$, let $\vec{a_i}$ be the $i$th column vector of the matrix $A$, and let $v_i$ be the $i$th entry in the vector $\vec{v}$.

The \textit{adjacency matrix} of an overlap graph can now be defined as follows,
%-------------------------------------------
\begin{defn}
The \textit{adjacency matrix} of a graph $G = (V,E)$ with vertex set $V = \{v_1,v_2,\dots,v_n\}$ is the $n \times n$ matrix $A_{\pi}$ such that, for $1 \le i,j \le n$, 
$$
A_\pi(i,j) = \left\{
	\begin{array}{ll}
           1 & \quad \{v_i,v_j\} \in E \\
            0 & \quad \text{otherwise}
	\end{array}
\right.
$$
For simplicity, we define the adjacency matrix of a permutation $\pi$ to be the adjacency matrix of its overlap graph, OG($\pi$).

\end{defn}
%-------------------------------------------
\begin{ex}
The adjacency matrix of the overlap graph of $\pi_3 = [4,5,2,6,1,7,3,8]$ is\\~\\
\[
A_{\pi_3} = 
 \begin{pmatrix}
  0 & 1 & 1 & 1 & 0 & 0 & 1 & 0 & 0 \\
  1 & 0 & 1 & 0 & 0 & 1 & 1 & 0 & 0 \\
  1 & 1 & 0 & 0 & 0 & 1 & 0 & 1 & 0 \\
  1 & 0 & 0 & 0 & 0 & 0 & 0 & 1 & 0 \\
  0 & 0 & 0 & 0 & 0 & 0 & 0 & 0 & 0 \\
  0 & 1 & 1 & 0 & 0 & 0 & 0 & 0 & 0 \\
  1 & 1 & 0 & 0 & 0 & 0 & 0 & 0 & 0 \\
  0 & 0 & 1 & 1 & 0 & 0 & 0 & 0 & 0 \\
  0 & 0 & 0 & 0 & 0 & 0 & 0 & 0 & 0 \\
 \end{pmatrix}.
\]
\end{ex}
%-------------------------------------------
\begin{defn}
For a permutation $\pi = [a_1, ..., a_n]$, let $f_\pi$ be the function defined by $f_\pi(0) = 0$, $f_\pi(a_i) = i$ for $1 \le i \le n$, and $f_\pi(n+1) = n+1$.
Then, the \textit{precedence matrix} of $\pi$ is the $(n+2) \times (n+2)$ matrix $P_{\pi}$ such that, for $1 \le i,j \le (n+2)$, 
$$
P_{\pi}(i,j) = \left\{
	\begin{array}{ll}
           1 & f_\pi(i-1) < f_\pi(j-1) \\
            0 & \quad \text{otherwise}
	\end{array}
\right.
$$
\end{defn}
%-------------------------------------------
\begin{ex}
The precedence matrix of the permutation $\pi_3 = [4,5,2,6,1,7,3,8]$ is\\~\\
\[
P_{\pi_3} = 
 \begin{pmatrix}
0 & 1 & 1 & 1 & 1 & 1 & 1 & 1 & 1 & 1 \\
0 & 0 & 0 & 1 & 0 & 0 & 0 & 1 & 1 & 1 \\
0 & 1 & 0 & 1 & 0 & 0 & 1 & 1 & 1 & 1 \\
0 & 0 & 0 & 0 & 0 & 0 & 0 & 0 & 1 & 1 \\
0 & 1 & 1 & 1 & 0 & 1 & 1 & 1 & 1 & 1 \\
0 & 1 & 1 & 1 & 0 & 0 & 1 & 1 & 1 & 1 \\
0 & 1 & 0 & 1 & 0 & 0 & 0 & 1 & 1 & 1 \\
0 & 0 & 0 & 1 & 0 & 0 & 0 & 0 & 1 & 1 \\
0 & 0 & 0 & 0 & 0 & 0 & 0 & 0 & 0 & 1 \\
0 & 0 & 0 & 0 & 0 & 0 & 0 & 0 & 0 & 0 \\
\end{pmatrix}.
\]
\end{ex}
%-------------------------------------------
%-------------------------------------------
We now provide three definitions of matrices that are the result of eliminating the first and last rows and/or columns of an initial matrix.
%-------------------------------------------
\begin{defn}
For an $n \times n$ matrix $A$, we define the \textit{central submatrix of} $A$ \textit{without its rows}, denoted cent$_{r}(A)$, to be the $(n-2) \times n$ submatrix of $A$ that excludes the first and last rows of $A$. Similarly, the \textit{central submatrix of} $A$ \textit{without its columns}, denoted cent$_{c}(A)$ is the $n \times (n-2)$ submatrix of $A$ that excludes the first and last columns of $A$. Analogously, we can also define the \textit{central submatrix of} $A$ \textit{without its rows and columns}, cent$_{r,c}(A)$.
\begin{comment}
\textit{middle submatrix} of $A$, denoted $\text{cent}_c(A)$, to be the $n \times (n-2)$ submatrix of $A$ that excludes the first and last columns of $A$, and we define the \textit{central submatrix} of $A$, denoted $\fcent(A)$, to be the $(n-2)\times (n-2)$ submatrix of $A$ that excludes the first and last rows and columns of $A$. In particular, if $M = \text{cent}_c(A)$ and $C = \fcent(A)$, then we have $M(i,j) = A(i+1,j)$ for $1 \le i \le n-2$, $1 \le j \le n$, and $C(i,j) = A(i+1,j+1)$ for $1 \le i,j \le n-2$. Additionally, we define the \textit{horizontal middle submatrix} of $A$, denoted $\text{\normalfont hmid}(A)$, to be the $(n-2) \times n$ submatrix of $A$ that excludes the first and last rows of $A$; we then have $\text{\normalfont hmid}(A) = \text{cent}_c(A)^T$.
\end{comment}
\end{defn}
%-------------------------------------------
\begin{ex}
The three different central submatrix of the matrix 
\[
A = 
 \begin{pmatrix}
  0 & 1 & 1 & 1 & 0 & 0 & 1 & 0 & 0 \\
  1 & 0 & 1 & 0 & 0 & 1 & 1 & 0 & 0 \\
  1 & 1 & 0 & 0 & 0 & 1 & 0 & 1 & 0 \\
  1 & 0 & 0 & 0 & 0 & 0 & 0 & 1 & 0 \\
  0 & 0 & 0 & 0 & 0 & 0 & 0 & 0 & 0 \\
  0 & 1 & 1 & 0 & 0 & 0 & 0 & 0 & 0 \\
  1 & 1 & 0 & 0 & 0 & 0 & 0 & 0 & 0 \\
  0 & 0 & 1 & 1 & 0 & 0 & 0 & 0 & 0 \\
  0 & 0 & 0 & 0 & 0 & 0 & 0 & 0 & 0 \\
 \end{pmatrix}.
\]
are shown below:
{\tiny
\[ \text{cent}_{r}(A) = 
	\begin{pmatrix}
		1 & 0 & 1 & 0 & 0 & 1 & 1 & 0 & 0 \\
		1 & 1 & 0 & 0 & 0 & 1 & 0 & 1 & 0 \\ 
		1 & 0 & 0 & 0 & 0 & 0 & 0 & 1 & 0 \\
		0 & 0 & 0 & 0 & 0 & 0 & 0 & 0 & 0 \\
		0 & 1 & 1 & 0 & 0 & 0 & 0 & 0 & 0 \\
		1 & 1 & 0 & 0 & 0 & 0 & 0 & 0 & 0 \\
		0 & 0 & 1 & 1 & 0 & 0 & 0 & 0 & 0 \\
	\end{pmatrix} 
	\text{, cent}_{c}(A) = 
	\begin{pmatrix}
		1 & 1 & 1 & 0 & 0 & 1 & 0 \\
		0 & 1 & 0 & 0 & 1 & 1 & 0 \\
		1 & 0 & 0 & 0 & 1 & 0 & 1 \\
		0 & 0 & 0 & 0 & 0 & 0 & 1 \\
		0 & 0 & 0 & 0 & 0 & 0 & 0 \\
		1 & 1 & 0 & 0 & 0 & 0 & 0 \\
		1 & 0 & 0 & 0 & 0 & 0 & 0 \\
		0 & 1 & 1 & 0 & 0 & 0 & 0 \\
		0 & 0 & 0 & 0 & 0 & 0 & 0 \\
	\end{pmatrix} 
	\text{, and cent}_{r,c}(A) = 
	\begin{pmatrix}
	0 & 1 & 0 & 0 & 1 & 1 & 0 \\
	1 & 0 & 0 & 0 & 1 & 0 & 1 \\
	0 & 0 & 0 & 0 & 0 & 0 & 1 \\
	0 & 0 & 0 & 0 & 0 & 0 & 0 \\
	1 & 1 & 0 & 0 & 0 & 0 & 0 \\
	1 & 0 & 0 & 0 & 0 & 0 & 0 \\
	0 & 1 & 1 & 0 & 0 & 0 & 0 \\
	\end{pmatrix}
\]
}
\end{ex}
%-------------------------------------------

\section{Generalizing context directed swaps to finite graphs and matrices.}

The idea of generalizing the \textbf{cds} operation from permutations to graphs was given in \textcolor{blue}{\cite{p4}}, Section 4. As noted in the introduction, this generalization was also observed independently in Section 7 of \textcolor{blue}{\cite{BHH}}. We will define a slightly modified version of the generalization of \textbf{cds} to graphs given in \textcolor{blue}{\cite{p4}}, and have \textbf{gcds} denote this operation on graphs. This modified \textbf{gcds} definition is as follows:
%-------------------------------------------
\begin{defn}\label{defn:gcds}
Let $G = (V,E)$ be a two-rooted undirected graph on $n$ vertices, and let $p$ and $q$ be non-root vertices of $G$ such that edge $\{p,q\} \in E$. For any two vertices $x, y \in V$, let $f_x(y) = 1$ if $x$ and $y$ are adjacent, and let $f_x(y) = 0$ otherwise. Then $\textbf{gcds}_{\{p,q\}}(G)$ is the unique graph $G' = (V,E')$ such that for any $u, v \in V$, edge $\{u,v\} \in E'$ if, and only if, $$f_p(u)f_q(v)+ f_q(u)f_p(v) +f_u(v) = 1.$$
\end{defn}
%-------------------------------------------

This is an alternative definition of \textbf{gcds} introduced in \textcolor{blue}{\cite{p4}}. The proof that the two definitions are equivalent is given in Appendix A.
%-------------------------------------------
\begin{defn}\label{defn:mcds}
Given a matrix $M$, the matrix \textbf{cds} operation, \textbf{mcds}, on entries $p,q$ is: \[ \textbf{mcds}(M) = M + MI_{pq}M \] where $I_{pq}$ is the matrix over $\mathbb{F}_2$ with $I_{pq}(i,j) = 1$ if $i = p$ and $j = q$ or $i = q$ and $j = p$ and $I_{pq}(i,j) = 0$ otherwise.
\end{defn}
%-------------------------------------------
%-------------------------------------------
Using the definitions of \textbf{gcds} and \textbf{mcds}, we now define \emph{sortability}\footnote{The terminology ``sortability" is inspired by the original context that is being generalized, rather than an intuitive idea of sorting for a graph or a matrix.} for graphs and matrices under the respective operations. Noting that the overlap graph of the sorted permutation $[1,2,\dots,n]$ is the discrete graph on $(n+1)$ vertices containing no edges and that the adjacency matrix of this graph is the $(n+1) \times (n+1)$ zero matrix, the following notations are defined:
%-------------------------------------------
\begin{defn}\label{defn:gcds_sort}
A two-rooted graph is sorted if it contains no edges, and \textbf{gcds}-sortable if finitely many applications of the \textbf{gcds} operation on the graph result in a graph containing no edges.
\end{defn}
%-------------------------------------------
\begin{defn}\label{defn:mcds_sort}
We call a matrix sorted if it is the zero matrix, and \textbf{mcds}-sortable if finitely many applications of the \textbf{mcds} operation on the matrix will result in the zero matrix.
\end{defn}
%-------------------------------------------
%-------------------------------------------
First, we justify our selection of terminology by showing that \textbf{gcds} applied to overlap graphs of permutations is equivalent to \textbf{cds} applied to permutations:
%-------------------------------------------
\begin{lemma}\label{lem:cds_eq_gcds}
Let $\varphi: S_n \to G_{n+1,2}$ (where $G_{n+1,2}$ is the set of two-rooted graphs on $n+1$ vertices), such that if $\pi \in S_n$ then $\varphi(\pi)$ is the overlap graph of $\pi$.

Let $\pi$ be a permutation of $[1,2,\dots,n]$, and let $\varphi(\pi) = (V,E)$, where $V = \{v_0,v_1,\dots,v_n\}$ and each vertex $v_i \in V$ corresponds to the pointer $(i,i+1)$ in $\pi$. Then $\textup{\textbf{gcds}}_{\{v_p,v_q\}}(\varphi(\pi))$ is defined if, and only if, $\textup{\textbf{cds}}_{\{p,q\}}(\pi)$ is defined, in which case we have $\textup{\textbf{gcds}}_{\{v_p,v_q\}}(\varphi(\pi)) = \varphi(\textbf{cds}_{\{p,q\}}(\pi))$.
\end{lemma}
%-------------------------------------------
\begin{proof}

By definition of \textbf{gcds}, there is an edge between $u$ and $v$ in $G' = \textbf{gcds}(G)$ if, and only if, $f_p(u)f_q(v)+f_q(u)f_p(v)+f_u(v) \equiv 1 \pmod 2$. This means that if $f_p(u)f_q(v)+f_q(u)f_p(v) \equiv 0 \pmod 2$, then there is an edge between $u$ and $v$ in $G'$ if, and only if, this edge is in $G$, so the edge "stays". Otherwise, if $f_p(u)f_q(v)+f_q(u)f_p(v) \equiv 1\pmod 2$, then there is an edge between $u$ and $v$ in $G'$ if, and only if, there is no edge between $u$ and $v$ in $G$ (e.g. the edge "switches" between $G$ and $G'$). We show that two pointers $u, v$ remain intersected or non-intersected if, and only if, $f_p(u)f_q(v)+f_q(u)f_p(v) = 0$ for their corresponding vertices in $G$.

Let us consider the permutation $\pi$ with intersecting pointers $p$ and $q$ as follows:
\begin{center}
\begin{tikzpicture}[scale=.5]
	\draw[black] (4,0) circle (2pt) node(p1)[anchor = north]{$p$};
	\draw[black] (8,0) circle (2pt) node(q1)[anchor = north]{$q$};
	\draw[black] (12,0) circle (2pt) node(p2)[anchor = north]{$p$};
	\draw[black] (16,0) circle (2pt) node(q2)[anchor = north]{$q$};
	\draw[black, out=60, in=120] (p1.north) to (p2.north);
	\draw[black, out=60, in=120] (q1.north) to (q2.north);
    
    \draw[decoration={brace,mirror,raise=5pt},decorate] (0,0) -- node[below=6pt] {$\alpha$} (p1.north west); 
    \draw[decoration={brace,mirror,raise=5pt},decorate] (p1.north east) -- node[below=6pt] {$\beta$} (q1.north west);  
    \draw[decoration={brace,mirror,raise=5pt},decorate] (q1.north east) -- node[below=6pt] {$\gamma$} (p2.north west);  
    \draw[decoration={brace,mirror,raise=5pt},decorate] (p2.north east) -- node[below=6pt] {$\delta$} (q2.north west);  
    \draw[decoration={brace,mirror,raise=5pt},decorate] (q2.north east) -- node[below=6pt] {$\alpha$} (20,0); 
\end{tikzpicture}\\
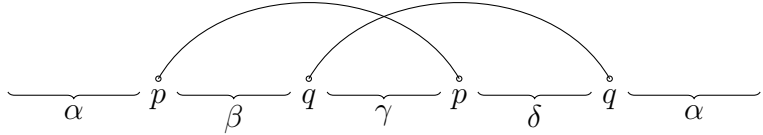
FIGURE 5. Blocks Defined by a Context
\end{center}

We can label the blocks between the pointers in order from left to right as $\alpha,\beta,\gamma,\delta,\alpha$ (we consider the first and last blocks to be part of the same block). Note that if $f_p(x) = 1, f_q(x) = 0$, then $x$ can have its two endpoints in either ($\alpha$,$\beta$) or ($\gamma$, $\delta$). Similarly, if $f_p(x) = 0$ and $f_q(x) = 1$, then $x$ can have its two endpoints in either ($\alpha$, $\delta$) or ($\beta$, $\gamma$). If $f_p(x) = f_q(x) = 1$, then $x$ can have its two endpoints in either ($\alpha$, $\gamma$) or ($\beta$, $\delta$). Finally, if $f_p(x) = f_q(x) = 0$, then both endpoints of $x$ must be in the same block (($\alpha,\alpha$),($\beta,\beta$),($\gamma,\gamma$), or ($\delta,\delta$)). When \textbf{cds} is applied, the $\beta$ and $\delta$ blocks are swapped, so we will call these two blocks the swapped blocks, and the other two blocks the unchanged blocks. 

Now, given pointers $u$ and $v$, assume without loss of generality that $u$ appears before $v$ in $\pi$. If $u$ and $v$ are non-intersecting (there is no edge between $u$ and $v$ in $G$), then the possible orientations are either $u...u...v...v$ or $u...v...v...u$. If they are intersecting (there is an edge between $u$ and $v$ in $G$), their orientation must be $u...v...u...v$. Note that the relative ordering of the pointers inside a block are preserved during \textbf{cds}, so if both ends of a pointer are in the same block, then no edge adjacent to the corresponding vertex will switch with \textbf{cds}, because no pointer can be moved in between or out from between the two ends. We consider seven cases:

\begin{enumerate}
	\item $f_p(u) = 1, f_q(u) = 0, f_p(v) = 1, f_q(v) = 1, f_p(u)f_q(v)+f_q(u)f_p(v) = 1$:  In this case the endpoints of $u$ will be in either $(\alpha, \beta)$ or $(\gamma, \delta)$ and the endpoints of $v$ will be in one of $(\alpha,\gamma)$ and $(\beta, \delta)$. So $u$ will have an endpoint in a swapped block and one in an unchanged block, and $v$ will either have both ends in the two swapped blocks or both endpoints in the two unchanged blocks. When \textbf{cds} is applied, $v$ will still have both endpoints in the two swapped blocks and $u$ will still have one endpoint in the same unchanged block. However, the other end pointer of $u$ will be either in between the end pointers of $v$ after the swap or between the pointers after the swap. Thus, the edge between the corresponding vertices switches.\\
    
	\item $f_p(u) = 0, f_q(u) = 1, f_p(v) = 1, f_q(v) = 1, f_p(u)f_q(v)+f_q(u)f_p(v) = 1$: In this case the endpoints of $u$ will be in either $(\beta, \gamma)$ or $(\alpha, \delta)$ and the endpoints of $v$ will be in either $(\alpha, \delta)$ or $(\beta, \gamma)$. So again one pointer will have an end in a swapped block and one in an unchanged block, and the other will either have both ends in the two swapped blocks or both ends in the two unchanged blocks, and by the reasoning above, the edge must switch.\\
    
	\item $f_p(u) = 1, f_q(u) = 0, f_p(v) = 0, f_q(v) = 1, f_p(u)f_q(v)+f_q(u)f_p(v) = 1$: In this case the endpoints of $u$ will be in either $(\alpha, \beta)$ or $(\gamma, \delta)$ and the endpoints of $v$ will be in either $(\alpha, \delta)$ or $(\beta, \gamma)$. So each pointer will have one endpoint in a swapped block, and one endpoint in an unchanged block. If the two endpoints in unchanged blocks are in the same unchanged block, their other endpoints must be in different swapped blocks, and thus when \textbf{cds} is applied the relative ordering of one of those two endpoints will change, switching the edge between two pointers. Otherwise, if the two endpoints are in the same swapped block, their other endpoints must be in different unchanged blocks on either side of the swapped block, so when \textbf{cds} is applied, the swapped block will be transposed with one of the endpoints, thus moving that endpoint either in between the endpoints of the other pointer or out from in between the pointers after the swap, causing the edge between the corresponding vertices to switch.\\
    
	\item $f_p(u) = 1, f_q(u) = 0, f_p(v) = 1, f_q(v) = 0, f_p(u)f_q(v)+f_q(u)f_p(v) = 0$: If two pointers have their endpoints in the same blocks, then applying \textbf{cds} will not switch the edge between them, because no point can be moved in between or out from between the endpoints of the other pointer. Otherwise, in this case we have the endpoints of one pointer in $(\alpha, \beta)$ and the endpoints of the other in $(\gamma, \delta)$, and after \textbf{cds} is applied, we have the endpoints of one pointer in $(\alpha, \delta)$ and the endpoints of the other in $(\beta, \gamma)$. Thus, the pointers must be non-intersecting both before and after \textbf{cds} is applied.\\
    
	\item $f_p(u) = 0, f_q(u) = 1, f_p(v) = 0, f_q(v) = 1, f_p(u)f_q(v)+f_q(u)f_p(v) = 0$: Again, because two pointers with their endpoints in the same blocks will not change their intersection when \textbf{cds} is applied, we consider the case in which we have the endpoints of one pointer in $(\beta, \gamma)$ and the endpoints of the other in $(\alpha, \delta)$, and after \textbf{cds} is applied, we must have the endpoints of one pointer in $(\alpha, \beta)$ and the endpoints of the other in $(\gamma, \delta)$. Thus, the pointers must be non-intersecting both before and after \textbf{cds} is applied.\\
    
	\item $f_p(u) = 1, f_q(u) = 1, f_p(v) = 1, f_q(v) = 1, f_p(u)f_q(v)+f_q(u)f_p(v) = 0$: Again, because two pointers with their endpoints in the same blocks will not change their intersection when \textbf{cds} is applied, we consider the case in which we have the endpoints of one pointer in $(\beta, \delta)$ and the endpoints of the other in $(\alpha, \gamma)$. This is the case both before and after \textbf{cds} application, so the pointers will be intersecting both times, and the edge will not change.\\
    
	\item $f_p(u) = f_q(u) = 0$ or $f_p(v) = f_q(v) = 0, f_p(u)f_q(v)+f_q(u)f_p(v) = 0$: In this case, either $u$ or $v$ must be contained in the same block, and since the order of points in blocks are unchanged, the edge between $u$ and $v$ will not switch.
\end{enumerate}
\end{proof}

%-------------------------------------------
The following lemmas show that the \textbf{gcds} and \textbf{mcds} operations are equivalent, confirming that our choice of \textbf{mcds} as generalization of \textbf{cds} is correct. 
%-------------------------------------------
\begin{lemma}\label{lem:gcds_eq_mcds}
Let $\varphi$ be the function from the set of undirected $n$-vertex graphs to the set of $n \times n$ matrices over $\mathbb{F}_2$ such that if $G$ is a graph, then $\varphi(G)$ is the adjacency matrix of $G$.

Let $G = (V, E)$ be a graph on $n$ vertices, and let $V = \{ v_1, v_2, \dots, v_n \}$. Then \textup{\textbf{mcds}} can be applied to rows $p$ and $q$ of $\varphi(G)$ if, and only if, \textup{\textbf{gcds}} can be applied to vertices $v_p$ and $v_q$ of $G$, in which case we have $\textup{\textbf{mcds}}_{\{p,q\}}(\varphi(G)) = \varphi(\textup{\textbf{gcds}}_{\{v_p,v_q\}}(G))$.
\end{lemma}
%-------------------------------------------
\begin{proof}
By definition, the \textbf{gcds} move that includes the vertices $v_p$ and $v_q$ is valid if, and only if, $v_p$ and $v_q$ are adjacent in $G$. The \textbf{mcds} move on rows $p$ and $q$ in $\varphi(G)$ is valid if, and only if, the element in row $p$, column $q$ of $\varphi(G)$ is a $1$. Note that by definition of the adjacency matrix, these conditions are equivalent, so $(p,q)$ is a valid \textbf{mcds} move if, and only if, $(v_p, v_q)$ is a valid \textbf{gcds} move, as desired.

Next, let $v_p$ and $v_q$ be adjacent vertices in $G$; we claim that $\textbf{mcds}_{\{p,q\}}(\varphi(G)) = \varphi(\textbf{gcds}_{\{v_p,v_q\}}(G))$. It suffices to show that for any $i$ and $j$ with $1 \le i,j \le n$, the element at position $(i,j)$ in $\textbf{mcds}_{\{p,q\}}(\varphi(G))$ is equal to the element at position $(i,j)$ in $\varphi(\textbf{gcds}_{\{v_p,v_q\}}(G))$. 

Let $M = \varphi(G)$. We defined $\textbf{mcds}_{\{p, q\}}(\varphi(G)) = M+MI_{pq}M$, where $I_{pq}$ is the $n \times n$ matrix with ones in positions $(p,q)$ and $(q,p)$ and the remaining entries all equal to zero. Then, we have

\begin{align*}
(M+MI_{pq}M)(i,j) &= M(i,j) + ((MI_{pq})M)(i,j) \\
&= M(i,j) + \sum_{k=1}^n (MI_{pq})(i,k)M(k,j) \\
&= M(i,j) + \sum_{k=1}^n \sum_{t=1}^n M(i,t)I_{pq}(t,k) M(k,j).
\end{align*}

Note that by definition of the matrix $I_{pq}$, we have that $I_{pq}(t,k) = 1$ if and only $t = p$ and $k = q$ or $t = q$ and $k = p$. Thus, it follows that 

\begin{align*}
(M+MI_{pq}M)(i,j) &= M(i,j) + \sum_{k=1}^n \sum_{t=1}^n M(i,t)I_{pq}(t,k) M(k,j) \\
&= M(i,j) + M(i,p)M(q,j) + M(i,q)M(p,j).
\end{align*}

Next, to find $\varphi(\textbf{gcds}_{\{v_p,v_q\}}(G))(i,j)$ 
, we have that if $f_{v_p}(v_i)f_{v_q}(v_j)+f_{v_p}(v_j)f_{v_q}(v_i) = 1$ then edge $\{v_i,v_j\}$ is included in $\textbf{gcds}_{\{v_p,v_q\}}(G)$ if, and only if, it is \textit{not} included in $G$, and otherwise edge $\{v_i,v_j\}$ is included in $\textbf{gcds}_{\{v_p,v_q\}}(G)$ if, and only if, it is included in $G$, where for vertices $v_x$ and $v_y$ $f_{v_x}(v_y)$ is defined to be $1$ if $v_x$ and $v_y$ are adjacent in $G$ and otherwise $0$. %But note 
By definition of the adjacency matrix, we have that $f_{v_x}(v_y) = M(x,y)$ for all pairs $(x,y)$ with $1 \le x,y \le n$. Therefore, we have that $f_{v_p}(v_i)f_{v_q}(v_j)+f_{v_p}(v_j)f_{v_q}(v_i) = M(p,i)M(q,j)+M(p,j)M(q,i)$. Furthermore, edge $\{v_i,v_j\}$ is included in $G$ if, and only if, $M(i,j) = 1$. 

We also have that when $M(p,i)M(q,j)+M(p,j)M(q,i) = 1$, edge $(v_i,v_j)$ will appear in $\textbf{gcds}_{\{v_p,v_q\}}(G)$ if, and only if, $M(i,j) = 0$. Additionally, when $M(p,i)M(q,j)+M(p,j)M(q,i) \not= 1$ implies edge $\{v_i,v_j\}$ is in $\textbf{gcds}_{\{v_p,v_q\}}(G)$ if, and only if, $M(i,j) = 1$. Since $0 \le M(p,i)M(q,j)+M(p,j)M(q,i) \le 2$, if $M(p,i)M(q,j)+M(p,j)M(q,i) \not= 1$ then $M(p,i)M(q,j)+M(p,j)M(q,i) \equiv 0 \pmod 2$. Therefore, in either case, edge $\{v_i,v_j\}$ is in $\textbf{gcds}_{\{v_p,v_q\}}(G)$ if, and only if, $M(i,j) + M(p,i)M(q,j)+M(p,j)M(q,i) \equiv 1 \pmod 2$. This implies that $\varphi(\textbf{gcds}_{\{v_p,v_q\}}(G))(i,j) = 1$ if $M(i,j) + M(p,i)M(q,j)+M(p,j)M(q,i) \equiv 1 \pmod 2$ and otherwise $\varphi(\textbf{gcds}_{\{v_p,v_q\}}(G))(i,j) = 0$. Thus,

\[\varphi(\textbf{gcds}_{\{v_p,v_q\}}(G))(i,j) \equiv M(i,j) + M(p,i)M(q,j)+M(p,j)M(q,i) \pmod 2.\]

Using simple arithmetic, we have $\varphi(\textbf{gcds}_{\{v_p,v_q\}}(G))(i,j) = M(i,j) + M(p,i)M(q,j)+M(p,j)M(q,i)$. Since the graph $G$ is undirected, the matrix $M = \varphi(G)$ is symmetric, so $M(p,i) = M(i,p)$ and $M(q,i) = M(i,q)$ and thus

\begin{align*}
\varphi(\textbf{gcds}_{\{v_p,v_q\}}(G))(i,j) &= M(i,j) + M(p,i)M(q,j)+M(p,j)M(q,i) \\&= M(i,j) + M(i,p)M(q,j) + M(i,q)M(p,j) \\&= (M+MI_{pq}M)(i,j) \\&= \textbf{mcds}_{\{p,q\}}(M)(i,j),
\end{align*} 

with addition and multiplication in $\mathbb{F}_2$. Therefore, we have

$$\textbf{mcds}_{\{p,q\}}(\varphi(G)) = \varphi(\textbf{gcds}_{\{v_p,v_q\}}(G)),$$ 

as desired.
\end{proof}
%-------------------------------------------

\begin{cor}
Define the function $\varphi : S_n \to M_n(\mathbb{F}_2)$ so that for a permutation $\pi$, $\varphi(\pi)$ is the adjacency matrix of $\pi$. Let $\pi \in S_n$. Then, \textbf{mcds} can be applied to rows $p$ and $q$ of $\varphi(\pi)$ if, and only if, \textbf{cds} can be applied to the pointers $(p,p+1)$ and $(q,q+1)$ of $\pi$, in which case we have \[ \textbf{mcds}_{\{p,q\}}(\varphi(\pi)) = \varphi(\textbf{cds}_{\{p,q\}}(\pi)). \]
\end{cor}

\begin{proof}
By Lemma \ref{lem:cds_eq_gcds}, we have that $\textbf{cds}_{\{p,q\}}(\pi)$ is defined if, and only if, $\textbf{gcds}_{\{v_p,v_q\}}($OG$(\pi))$ is defined, and by Lemma \ref{lem:gcds_eq_mcds} this \textbf{gcds} operation is defined if, and only if, $\textbf{mcds}_{\{p,q\}}(\varphi(\pi))$ is defined, so each \textbf{mcds} operation is defined if, and only if, the corresponding \textbf{cds} operation is defined.

Furthermore, if we let $\varphi_1$ be the function defined in Lemma \ref{lem:cds_eq_gcds} taking a permutation to its overlap graph and let $\varphi_2$ be the function defined in Lemma \ref{lem:gcds_eq_mcds} taking a graph to its adjacency matrix, then we have $\varphi = \varphi_2 \circ \varphi_1$. Applying Lemmas \ref{lem:cds_eq_gcds} and \ref{lem:gcds_eq_mcds}, we have 
\begin{align*}
\textbf{mcds}_{\{p,q\}}(\varphi(\pi)) &= \textbf{mcds}_{\{p,q\}}(\varphi_2(\varphi_1(\pi))) \\
&= \varphi_2(\textbf{gcds}_{\{v_p,v_q\}}(\varphi_1(\pi))) \\
&= \varphi_2(\varphi_1(\textbf{cds}_{\{p,q\}}(\pi))) \\
&= \varphi(\textbf{cds}_{\{p,q\}}(\pi)),
\end{align*}
as desired.
\end{proof}
%-------------------------------------------
%-------------------------------------------
\section{Parity Cuts and Preservation Properties}\label{sec3} 
%-------------------------------------------
\begin{defn}\cite{p2}%[\textcolor{blue}{Li, 2015}]
\label{defn:pc}
Let $G = (V,E)$ be a two-rooted graph. A parity cut of $G$ is a partition $V = V_1 \cup V_2, V_1 \cap V_2 = \emptyset$ %into 2 sets, $V_1$ and $V_2$, - editor 13 suggested math notation
such that for any non-root vertex $v \in V_i, \delta_{V_j}(v), j \neq i$, is even. 
\end{defn}
%-------------------------------------------
For the purpose of this paper, we will only consider parity cuts to mean those cuts in which the roots also have an even number of edges to vertices of the opposite set.
%-------------------------------------------
\begin{lem}\label{lem:pres_eulerian}
Let $G = (V,E)$ be an Eulerian graph with $G' = (V,E') = \textbf{gcds}_{\{u,v\}}(G)$ for some vertices $u,v \in V$. Then $G'$ is also an Eulerian graph.
\end{lem}
%-------------------------------------------
\begin{proof}
Let $p, q \in V$ such that $\{p,q\} \in E$. Consider $G' = \textbf{gcds}_{\{p,q\}}(G)$. Let $u \in V$ such that $u \neq p$ and $u \neq q$. There are three cases:

\begin{enumerate}
    \item The vertex $u$ is adjacent to neither $p$ nor $q$. In this case, $N_{G}(u) = N_{G'}(u)$.
    
    \item The vertex $u$ is adjacent to either $p$ or $q$ but not both. Without loss of generality, assume that $u$ is adjacent to $p$. Then, let $S = N_G(q) \setminus \{p\}$. For each $s \in S$, $\{u,s\} \in E'$ if, and only if, $\{u,s\} \not\in E$. Let $x$ be the number of adjacencies in $S$ gained and $y$ be the number of adjacencies in $S$ lost. Since $x+y = |S|$,(and $|S|$ must be odd) exactly one of $x$ or $y$ is odd. Therefore $x-y$, the net change of adjacencies of $u$ in $S$, is odd. However, the edge $\{u,p\}$ does not exist in $G'$. Therefore, the total net change in adjacencies of $u$ is even. Because $u$ had an even number of adjacencies in $G$, $u$ must have an even number of adjacencies in $G'$.
    
    \item The vertex $u$ is adjacent to both $p$ and $q$. Let $a_1, ..., a_i$ be the vertices adjacent to $p$, but not to $q$. Let $b_1, ..., b_j$ be the vertices adjacent to $q$, but not to $p$, and $c_1, ..., c_k$ be the vertices adjacent to both $p$ and $q$ other than $u$. For all $r, 1 \leq r \leq i$, the edge $\{a_r,u\}$ will exist in $G'$ if, and only if, it does not exist in $G$. Likewise, for all $r, 1 \leq r \leq j$, $\{b_r,u\}$ will exist in $G'$ if, and only if, it did not exist in $G$. However, for all $r, 1 \leq r \leq k$, $\{c_r,u\}$ will exist in $G'$ if, and only if, it did exist in $G$. Therefore, the total number of edges switched (where switched means the edge exists in exactly one of $G$ and $G'$) is equal to $i+j$. Since the vertex $p$ has even degree and $a_1, ..., a_i, c_1, ..., c_k$ does not include $u$ or $q$, $i+k$ must be even. Similarly, since the vertex $q$ has even degree, so $j+k$ must be even. Therefore, the net change in adjacencies of $u$ must be even, so $u$ must be of even degree in $G'$.  
\end{enumerate}

By combining the above cases, we have shown that in a graph with only even degree vertices, a $\textbf{gcds}$ operation will result in a graph with only even degree vertices.
\end{proof}
%-------------------------------------------
Another property that is preserved under \textbf{gcds} on an Eulerian graph is presented in the following theorem.
%-------------------------------------------
\begin{thm}\label{thm:pres_pc_eulerian}
Let $G = (V,E)$ be an Eulerian graph with $G' = (V,E') = \textbf{gcds}_{\{u,v\}}(G)$ for some vertices $u,v \in V$. Let $p,q$ be vertices and $(V_1, V_2)$ be a partition of $V$, then $(V_1, V_2)$ is a parity cut of $G$ if, and only if, there exists some parity cut $(V_1', V_2')$ of $G'$ such that for all vertices $v \neq p,q \in V$, $v \in V_1 \Leftrightarrow v \in V_1'$.
\end{thm}
%-------------------------------------------
\begin{proof}
($\Leftarrow$)  % change to $\Rightarrow$?
Suppose that $(V_1,V_2)$ is a valid parity cut of $G$. We will show that for any vertex $u \in V_i$, when \textbf{gcds} is performed on two vertices $p,q$, the parity of $\delta_{V_j}(v), j \neq i$ does not change, and that therefore this is also a valid parity cut for $G'$. We note that because $G$ is Eulerian, for any $v \in V$, $\delta_{V_1}(v)$ and $\delta_{V_2}(v)$ will both be even. Consider the sets $A = \{v \in V_j| v \neq q, \text{ v is adjacent only to $p$}\}, B = \{v \in V_j| v \neq p, \text{ v is adjacent only to $q$}\}$, and $C = \{v \in V_j| \text{ v is adjacent to both $p$ and $q$}\}$. Suppose $u$ is adjacent to $x$ vertices in $A$, $y$ vertices in $B$, and $z$ vertices in $C$. There are 3 possible cases:
\begin{enumerate}
	\item Suppose $p, q \in V_i$. In this case, $\delta_{V_j}(p) = |A \bigcup C| = |A|+|C|$, which is even, and similarly, $|B|+|C|$ is even, so $|A|+|B|$ must also be even.
    \begin{enumerate}
	\item Suppose $u$ is adjacent to both $p$ and $q$. When the \textbf{gcds} operation is performed, edges between $u$ and vertices of $A$ and $B$ will switch (for any vertex $v \in A \cup B$, the edge between $u$ and $v$ is in $G'$ if and only if it is not in $G$), which means $u' \in V(G')$ will have $\delta_{V_j}(u') = \delta_{V_j}(u)+(|A|-x)-x+(|B|-y)-y = \delta_{V_j}(u)+|A|+|B|-2x-2y$ edges, so the change in $\delta_{V_j}(u)$ is even. \\  
	\item Suppose $u$ is adjacent only to $p$. When the \textbf{gcds} operation is performed, edges between $u$ and vertices of $B$ and $C$ will switch, which means $u' \in V(G')$ will have $\delta_{V_j}(u') = \delta_{V_j}(u)+(|B|-y)-y+(|C|-z)-z = \delta_{V_j}(u)+|B|+|C|-2y-2z$ edges, so the change in $\delta_{V_j}(u)$ is even. \\    
    \begin{center}\begin{tikzpicture} [scale = .7]
    \filldraw[black] (0, 4) circle (2pt) node(q)[anchor=south]{q};
    \filldraw[black] (-1, 2) circle (2pt) node(u)[anchor=south]{u};
    \filldraw[black] (1, 2) circle (2pt) node(p)[anchor=south]{p}; 
    \draw[black] (p) -- (q);
    \draw[black] (u) -- (p);
    \draw[black] (u) -- (q);
    \draw[black, dashed] (0, 3) circle (48pt);
    \draw (0,5) node(v1)[anchor=south]{$V_1$};
    \draw[black, dashed] (4, 3) circle (48pt);
    \draw (4,5) node(v2)[anchor=south]{$V_2$};
    \draw (2,1) node()[anchor=center]{(a)};
    \end{tikzpicture}
    \hspace{1cm}
    \begin{tikzpicture} [scale = .7]
    \filldraw[black] (0, 4) circle (2pt) node(q)[anchor=south]{q};
    \filldraw[black] (-1, 2) circle (2pt) node(u)[anchor=south]{u};
    \filldraw[black] (1, 2) circle (2pt) node(p)[anchor=south]{p}; 
    \draw[black] (p) -- (q);
    \draw[black] (u) -- (p);
    \draw[black, dashed] (0, 3) circle (48pt);
    \draw (0,5) node(v1)[anchor=south]{$V_1$};
    \draw[black, dashed] (4, 3) circle (48pt);
    \draw (4,5) node(v2)[anchor=south]{$V_2$};
    \draw (2,1) node()[anchor=center]{(b)};
    \end{tikzpicture}\\
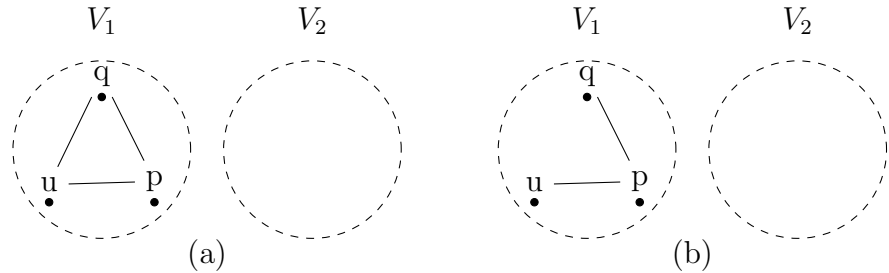
	FIGURE 6. Case $1$: $p,q \in V_i$.\\~\\
    \end{center}
    \end{enumerate}
	\item Suppose that $p \in V_i, q \in V_j$. In this case, $\delta_{V_j}(p) = |A \cup C \cup \{q\}| = |A|+|C|+1$, which is even, so $|A|+|C|$ must be odd and similarly, $|B|+|C|$ must be even, so $|A|+|B|$ must be odd.
    \begin{enumerate}
	\item Suppose $u$ is adjacent to both $p$ and $q$. When the \textbf{gcds} operation is performed, edges between $u$ and vertices of $A$ and $B$ will switch and the edge between $u$ and $q$ will be removed, which means $u' \in V(G')$ will have $\delta_{V_j}(u') = \delta_{V_j}(u)+(|A|-x)-x+(|B|-y)-y-1 = \delta_{V_j}(u)+|A|+|B|-2x-2y-1$ edges, so the change in $\delta_{V_j}(u)$ is even. \\ 
	\item Suppose $u$ is adjacent only to $p$. When the \textbf{gcds} operation is performed, edges between $u$ and vertices of $B$ and $C$ will switch, which means $u' \in V(G')$ will have $\delta_{V_j}(u') = \delta_{V_j}(u)+(|B|-y)-y+(|C|-z)-z = \delta_{V_j}(u)+|B|+|C|-2y-2z$ edges, so the change in $\delta_{V_j}(u)$ is even. \\
	\item Suppose $u$ is adjacent only to $q$. When the \textbf{gcds} operation is performed, edges between $u$ and vertices of $A$ and $C$ will switch and the edge between $u$ and $q$ will be removed, which means $u' \in V(G')$ will have $\delta_{V_j}(u') = \delta_{V_j}(u)+(|A|-x)-x+(|C|-z)-z-1 = \delta_{V_j}(u)+|A|+|C|-2x-2z-1$ edges, so the change in $\delta_{V_j}(u)$ is even. \\ 
    \begin{center} 
    \begin{tikzpicture} [scale = .7]
    \filldraw[black] (0, 4) circle (2pt) node(p)[anchor=south]{p};
    \filldraw[black] (-1, 2) circle (2pt) node(u)[anchor=south]{u};
    \filldraw[black] (3.5, 2) circle (2pt) node(q)[anchor=south]{q}; 
    \draw[black] (p) -- (q);
    \draw[black] (u) -- (p);
    \draw[black] (u) -- (q);
    \draw[black, dashed] (0, 3) circle (48pt);
    \draw (0,5) node(v1)[anchor=south]{$V_1$};
    \draw[black, dashed] (4, 3) circle (48pt);
    \draw (4,5) node(v2)[anchor=south]{$V_2$};
    \draw (2,1) node()[anchor=center]{(a)};
    \end{tikzpicture}
    \hspace{.22cm}
    \begin{tikzpicture} [scale = .7]
    \filldraw[black] (0, 4) circle (2pt) node(p)[anchor=south]{p};
    \filldraw[black] (-1, 2) circle (2pt) node(u)[anchor=south]{u};
    \filldraw[black] (3.5, 2) circle (2pt) node(q)[anchor=south]{q}; 
    \draw[black] (p) -- (q);
    \draw[black] (u) -- (p);
    \draw[black, dashed] (0, 3) circle (48pt);
    \draw (0,5) node(v1)[anchor=south]{$V_1$};
    \draw[black, dashed] (4, 3) circle (48pt);
    \draw (4,5) node(v2)[anchor=south]{$V_2$};
    \draw (2,1) node()[anchor=center]{(b)};
    \end{tikzpicture}
    \hspace{.22cm}
    \begin{tikzpicture} [scale = .7]
    \filldraw[black] (0, 4) circle (2pt) node(p)[anchor=south]{p};
    \filldraw[black] (-1, 2) circle (2pt) node(u)[anchor=south]{u};
    \filldraw[black] (3.5, 2) circle (2pt) node(q)[anchor=south]{q}; 
    \draw[black] (p) -- (q);
    \draw[black] (u) -- (q);
    \draw[black, dashed] (0, 3) circle (48pt);
    \draw (0,5) node(v1)[anchor=south]{$V_1$};
    \draw[black, dashed] (4, 3) circle (48pt);
    \draw (4,5) node(v2)[anchor=south]{$V_2$};
    \draw (2,1) node()[anchor=center]{(c)};
    \end{tikzpicture}\\
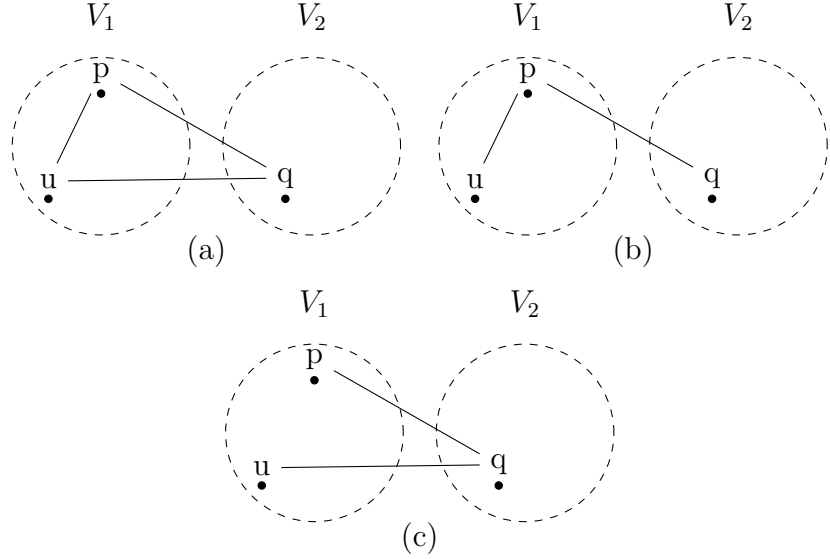
	FIGURE 7. Case $2$: $p \in V_i, q \in V_j$. \\~\\
    \end{center}
    \end{enumerate}
	\item Suppose $p, q \in V_j$. In this case, $\delta_{V_j}(p) = |A \bigcup C \bigcup \{q\}| = |A|+|C|+1$, which is even, so $|A|+|C|$ must be odd and similarly, $|B|+|C|$ must be odd, so $|A|+|B|$ must be even.
    \begin{enumerate}
	\item Suppose $u$ is adjacent to both $p$ and $q$. When the \textbf{gcds} operation is performed, edges between $u$ and vertices of $A$ and $B$ will switch and the edges between $u$ and both $p$ and $q$ will be removed, which means $u' \in V(G')$ will have $\delta_{V_j}(u') = \delta_{V_j}(u)+(|A|-x)-x+(|B|-y)-y-2 = \delta_{V_j}(u)+|A|+|B|-2x-2y-2$ edges, so the change in $\delta_{V_j}(u)$ is even. \\ 
	\item Suppose $u$ is adjacent only to $p$. When the \textbf{gcds} operation is performed, edges between $u$ and vertices of $B$ and $C$ will switch and the edges between $u$ and $p$ will be removed, which means $u' \in V(G')$ will have $\delta_{V_j}(u') = \delta_{V_j}(u)+(|B|-y)-y+(|C|-z)-z-1 = \delta_{V_j}(u)+|B|+|C|-2y-2z-1$ edges, so the change in $\delta_{V_j}(u)$ is even. \\    
    \begin{center}
    \begin{tikzpicture} [scale = .7]
    \filldraw[black] (0, 4) circle (2pt) node(u)[anchor=south]{u};
    \filldraw[black] (3.5, 3) circle (2pt) node(p)[anchor=south]{p};
    \filldraw[black] (4.5, 3) circle (2pt) node(q)[anchor=south]{q}; 
    \draw[black] (p) -- (q);
    \draw[black] (u) -- (p);
    \draw[black] (u) -- (q);
    \draw[black, dashed] (0, 3) circle (48pt);
    \draw (0,5) node(v1)[anchor=south]{$V_1$};
    \draw[black, dashed] (4, 3) circle (48pt);
    \draw (4,5) node(v2)[anchor=south]{$V_2$};
    \draw (2,1) node()[anchor=center]{(a)};
    \end{tikzpicture}
    \hspace{1cm}
    \begin{tikzpicture} [scale = .7]
    \filldraw[black] (0, 4) circle (2pt) node(u)[anchor=south]{u};
    \filldraw[black] (3.5, 3) circle (2pt) node(p)[anchor=south]{p};
    \filldraw[black] (4.5, 3) circle (2pt) node(q)[anchor=south]{q}; 
    \draw[black] (p) -- (q);
    \draw[black] (u) -- (p);
    \draw[black, dashed] (0, 3) circle (48pt);
    \draw (0,5) node(v1)[anchor=south]{$V_1$};
    \draw[black, dashed] (4, 3) circle (48pt);
    \draw (4,5) node(v2)[anchor=south]{$V_2$};
    \draw (2,1) node()[anchor=center]{(b)};
    \end{tikzpicture}\\
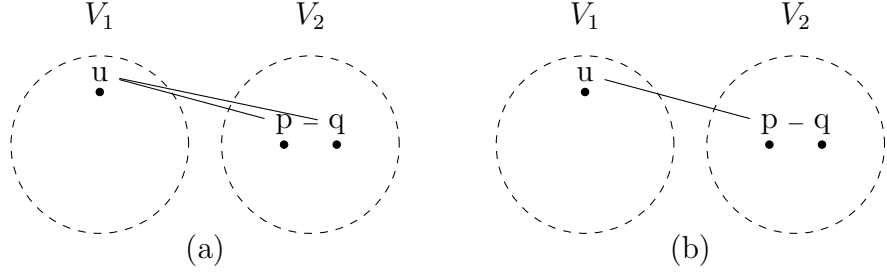
	FIGURE 8. Case $3$: $p,q \in V_j$. \\~\\
    \end{center}
    \end{enumerate}
\end{enumerate}
In all three cases, it is clear that $\delta_{V_j}(u')$ is even if, and only if, $\delta_{V_j}(u)$ is even. Additionally, $\delta_{V_j}(p') = \delta_{V_j}(q') = 0$, and we assumed that $\delta_{V_j}(p)$ and $\delta_{V_j}(q)$ were both even, so if $\delta_{V_1}(p)$, $\delta_{V_1}(q)$, $\delta_{V_2}(p)$, and $\delta_{V_2}(q)$ are all even, then for all $u \in V_i$, the parity of $\delta_{V_j}(u), j \neq i$ does not change. \\

($\Rightarrow$)
Conversely, suppose $G'$ has a valid parity cut with sets $(V_1', V_2')$. We will show that there is also a valid parity cut $(V_1, V_2)$ for $G$ such that $\forall v \neq p,q \in V, v \in V_1' \Leftrightarrow v\in V_1$. Since $p$ and $q$ are isolated vertices in $G'$, either $p,q \in V_1$ or $p,q \in V_2$. %they can be placed in either $V_1$ or $V_2$.
Suppose in $G$, $p$ had an even number of edges to vertices (other than $q$) in $V_i$ (this must be true for some $i = 1,2$ since $p$ had even degree and one edge to $q$). Since $p$ must also have had an edge to $q$ and it had even degree, it must then have had an odd number of edges to vertices (other than $q$) in $V_j, j \neq i$. Now, if $q$ also had an even number of edges to vertices in $V_i$, we can consider the partition of $G$ with all vertices other than $p$ and $q$ in the same sets that they are in in the valid parity cut of $G'$ and with $p$ and $q$ in $V_j$. Otherwise, if $q$ had an odd number of edges to vertices in $V_i$, we can consider the partition of $G$ with all vertices other than $p$ and $q$ in the same sets that they are in in the valid parity cut of $G'$ and with $p$ in $V_i$ and $q$ in $V_j$. In either case, $p$ and $q$ must have an even number of edges to vertices in the partition they are not it. Furthermore, as shown above, for any other vertex $u \in V$, after application of $\textbf{gcds}$ on $p$ and $q$, the parity of the number of edges between $u$ and vertices in the partition that does not contain $u$ does not change. Therefore, since this quantity is even in $G'$, it must have been even in $G$. Therefore, this cut must be a valid parity cut in $G$.
\end{proof}
%-------------------------------------------
To further analyze characteristics of Eulerian graphs, three properties on graphs will be presented below, them first introduced in \cite{p2}. %[Li, 2015].
%-------------------------------------------
\begin{defn}\cite{p2} %[Li, 2015]
Let $G = (V,E)$ be a two-rooted graph with roots $x,y \in V$. We define the following three properties:
	\begin{description}
    	\item[a)] $G$ has property $a$ if there exists some parity cut $(V_1, V_2)$ such that $x \in V_1$, $y \in V_2, \delta_{V_2}(x)$ is even, and $\delta_{V_1}(y)$ is even.
    	\item[b)] $G$ has property $b$ if there exists some parity cut $(V_1, V_2)$ such that $x \in V_1$, $y \in V_2, \delta_{V_2}(x)$ is odd, and $\delta_{V_1}(y)$ is odd.
    	\item[c)] $G$ has property $c$ if there exists some parity cut $(V_1, V_2)$ such that $x,y \in V_1$, and $\delta_{V_2}(x)$ and $\delta_{V_2}(y)$ are both odd.
    \end{description}
\end{defn}
%-------------------------------------------
Again, for the purposes of this paper, we will adapt this definition as follows:
	\begin{description}
    	\item[a)] $G$ has property $a$ if there exists some parity cut $(V_1, V_2)$ such that $x \in V_1$ and $y \in V_2$.
    \end{description}
Since we define parity cuts of two-rooted graphs to only be those cuts in which the roots have an even number of edges to vertices of the opposite set, this definition of property $a$ is equivalent.

Of the above three properties, we have that property $a$ is preserved when applying \textbf{gcds} to an Eulerian graph, and therefore we have a sortability criterion with this property.
%-------------------------------------------
\begin{thm}
Let $G = (V,E)$ be an Eulerian graph with roots $x,y \in V$ and let $G' = (V,E') = \textbf{gcds}_{\{u,v\}}(G)$ for some $u,v \in V$. In this case $G'$ has property $a$ if, and only if, $G$ has property $a$. 
\end{thm}
%-------------------------------------------
\begin{proof}
$\Leftarrow$ Suppose $G$ has property $a$. In this case, there exists some parity cut $(V_1, V_2)$ of $G$ with $x \in V_1$ and $y \in V_2$. Therefore, by Lemma \ref{thm:pres_pc_eulerian}, there exists some parity cut $(V_1', V_2')$ of $G'$ with $x \in V_1'$ and $y \in V_2'$.

$\Rightarrow$ Similarly, if $G'$ has property $a$, there exists some parity cut $(V_1', V_2')$ of $G'$ with $x \in V_1'$ and $y \in V_2'$, so by Lemma \ref{thm:pres_pc_eulerian}, there exists some parity cut $(V_1, V_2)$ of $G$ with $x \in V_1$ and $y \in V_2$.
\end{proof}
%-------------------------------------------
\begin{cor}\label{cor:gcdssort_propa}
An Eulerian graph is \textbf{gcds} sortable if, and only if, it has property $a$.
\end{cor}
%-------------------------------------------
\begin{proof}
In any fixed point graph, there must exist some non-isolated, non-root vertex $v$. If this graph is Eulerian, then it must be adjacent to both roots, and no other vertices. As a result, both roots must be in the same set of any parity cut of this graph. Thus, the only Eulerian graph which has property $a$ and on which it is not possible to perform the \textbf{gcds} operation is the discrete graph. Therefore, if an Eulerian graph has property $a$, after any number of applications of the \textbf{gcds} operation it will still be Eulerian and have property $a$, which means it must eventually become the discrete graph and thus be sortable, whereas if the graph does not have property $a$, after any number of applications of the \textbf{gcds} operation it will still be Eulerian and not have property $a$, which means it cannot ever become the discrete graph and must be unsortable.
\end{proof}
%-------------------------------------------
Besides obtaining theorems related to property $a$, we have the following lemma referring to property $b$.
%-------------------------------------------
\begin{lem}
The overlap graph of a valid unsigned permutation cannot have property $b$.
\end{lem}
%-------------------------------------------
\begin{proof}
Suppose an overlap graph $G = (V,E)$ has property $b$. First note that such a graph must be Eulerian. Therefore, given a parity cut $(V_1, V_2)$, for all non-root vertices $v \in V$, $\delta_{V_1}(v)$ and $\delta_{V_2}(v)$ must both be even, while for the roots $x,y \in V$, these degrees must be odd for sets. In this case, consider one partition of this graph when a valid parity cut with property $b$ is made. In the subgraph induced on $V_1$, every vertex except the one root in this set must have even degree, while the root has odd degree. However, it is not possible for only one vertex in a graph to have odd degree, so this is a contradiction.
\end{proof}
%-------------------------------------------
Finally, focusing on property $c$, we have the corollary below.
%-------------------------------------------
\begin{cor}
All Eulerian graphs must have exactly one of property $a$ or property $c$.
\end{cor}
%-------------------------------------------
\begin{proof}
We know from \cite{p2} %[\textcolor{blue}{Li, 2015}] 
that all graphs must have at least one of properties $a$, $b$, or $c$. Therefore, by the above lemma, all Eulerian graphs must have at least one of properties $a$ or $c$. Furthermore, Theorem 2.2 from \cite{p2} %[\textcolor{blue}{Li, 2015}] 
tells us that if a graph has two of the three properties, it must also have the third. Since it is not possible for an Eulerian graph to have property $b$, this means that no Eulerian graph can have both properties $a$ and $c$.
\end{proof}
%-------------------------------------------
Aside from the properties $a$, $b$, and $c$, one can look at the parity cut space in order to obtain more information on Eulerian graphs. 
%-------------------------------------------
\begin{defn}
We define the parity cut space of a graph to be the vector space of parity cuts on the graph.
\end{defn}
%-------------------------------------------
\begin{defn}
We define the \textit{characteristic vector} of a parity cut $S$ to be the vector $\vec{x}$ such that, for each $1 \le i \le n$, we have $x_i = 1$ if vertex $v_i \in S$ and $x_i = 0$ if $v_i \not \in S$. 
\end{defn}
%-------------------------------------------
The following major result shows the relation between the overlap graph of a permutation and its adjacency matrix, connecting two major concepts in graph theory and linear algebra.
%-------------------------------------------
\begin{thm}\label{thm:pcspace_eq_kernel}
The parity cut space of an Eulerian graph of a permutation is equivalent to the kernel of its adjacency matrix.
\end{thm}
%-------------------------------------------
\begin{proof}
Suppose we have a parity cut $S$ of an Eulerian graph $G = (V,E)$.  Now because $S$ is a parity cut, we know that for each vertex $v \in V\setminus S$, the number of edges of $v$ to vertices in $S$ will be even. Similarly, for each $v \in S$, the number of edges to vertices of in $V\setminus S$ will be even. However, since $G$ is Eulerian, this means that for each $v \in S$, the number of edges to vertices of in $S$ will also be even. That is, $S \subseteq V$ is a parity cut of $G$ if, and only if, for all $v \in V$, the number of edges from $v$ to vertices in $S$ is even. 

Suppose we have such a parity cut, $S$, and let $\vec{s}$ be the characteristic vector of this cut and $\vec{v}$ be the row representing vertex $v$ in the adjacency matrix of $G$. This means that $\vec{v} \cdot \vec{s}$ will be 0 (working over the finite field of size 2). Therefore, given the adjacency matrix $M$, this means that $M\vec{s} = \vec{0}$, so $\vec{s}$ must be in the kernel of $M$.

Conversely, if we find a vector $\vec{s}$ in the kernel of $M$, this means that for each vertex $v \in V$, $\vec{v} \dot \vec{s} = 0$, or that the number of edges of $v$ to vertices in the set $S$ defined by $\vec{s}$ must be even. Therefore, $\vec{s}$ must necessarily define a parity cut of $G$.
\end{proof}
%-------------------------------------------
Thus, we have the following sortability criteria for an Eulerian graph.
%-------------------------------------------
\begin{cor}
An Eulerian graph is \textbf{gcds} sortable if, and only if, the kernel of its adjacency matrix contains a vector $\vec{x}$ such that $x_1+x_n = 1$.
\end{cor}
%-------------------------------------------
\begin{proof}
This follows directly from Corollary \ref{cor:gcdssort_propa} and Theorem \ref{thm:pcspace_eq_kernel}.
\end{proof}
%-------------------------------------------
In this section we have talked about concepts related to Eulerian graphs. Because overlap graphs are a subsets of Eulerian graphs, the above lemmas and theorems are applicable in the context of permutations. Nevertheless, if we restrict ourselves in working on overlap graphs, we can have more precise properties of them that are related to what has already been covered. Such properties are what will be discussed in the next subsection.
%-------------------------------------------
\subsection{Permutations, Alternating Cycles, and Adjacency Matrices} 
%-------------------------------------------
The first lemma we will look at connects parity cut and alternating cycles, the latter being a concept that was introduced early on in the Terminology and Notations section.
%-------------------------------------------
\begin{lem}\label{lem:alt_cycle_pc}
For an overlap graph $G = (V,E)$ of a valid permutation, a partition of $V$ such that elements of each alternating cycle are contained in the same set constitutes a parity cut.
\end{lem}
%-------------------------------------------
\begin{proof}
It suffices to show that given an alternating cycle $(a_1, a_2, ..., a_k)$ of a permutation $\pi$ and any pointer $p$ not in that cycle, the pointer will intersect an even number of pointers of the cycle. First note that in $\pi$, the pointer $a_1, ...a_k$ will be arranged in pairs $a_{i}(a_{i-1}+1)$ (where subtraction is mod $k$). Now, between the $p$ tail pointer and the $p+1$ head pointer, there will be an even number of pointers of the alternating cycle, because they are situated in pairs. We say that pointer $q$ intersects pointer $p$ if either its head or tail pointer lies between the head and tail pointer of $p$, and its other pointer either precedes or succeeds both pointers of $p$ in the permutation. Furthermore, the pointers that we consider to belong to the alternating cycles are the head pointers of the $(a_{i-1}+1)$ and the tail pointers of the $a_i$. Since there are exactly two of these pointers per pair, there must be an even number of these pointers in between the head and tail pointers of $p$. Now, since each of these pointers corresponds to exactly one other pointer in the alternating cycle, we see that an even number of these pointers must connect to other such pointers in between the pointer of $p$, which means an even number must correspond to pointers either preceding or succeeding both pointers of $p$ in the permutation, which means an even number of these pointers intersect $p$. Thus, given any pointer $p$ and an alternating cycle of $\pi$, $p$ must have an even number of intersections with pointers of the alternating cycle, so in any partition of $V$ consisting of alternating cycles, every vertex in one set must intersect an even number of vertices of each cycle of the other set, and therefore this forms a parity cut.
\end{proof}
%-------------------------------------------
Having the restriction of working on overlap graphs rather than on Eulerian graphs, we obtain a more detailed property connecting the alternating cycles and the kernels of the adjacency matrix.
%-------------------------------------------
\begin{thm}
The alternating cycles of $\pi$ form an orthogonal basis for the kernel of the adjacency matrix of $\pi$.
\end{thm}
%-------------------------------------------
\begin{proof}
 We know that the characteristic vectors of the alternating cycles must be orthogonal (since the cycles are disjoint) and also that they all define parity cuts (see Lemma \ref{lem:alt_cycle_pc}, so it only remains to show that there are no parity cuts that are not formed by a linear combination of parity cuts formed from alternating cycles (under the symmetric difference operation). That is, every parity cut must be composed of complete alternating cycles.

\begin{center}
	\begin{tikzpicture}[scale=.8]
    %-------------------------------------------
    \filldraw[lightgray] (0,0) node(a1r)
    [shape=circle, draw=black, fill=lightgray, inner sep=3pt]
    [anchor=north]{};
    \filldraw[black] (1,0) node(g1)
    [shape=circle, draw=black, fill=black, inner sep=3pt]
    [anchor=north]{};
    \filldraw[lightgray] (2,0) node(a2r)
    [shape=circle, draw=black, fill=lightgray, inner sep=3pt]
    [anchor=north]{};
    \filldraw[white] (3,0) node(r1)
    [shape=circle, draw=black, fill=white, inner sep=3pt]
    [anchor=north]{};
    \filldraw[white] (4,0) node(r2)
    [shape=circle, draw=black, fill=white, inner sep=3pt]
    [anchor=north]{};
    \filldraw[lightgray] (5,0) node(a3r)
    [shape=circle, draw=black, fill=lightgray, inner sep=3pt]
    [anchor=north]{};
    \filldraw[black] (6,0) node(g2)
    [shape=circle, draw=black, fill=black, inner sep=3pt]
    [anchor=north]{};
    \filldraw[lightgray] (7,0) node(a1l)
    [shape=circle, draw=black, fill=lightgray, inner sep=3pt]
    [anchor=north]{};
    \filldraw[white] (8,0) node(r3)
    [shape=circle, draw=black, fill=white, inner sep=3pt]
    [anchor=north]{};
    \filldraw[lightgray] (9,0) node(a4r)
    [shape=circle, draw=black, fill=lightgray, inner sep=3pt]
    [anchor=north]{};
    \filldraw[black] (10,0) node(g3)
    [shape=circle, draw=black, fill=black, inner sep=3pt]
    [anchor=north]{};
    \filldraw[black] (11,0) node(g4)
    [shape=circle, draw=black, fill=black, inner sep=3pt]
    [anchor=north]{};
    \filldraw[lightgray] (12,0) node(a3l)
    [shape=circle, draw=black, fill=lightgray, inner sep=3pt]
    [anchor=north]{};
    \filldraw[white] (13,0) node(r4)
    [shape=circle, draw=black, fill=white, inner sep=3pt]
    [anchor=north]{};
    \filldraw[white] (14,0) node(r5)
    [shape=circle, draw=black, fill=white, inner sep=3pt]
    [anchor=north]{};
    \filldraw[lightgray] (15,0) node(a2l)
    [shape=circle, draw=black, fill=lightgray, inner sep=3pt]
    [anchor=north]{};
    \filldraw[black] (16,0) node(g5)
    [shape=circle, draw=black, fill=black, inner sep=3pt]
    [anchor=north]{};
    \filldraw[black] (17,0) node(g6)
    [shape=circle, draw=black, fill=black, inner sep=3pt]
    [anchor=north]{};
    \filldraw[lightgray] (18,0) node(a4l)
    [shape=circle, draw=black, fill=lightgray, inner sep=3pt]
    [anchor=north]{};
    \filldraw[white] (19,0) node(r6)
    [shape=circle, draw=black, fill=white, inner sep=3pt]
    [anchor=north]{};
    \filldraw[white] (20,0) node(r7)
    [shape=circle, draw=black, fill=white, inner sep=3pt]
    [anchor=north]{};
    
     \foreach \z in {1,...,4}
             {\draw[black, out=60, in=120] (a\z r.north) to (a\z l.north);}
	\end{tikzpicture}\\
    \vspace{2.5mm}
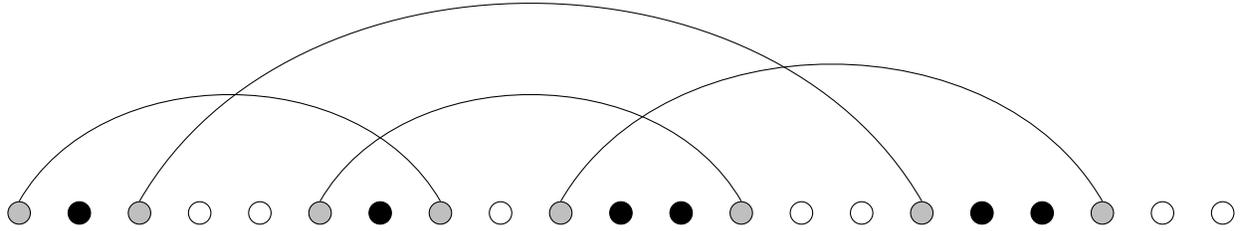
	FIGURE 9. Example Breakpoint Graph\\
\end{center}

Suppose we have a permutation and a set of pointers that constitute a parity a cut. We will show that these pointers must form complete alternating cycles. 

First, let us consider the head and tail of each pointer. We will call these $dots$. Given this permutation, we know that the first dot must be the $(0,1)$ tail pointer and the last dot must be the $(n,n+1)$ head pointer. Furthermore, if we start at the first dot and follow the pointer along to the second dot of that pointer, we will end up at the $(0,1)$ head pointer, and if we move to the dot directly to the right, we will be at the $(1,2)$ tail pointer. Since this is a valid permutation, if we continue to move along each pointer and then jump to the dot directly to the right until we reach the rightmost dot (the $(n,n+1)$ tail pointer), we will end up visiting every dot in the permutation exactly once.

Now, let us color blue the dots that correspond to pointers in the parity cut. We can now label the remaining dots as follows: dots with an even number of blue dots preceding them in the permutation will be colored red, and dots with an odd number of blue dots preceding them will be colored green. This can be further generalized by coloring the "territory" between every pair of blue pointers either red or green and noting that all vertices in that territory, if there are any, must have that color. There are a few things we can note. First, since every pointer in the parity cut must intersect an even number of other parity cut pointers, between any two blue dots of the same pointer, there must be an even number of blue dots. This means that if there is a red dot directly to the left of one blue dot of a pointer, there must also be a red dot directly to the right of the other blue dot of that pointer. More generally, if the left side of a blue dot is red territory, the right side of the other blue dot in that pointer must also be red territory. Second, since there must be an odd number of blue dots between dots of different color (red or green), dots can only be in the same pointer with dots of the same color. Finally, the first dot must be either red or blue, and if it is blue, the territory to its left must be red.

We can now attempt to visit every dot along the permutation as described above. Since the first dot must be red or blue, we start at either a red or blue dot and follow the pointer to another dot of the same color. Now each time we are at a red dot, we must follow the pointer to another red dot, and the dot directly to the right of that dot can either be red or blue. If it is blue, it must have red territory to its left. Each time we visit a blue dot with red territory to its left, we will follow the pointer to a blue dot with red territory to its right, which means the dot directly to the right of it must either be red or blue with red territory to its left. Thus, since we start at a dot that is either red or blue with red to its left, and any such dot must be connected by a pointer to a dot with another such dot directly to its right, the only possible dots we can visit must be either red, blue with red to its right, or blue connected by a pointer to another blue with red to its right. Since we must visit every single dot in the permutation, this means that there can be no green dots.

However, if there are no green dots, the blue pointers in the parity cut must form alternating cycles, because the blue dots will appear in perfect adjacent pairs. Since the green territories must all be empty, the first blue to will be adjacent to the second blue dot, the third will be adjacent to the fourth, and so on. Since the first dot has green territory to its right, its pointer pair must have green to its left, so there must be a blue dot directly to its left, and this blue dot will have green to its right also, so its pointer pair must have green to its left, and so on, thus forming a complete alternating cycle.
\end{proof}
%-------------------------------------------
A result related to the theorem above is presented in the following lemma, although now we are eliminating the first and last rows of the adjacency matrix.
%-------------------------------------------
\begin{lem}
Let $M$ be the adjacency matrix of the overlap graph of a permutation and $M'$ be $M$ without the first and last rows. The characteristic vector of the strategic pile is the only element of the kernel of $M'$ that consists of elements of the alternating cycle containing both roots, but does not itself contain both roots.
\end{lem}
%-------------------------------------------
\begin{proof}
First of all, every alternating cycle will be in this basis, and since they are disjoint, no alternating cycle will be a linear combination of the others. Now, consider the alternating cycle containing $(0,1)$, but without the strategic pile elements. We know that all of the dots in this set, except the first and last will be in the pairs $(b_{k+1}, b_k +1)$. However, we also know that no pointer can have one of its dots either before the $(0,1)$ tail pointer (the first dot in this set) or after the $(n,n+1)$ head pointer (the last dot in this set). Therefore, given any pointer other than the $(0,1)$ pointer or the $(n,n+1)$ pointer, there must be an even number of dots in between the head and tail of this pointer, which means this pointer intersects an even number of dots in this set. Therefore, the characteristic vector of this set must also be in the kernel.

Now, we must show that there are no other elements in the kernel. Take any element of this kernel. Because no valid permutation graph can have property $b$, we know that either both roots are in the same side of the cut and have odd degree with respect to the cut, or both roots have even degree with respect to the cut. In the latter case, our argument from above shows that the cut must be composed of alternating cycles. In the former case, we can again color the dots of the permutation as before, with dots in the set of the parity cut containing the roots colored blue, dots with an even number of blue dots preceding them colored red, and dots with an odd number of blue dots preceding them colored green. Now, since we know the first dot of the $(0,1)$ pointer must have red territory to its left, and that it has odd degree with respect to the parity cut, the second dot of this pointer must have an odd number of blue dots between the first dot and itself, and therefore must have green territory directly to its right. Similarly, since we know the last dot of the $(n,n+1)$ pointer must have red territory to its right, and that it has odd degree with respect to the parity cut, the first dot of this pointer must have an odd number of blue dots between the second dot and itself, and therefore must have green territory directly to its left. Therefore, for these two pointers, there is red territory directly to different sides of each dot, and green on the other. However, as our previous argument showed, fore every other pointer, the red territory must be on the same side of both dots. Therefore, if we again try to follow a path along the dots of the permutation in order, we will visit only green dots, which means there can be no red dots. Again, as before, this shows that the parity cut must be composed of alternating cycles, except for the cycle containing the $(0,1)$ and $(n,n+1)$ pointer. In this case, if we follow the pointer, we see that only the part of the cycle from the $(0,1)$ pointer up to and including the $(n,n+1)$ pointer are in the cut. However, this is also an element of the kernel. Therefore, the sets of alternating cycles and the strategic pile for a basis for the kernel for the adjacency matrix without the first and last rows.
\end{proof}
%-------------------------------------------
However, regardless of eliminating the first and the last rows of a sortable permutation's adjacency matrix, it conserves a property if compared to the original adjacency matrix.
%-------------------------------------------
\begin{lem}\label{lem:cent_pres_ker}
For a sortable permutation $\pi$ with adjacency matrix $M$ and $M'$ equal to $M$ with the first and last rows removed, $\text{ker}(M') = \text{ker}(M)$. 
\end{lem}
%-------------------------------------------
\begin{proof}
We know that if a permutation is sortable, then there will some element $\vec{x}$ such that $\vec{x}_1 = 1$ and $\vec{x}_n = 0$ and some element $\vec{y}$ such that $\vec{y}_1 = 0$ and $\vec{y}_n = 1$ in the kernel of $M$. However, this implies that the first column is linear combination of other columns, not including the last, and that the last column is a linear combination of other columns, not including the first. But this means that the first row and last rows are also linear combinations of the middle rows. Thus, removing these two rows does not reduce the dimension of $M$, and therefore the nullity of $M'$ is the same as the nullity of $M$. Furthermore, it is clear that every element of the kernel of $M$ is an element of the kernel of $M'$, so the kernel of $M'$ must be exactly the kernel of $M$. 
\end{proof}
%-------------------------------------------
\section{A decision problem and its complexity}

We shall now provide some definitions which will later on help us convert from the adjacency to the precedence matrix (and vice-versa). The first definition is the following:
%-------------------------------------------
\begin{defn}
For each positive integer $n$, define $B_n \in M_n(\mathbb{F}_2)$ so that for $1 \le i,j \le n$, we have $B_{n}(i,j) = 1$ if, and only if, $i = j$ or $i+1 = j$. 
\end{defn}
%-------------------------------------------
For example,
\[ B_5 =
  \begin{bmatrix}
    1 & 1 & 0 & 0 & 0 \\
    0 & 1 & 1 & 0 & 0\\
    0 & 0 & 1 & 1 & 0 \\
    0 & 0 & 0 & 1 & 1 \\
    0 & 0 & 0 & 0 & 1 \\
  \end{bmatrix}. \] 
%\end{ex} 
%-------------------------------------------
%-------------------------------------------
\begin{defn}
Define $Z : M_n(\mathbb F_2) \to M_{n+1}(\mathbb F_2)$ as follows: $$
Z(A)(i,j) = \left\{
	\begin{array}{ll}
           1 & \quad i = j = 1\\
            0 & \quad \text{exactly one of }$i$\text{ and }$j$\text{ is }$1$ \\
            A(i-1,j-1)& \quad \text{otherwise}\\
	\end{array}
\right.
$$
\end{defn}
%-------------------------------------------
For example,
\begin{ex} 
	\[Z \left( \begin{bmatrix}
    0 & 1 & 1 & 0 \\
    1 & 0 & 1 & 1 \\
    1 & 1 & 0 & 0 \\
    0 & 1 & 0 & 0 \\
  \end{bmatrix} \right) = \begin{bmatrix}
      1 & 0 & 0 & 0 & 0 \\
      0 & 0 & 1 & 1 & 0  \\
      0 & 1 & 0 & 1 & 1 \\
      0 & 1 & 1 & 0 & 0 \\
      0 & 0 & 1 & 0 & 0 \\
    \end{bmatrix}. \]
\end{ex}
%-------------------------------------------
%This next function, on the other hand, reduces the size of a matrix by one row and column rather than increasing it.
%-------------------------------------------
Define $F : M_n(\mathbb{F}_2) \to M_{n-1}(\mathbb{F}_2)$ as follows: \[F(A)(i,j) = A(i,j)\oplus A(i+1,j)\oplus A(i,j+1)\oplus A(i+1,j+1),\] for $1 \le i,j \le n-1$, where $\oplus$ denotes the addition operation in $\mathbb F_2$.
%\end{defn}
%-------------------------------------------
%\begin{ex} 
Using the same input as the one that was used in the previous example, we now show what happens if we apply $F$ rather than $Z$ to such input.
For example, 
\[ F \left( \begin{bmatrix}
    0 & 1 & 1 & 0 \\
    1 & 0 & 1 & 1 \\
    1 & 1 & 0 & 0 \\
    0 & 1 & 0 & 0 \\
  \end{bmatrix} \right) = 
  \begin{bmatrix}
  0 & 1 & 1 \\
  1 & 0 & 0 \\
  1 & 0 & 0 \\
  \end{bmatrix}, \]
%\end{ex}
%-------------------------------------------
%The last function of this section is also over matrices, yet it does not change its size.
%-------------------------------------------
%\begin{defn}
Define $S : M_n(\mathbb F_2) \to M_n(\mathbb F_2)$ as follows: \[ S(A)(i,j) \equiv \sum_{r=1}^i \sum_{c=1}^j A(r,c) \pmod 2. \] 
%\end{defn}
%-------------------------------------------
%\begin{ex} We now show the result of applying $S$ to the same input that was previously used.
For example,
\[ S \left( \begin{bmatrix}
    0 & 1 & 1 & 0 \\
    1 & 0 & 1 & 1 \\
    1 & 1 & 0 & 0 \\
    0 & 1 & 0 & 0 \\
  \end{bmatrix} \right) = \begin{bmatrix} 
  0 & 1 & 0 & 0 \\
  1 & 0 & 0 & 1 \\
  0 & 0 & 0 & 1 \\
  0 & 1 & 1 & 0 \\
   \end{bmatrix}.
   \]
%\end{ex}
%-------------------------------------------
We now show a relation between the adjacency and precedence matrices of a permutation.
%-------------------------------------------
\begin{lemma}\label{lem:adjprec}
Let $\pi \in S_n$ be a permutation, and let $A$ and $P$ be the adjacency and precedence matrices of $\pi$, respectively. Also, let $\delta(i,j) = 1$ if $i = j$, and $\delta(i,j) = 0$ if $i \not= j$ (i.e. $\delta$ is the Kronecker delta function). Then, for any pair $(i,j)$ with $0 \le i, j \le n$, we have that \[ A(i,j) = P(i,j)\oplus P(i+1,j)\oplus P(i,j+1)\oplus P(i+1,j+1)\oplus \delta(i,j) \oplus \delta(i+1,j), \] where  $\oplus$ denotes the addition operation in $\mathbb F_2$.
\end{lemma}
%-------------------------------------------
\begin{proof}
By definition, we have that $A(i,j) = 1$ if, and only if, the edges connecting the first and second occurrences of the pointers $(i,i+1)$ and $(j,j+1)$ in $\pi$ overlap. This will occur if, and only if, exactly one occurrence of $(j,j+1)$ lies between the two $(i,i+1)$ pointers; one such case is shown below.
\[ ^{(j,j+1)}{j+1} \dots i^{(i,i+1)} \dots j^{(j,j+1)} \dots ^{(i,i+1)}{(i+1)} \]

%We will prove the formula in four cases:
We consider the following cases: $i+1 < j$, $i+1 = j$, $i > j$, and $i = j$.\\

\noindent \emph{Case $1$:} $i+1 < j$. \\
In this case, we have that $i+1 < j$ implies that $i < i+1 < j < j+1$, so the four elements of the permutation adjacent to $(i,i+1)$ and $(j,j+1)$ pointers are all distinct. Therefore, we have that the $(j,j+1)$ pointer adjacent to $j$ is between the $(i,i+1)$ pointers if, and only if, $j$ is between $i$ and $i+1$ in $\pi$. This is the case if, and only if, exactly one of $i$ and $i+1$ precedes $j$ in $\pi$, or if $P(i,j)\oplus P(i+1,j) = 1$. Similarly, the $(j,j+1)$ pointer adjacent to $j+1$ is between the $(i,i+1)$ pointers if, and only if, $P(i,j+1) \oplus P(i+1,j+1) = 1$.

Since we have $A(i,j) = 1$ if, and only if, exactly one of the occurrences of the $(j,j+1)$ pointer is between the two $(i,i+1)$ occurrences, we will have $P(i,j) \oplus P(i+1,j) \oplus P(i,j+1) \oplus P(i+1,j+1)  = 1$. Therefore, it follows from the fact that $\delta(i,j) = \delta(i+1,j) = 0$ since $i+1 > j$ that we indeed have \[ A(i,j) = P(i,j)\oplus P(i+1,j)\oplus P(i,j+1)\oplus P(i+1,j+1)\oplus\delta(i,j) \oplus \delta(i+1,j) \] in this case.\\

\noindent \emph{Case $2$:} $i+1 = j$. \\
Since $i,i+1,$ and $j+1$ are distinct, similarly we have that the $(j,j+1)$ pointer adjacent to $j+1$ is between the two $(i,i+1)$ pointers if, and only if, $P(i,j+1)\oplus P(i+1,j+1) = 1$. 

However, since $i+1 = j$, determining whether the $(j,j+1)$ pointer adjacent to $j$ is between the $(i,i+1)$ pointers requires a subtler analysis. The $(j,j+1)$ and $(i,i+1)$ pointers adjacent to $j$ appear in the following order, since $i+1=j$:  \[ ^{(i,i+1)}j^{(j,j+1)}. \] Thus, we have that the $(j,j+1)$ pointer adjacent to $j$ appears between the two $(i,i+1)$ pointers if, and only if, the $(i,i+1)$ pointer adjacent to $i$ appears after $j$ in $\pi$; this holds if, and only if, the element $i$ comes after $j$ in the permutation $\pi$, or $P(i,j) = 0$. 

Since $\delta(i+1,j) = 1$ in this case, we have that $P(i,j) = 0$ if, and only if, $P(i,j) \oplus \delta(i+1,j) = 1$. As $i+1 = j$, we must have $P(i+1,j) = 0$, so $P(i,j) \oplus \delta(i+1,j) = P(i,j) \oplus P(i+1,j) \oplus \delta(i+1,j)$.

Together with the result for the $(j,j+1)$ pointer adjacent to $j+1$, we have that the $(i,i+1)$ and $(j,j+1)$ pointer pairs overlap if, and only if, $ P(i,j)\oplus P(i+1,j)\oplus P(i,j+1)\oplus P(i+1,j+1)\oplus \delta(i+1,j) = 1$. We also have $\delta(i,j) = 0$. Thus, \[P(i,j) =  P(i,j)\oplus P(i+1,j)\oplus P(i,j+1)\oplus P(i+1,j+1)\oplus\delta(i,j)\oplus \delta(i+1,j)\] when $i+1 = j$.\\

\noindent \emph{Case $3$:} $i > j$.\\
Since the adjacency matrix is symmetric, as the overlap graph is undirected, we have that $A(i,j) = A(j,i)$. By the results proven in cases $1$ and $2$, we have that \[ A(i,j) = A(j,i) = P(j,i) \oplus P(j+1,i) \oplus P(j,i+1) \oplus P(j+1,i+1) \oplus\delta(j,i) \oplus \delta(j+1,i). \] Note that, for $s \not= t$, we have that $P(s,t) = P(t,s) \oplus 1$, since $s$ precedes $t$ if, and only if, $t$ does not precede $s$, and if $s = t$ then $P(s,t) = P(t,s)$. Thus, if $i > j+1$ (in which case $\delta(j+1,i) = \delta(i+1,j) = 0$), then we have \begin{align*}
A(i,j) &= A(j,i) \\
&= P(j,i) \oplus P(j+1,i) \oplus P(j,i+1)\oplus P(j+1,i+1) \oplus \delta(j,i) \oplus \delta(j+1,i) \\
&= (1 \oplus P(i,j)) \oplus (1 \oplus P(i,j+1)) \oplus (1 \oplus P(i+1,j)) \oplus (1 \oplus P(i+1,j+1)) \oplus 0 \\
&= (1 \oplus 1 \oplus 1 \oplus 1) \oplus  P(i,j) \oplus P(i,j+1) \oplus  P(i+1,j) \oplus P(i+1,j+1) \\
&= P(i,j) \oplus P(i,j+1) \oplus  P(i+1,j) \oplus P(i+1,j+1) \oplus \delta(i,j) \oplus \delta(i+1,j),
\end{align*} 
as desired. If $i = j+1$, we have both $\delta(j+1,i) = 1$ and $P(j+1,i) = P(i,j+1)$ but $\delta(i+1,j) = 0$, so \begin{align*}
A(i,j) &= A(j,i) \\
&= P(j,i) \oplus P(j+1,i) \oplus P(j,i+1)\oplus P(j+1,i+1)\oplus \delta(j,i) \oplus \delta(j+1,i) \\
&= (1 \oplus P(i,j)) \oplus P(i,j+1) \oplus (1 \oplus P(i+1,j)) \oplus (1 \oplus P(i+1,j+1)) \oplus 1 \\
&= (1 \oplus 1 \oplus 1 \oplus 1) \oplus P(i,j) \oplus P(i,j+1) \oplus  P(i+1,j) \oplus P(i+1,j+1) \\
&= P(i,j) \oplus P(i,j+1) \oplus  P(i+1,j) \oplus P(i+1,j+1) \oplus \delta(i,j) \oplus \delta(i+1,j).
\end{align*} Hence, the result follows also in the case when $i > j$.\\

\noindent \emph{Case $4$:} $i = j$.\\
In this case, since no element precedes itself, we obtain $P(i,j) = P(i+1,j+1) = 0$. We also have that since $i+1$ precedes $i$ if, and only if, $i$ does not precede $i+1$, $P(i+1,j) \oplus P(i,j+1) = 1$. Thus, we have $P(i,j)\oplus P(i+1,j)\oplus P(i,j+1)\oplus P(i+1,j+1) = 1$. Since $\delta(i,j) = 1$ and $\delta(i+1,j) = 0$ in this case, and $A(i,j) = 0$ since no pointer overlaps itself, we have that \[A(i,j) =  P(i,j) \oplus P(i,j+1) \oplus  P(i+1,j) \oplus P(i+1,j+1) \oplus \delta(i,j) \oplus \delta(i+1,j) = 1 \oplus 1 \oplus 0 = 0 \] in this case as well.\\

Therefore, in all cases we have shown that \[ A(i,j) =  P(i,j)\oplus P(i+1,j)\oplus P(i,j+1)\oplus P(i+1,j+1)\oplus\delta(i,j) \oplus \delta(i+1,j),\] so the lemma is proven.
\end{proof}
%-------------------------------------------
%The previous lemma provides a starting point to having a conversion from the precedence matrix to the adjacency matrix of a permutation.
%-------------------------------------------
\begin{lemma}%{(Precedence to Adjacency Matrix)\\}
For any permutation $\pi \in S_n$ with adjacency matrix $A_\pi$ and precedence matrix $P_\pi$,%we have that 
\[ A_\pi = F(P_\pi)+B_{n+1}, \]
where the function $F$ and the matrix $B_{n+1}$ are as defined earlier.
\end{lemma}
%-------------------------------------------
\begin{proof}
We consider the above formula elementwise. By definition, we have that the element at position $(i,j)$ of $F(P_\pi)$ is $$P_\pi(i,j)\oplus P_\pi(i+1,j)\oplus P_\pi(i,j+1) \oplus P_\pi(i+1,j+1).$$ Also, note that $B_{n+1}(i,j) = 1$ if, and only if, $i=j$ or $i+1=j$, so it is equal to $\delta(i,j)\oplus \delta(i+1,j)$. Thus, for every pair $(i,j)$ with $ 1 \le i,j \le n+1$, the element at position $(i,j)$ of the matrix $F(P_\pi)+B_{n+1}$ is equal to the sum \[ P_\pi(i,j)\oplus P_\pi(i+1,j)\oplus P(i,j+1) \oplus P(i+1,j+1) \oplus \delta(i,j)\oplus \delta(i+1,j), \] which by Lemma \ref{lem:adjprec} is equal to $A_\pi(i,j)$; hence, we have that $A_\pi = F(P_\pi)+B_{n+1}$.
\end{proof}
%-------------------------------------------
%The opposite conversion, from the adjacency matrix to the precedence matrix, is stated and proved in the following lemma.
%-------------------------------------------
\begin{lemma}%{(Adjacency to Precedence Matrix)\\}
For any permutation $\pi \in S_n$, the following is true: \[ P_\pi = S(Z(A_\pi)+B_{n+2}). \]
\end{lemma}
%-------------------------------------------
\begin{proof}
For brevity, let $Z_\pi = Z(A_\pi)$, let $M_\pi = Z(A_\pi)+B_{n+2}$, and let $S_\pi = S(Z(A_\pi)+B_{n+2})$.
We claim that $P_\pi(i,j) = S_\pi(i,j)$ for all pairs $(i,j)$ with $1 \le i,j \le n+2$. 

We proceed by induction. For the base case, we will prove that for any $1 \le i \le n+2$, we have $P_\pi(i,1) = S_\pi(i,1)$ and $P_\pi(1,i) = S_\pi(1,i)$, and for the inductive step we will show that for $1 \le i,j \le n+1$, if $P_\pi(i,j) = S_\pi(i,j)$, $P_\pi(i+1,j) = S_\pi(i+1,j)$, and $P_\pi(i,j+1) = S_\pi(i,j+1)$, then $P_\pi(i+1,j+1) = S_\pi(i+1,j+1)$.

First, note that $M_\pi(1,1) = 1\oplus 1 = 0$, $M_\pi(1,2) = 0 \oplus 1 = 1$, for any $i \ge 2$ we have $M_\pi(1,i) = 0 \oplus 0 = 0$, and for any $i \ge 2$ we have $M_\pi(i,1) = 0 \oplus 0 = 0$. By definition of $S$, we have $$S_\pi(i,1) = \sum_{k=1}^i M_\pi(k,1) = 0+0+\dots+0 = 0$$ for each $1 \le i \le n+2$. Also, for $2 \le i \le n+2$ we have that $$S_\pi(1,i) = \sum_{k=1}^i M_\pi(1,k) = 0+1+0+\dots+0 = 1.$$ Moreover, since %$0$ is the first element in the extended permutation 
$\pi' = (0,\pi_1,\pi_2,\dots,\pi_n,n+1)$, %we have that $0$ precedes every element except itself and that there is no element that precedes $0$, so 
we have that $P_\pi(i,1) = 0$ for all $1 \le i \le n+2$ and $P_\pi(1,i) = 1$ for $2 \le i \le n+2$. Hence, we have shown that $P_\pi(1,i) = S_\pi(1,i)$ and $P_\pi(i,1) = S_\pi(i,1)$ for all $i$ with $1 \le i \le n+2$.%, so the base case is proven.

We will now prove the inductive step. Let $i$ and $j$ be integers with $1 \le i,j \le n+1$. Following the definition of $Z$ and $B$, we have that $Z_\pi(i+1,j+1) = A_\pi(i,j)$ and $B_{n+2}(i+1,j+1) = \delta(i+1,j+1) \oplus \delta(i+2,j+1)$. Since $i+1 = j+1 \iff i = j$ and $i+2 = j+1 \iff i+1 = j$, we also have that $B_{n+2}(i+1,j+1) = \delta(i,j)\oplus\delta(i+1,j)$. Hence, it follows that \[ M_\pi(i+1,j+1) = Z_\pi(i+1,j+1)\oplus B_{n+2}(i+1,j+1) = A_\pi(i,j) \oplus \delta(i,j)\oplus\delta(i+1,j). \]

Using induction, we assume that $P_\pi(i,j) = S_\pi(i,j)$, $P_\pi(i+1,j) = S_\pi(i+1,j)$, and $P_\pi(i,j+1) = S_\pi(i,j+1)$; we will prove that $P_\pi(i+1,j+1) = S_\pi(i+1,j+1)$. From the definition of $S$, we have the following equalities: 
\begin{align*}
S_\pi(i+1,j+1) &= \sum_{r=1}^{i+1} \sum_{c=1}^{j+1} M_\pi(r,c) \\
&= \sum_{r=1}^{i+1} \sum_{c=1}^{j+1} M_\pi(r,c) \oplus  2 \sum_{r=1}^i \sum_{c=1}^j M_\pi(r,c) \\
&= \sum_{r=1}^i \sum_{c=1}^j M_\pi(r,c) \oplus \sum_{r=1}^i M_\pi(r,j+1) \oplus \sum_{c=1}^j M_\pi(i+1,c) \\&\oplus M_\pi(i+1,j+1) \oplus \sum_{r=1}^i \sum_{c=1}^j M_\pi(r,c) \oplus \sum_{r=1}^i \sum_{c=1}^j M_\pi(r,c) \\
&= \sum_{r=1}^i \sum_{c=1}^j M_\pi(r,c) \oplus \left( \sum_{r=1}^i M_\pi(r,j+1) \oplus \sum_{r=1}^i \sum_{c=1}^j M_\pi(r,c) \right) \\&\oplus \left( \sum_{c=1}^j M_\pi(i+1,c) \oplus \sum_{r=1}^i \sum_{c=1}^j M_\pi(r,c) \right) \oplus M_\pi(i+1,j+1) \\
&= \sum_{r=1}^i \sum_{c=1}^j M_\pi(r,c) \oplus \sum_{r=1}^i \sum_{c=1}^{j+1} M_\pi(r,c) \oplus \sum_{r=1}^{i+1} \sum_{c=1}^j M_\pi(r,c) \oplus M_\pi(i+1,j+1) \\
&= S_\pi(i,j)\oplus S_\pi(i,j+1)\oplus S_\pi(i+1,j)\oplus M_\pi(i+1,j+1) \\
&= S_\pi(i,j)\oplus S_\pi(i,j+1)\oplus S_\pi(i+1,j)\oplus A(i,j) \oplus \delta(i,j) \oplus \delta(i+1,j)
\end{align*}
By the inductive assumption, we have that $S_\pi(i,j) = P_\pi(i,j), S_\pi(i,j+1) = P_\pi(i,j+1)$, and $S_\pi(i+1,j) = P_\pi(i+1,j)$, so it follows that \[ S_\pi(i+1,j+1) = P_\pi(i,j)\oplus P_\pi(i,j+1)\oplus P_\pi(i+1,j)\oplus A_\pi(i,j) \oplus \delta(i,j) \oplus \delta(i+1,j).\]
By Lemma \ref{lem:adjprec}, we have that \[ A_\pi(i,j) = P_\pi(i,j)\oplus P_\pi(i+1,j)\oplus P_\pi(i,j+1)\oplus P_\pi(i+1,j+1) \oplus \delta(i,j) \oplus \delta(i+1,j). \] Adding $A_\pi(i,j) \oplus P_\pi(i+1,j+1)$ to both sides and reducing the resulting sums modulo $2$, we have that 
\begin{align*} A_\pi(i,j) \oplus P_\pi(i+1,j+1) \oplus A_\pi(i,j) &= P_\pi(i,j)\oplus P_\pi(i+1,j)\oplus P_\pi(i,j+1)\oplus P_\pi(i+1,j+1) \\&\oplus \delta(i,j) \oplus \delta(i+1,j) \oplus A_\pi(i,j) \oplus P_\pi(i+1,j+1), \end{align*} and so \begin{align*} P_\pi(i+1,j+1) &= P_\pi(i,j)\oplus P_\pi(i+1,j)\oplus P_\pi(i,j+1)\oplus \delta(i,j) \oplus \delta(i+1,j) \oplus A_\pi(i,j) \\&=  P_\pi(i,j)\oplus P_\pi(i,j+1)\oplus P_\pi(i+1,j)\oplus A_\pi(i,j) \oplus \delta(i,j) \oplus \delta(i+1,j). \end{align*} Therefore, it follows that $ S_\pi(i+1,j+1) = P_\pi(i+1,j+1)$, and so by induction we have that $S_\pi(i,j) = P_\pi(i,j)$ for all $i$ and $j$ with $ 1 \le i,j \le n+2$. Hence, $P_\pi = S_\pi(Z_\pi(A_\pi)+B_{n+2})$.
\end{proof}
%-------------------------------------------
With the insight obtained from the above lemmas on conversions between the two main matrices of a permutation, we can address the following decision problems. The first one, involving the \textbf{cds} move graph of a permutation, was posed in \cite{p3}. \\\\
%-------------------------------------------
\indent Decision problem $1$: \textbf{CDS MOVE GRAPH} \cite{p3}\\

\indent INSTANCE: A finite graph $G$.\\
\indent QUESTION: Is there a permutation $\pi$ with \textbf{cds} move graph isomorphic to $G$?\\
 %-------------------------------------------
Although this question of determining whether a permutation with move graph isomorphic to a given graph exists still appears to be difficult, we found that a polynomial time algorithm exists for the following simpler decision problem.\\
%-------------------------------------------
\vspace{0.2in}

\indent Decision problem $2$: \textbf{LABELED CDS MOVE GRAPH}\\

\indent INSTANCE: A finite graph $G$ with vertices labeled $(1,2)$ through $(n-1,n)$.\\
\indent QUESTION: Is there a permutation $\pi$ with \textbf{cds} move graph equal to $G$?\\
%-------------------------------------------
\begin{lemma}
The decision problem \textbf{LABELED CDS MOVE GRAPH} is in $\mathbf{P}$, where $\mathbf{P}$ denotes the set of decision problems solvable in polynomial time.
\end{lemma}
%-------------------------------------------
\begin{proof}
Let $G$ be such a labeled move graph on $n-1$ vertices, and let $M$ be the adjacency matrix of $G$. If there is some permutation $\pi$ on $n$ vertices with move graph adjacent to $G$, then the overlap graph of $\pi$ will consist of $G$ with two additional handle vertices added and additional edges adjacent to at least one handle added as well. The adjacency matrix of the overlap graph of $\pi$ is thus an $(n+1) \times (n+1)$ matrix $A$ such that $A(i+1,j+1) = M(i,j)$ for $1 \le i,j \le n-1$.

Since the overlap graph of a permutation is an Eulerian graph, we have that $A(i,n+1) \equiv  \sum_{j=1}^n A(i,j)\pmod 2$, so the final row and column of the adjacency matrix $A$ are determined by the remaining rows and columns. 

Let $P \in M_{n+2}$ be the precedence matrix of the permutation $\pi$. Then, row $1$ of $P$ is the vector $[0\ 1\ 1\ \dots\ 1]$ and row $n+2$ is the vector $[0\ 0\ 0\ \dots\ 0]$. For each $i$ with $1 \le i \le n$, let $s_i$ be the number of indices $j$ such that $P(i+1,j) = 1$. Let $\pi_i$ be the element in position $i$ of the permutation $\pi$. Then, we have that $s_{\pi_i} = n+1-i$, since $\pi_i$ precedes the elements $\pi_{i+1}$ through $\pi_n$ and $n+1$ in the permutation $\pi$. If we exclude the entries in columns $0$ and $(n+1)$ from the sum, it will be $n-i$, since all $\pi_i$ in the permutation precede $n+1$ and are preceded by $0$. Therefore, the row sums of the central submatrix of the precedence matrix $P$ form a permutation of the sequence $(0,1,2,\dots, n-1)$.

Let $T_n$ be the set of all $n \times n$ matrices $C$ such that the sums $\sum_{j=1}^n C(i,j)$ for $1 \le i \le n$ form a permutation of $(0,1,2,\dots,n-1)$, $C(i,i) = 0$, and for all $i \not= j$ we have $C(i,j) = 1-C(j,i)$. We claim that a matrix $C$ is the central submatrix of the precedence matrix of some permutation $\pi$ of length $n$ if, and only if, $C \in T_n$. We showed that if $C$ is the central submatrix of a permutation's precedence matrix, then $C$ is an element of $T_n$ (the final two conditions follow since no element can precede itself, and if $i$ precedes $j$ then $j$ does not precede $i$), and so it suffices to show that each element of $T_n$ corresponds to some permutation. 

First, note that if $\pi$ and $\pi'$ are distinct permutations in $S_n$, then $\pi_i \not= \pi'_i$ for some $i>0$. This implies that different rows in the matrices associated with $\pi$ and $\pi'$ will have sum $(n-i)$. Since there is only one row in each matrix with this sum, we have that the two matrices are distinct. Hence, since $|S_n| = n!$, there are $n!$ distinct matrices in $T_n$ corresponding to permutations. 

We claim that $|T_n| = n!$. By our earlier argument, it suffices to show that each element of $T_n$ corresponds to some permutation. Furthermore, since each of the $n!$ permutations corresponds to a distinct element of $T_n$, we already have $|T_n| \ge n!$, so it suffices to show that $|T_n| \le n!$. To show this, we proceed by induction on $n$. For $n = 1$, the only $1 \times 1$ matrix whose row sums form a permutation of $(0)$ is the matrix $\begin{bmatrix}0\\\end{bmatrix}$, so we indeed have $|T_1| \le 1$ and the base case holds. 

Now assume that $|T_n| \le n!$ for $n = k$; we will prove that this also holds for $n=k+1$. Let $C$ be a matrix in $T_{k+1}$. Then, by definition, there is some integer $i$ with $1 \le i \le k+1$ such $C(i,j) = 0$ for all $j$ with $1 \le j \le k+1$. By definition of $T_n$, this implies that $C(j,i) = 1$ for each $1 \le j \le k+1$ such that $j \not= i$. Let $C'$ be the $k \times k$ submatrix of $C$ with the $i$th row and $i$th column removed. Then, the sum of each row in $C'$ is equal to the corresponding sum in $C$ minus $1$, and the row with sum $0$ in $C$ was removed to form $C'$, so the row sums of $C'$ form a permutation of $(0,1,2,\dots,k-1)$. Furthermore, the equalities $C'(i,i) = 0$ for $1 \le i \le k$ and $C'(i,j) = 1-C'(j,i)$ for $1 \le i,j \le k$, $i \not= j$ follow from the corresponding properties of $C$, since the indices of the deleted row and column were equal. Thus, we have that $C' \in T_{k}$. By our inductive hypothesis, there are at most $k!$ matrices $C' \in T_k$, so since there are $k+1$ possible locations of the zero row in the matrix $T_{k+1}$, there are at most $(k+1)!$ possibilities for the matrix $C$. Hence, we have that $|T_{k+1}| \le (k+1)!$, and so by induction our claim is proven.

Therefore, we have that an $n \times n$ matrix $C$ is the central submatrix of the precedence matrix of some permutation $\pi \in S_n$ if, and only if, $C \in T_n$. 

Let $M$ be the adjacency matrix of the $(n-1)$ vertex move graph $G$, let $\vec{u} \in \mathbb{F}_2^n$ such that $u_i=1$ if, and only if, the vertex $(i-1,i)$ is adjacent to the last handle $(n,n+1)$ after it is added, and let $\vec{v}\in \mathbb{F}_2^n$ such that $v_i= 1$ if, and only if, the vertex $(i-1,i)$ is adjacent to the first handle $(0,1)$ after it is added (so in particular $v_1 = 0$). Then, the complete adjacency matrix, in block form, is the matrix \[ A = \left[ \begin{array}{c|c|c}
0&\vec{v}^T&x \\\hline
\vec{v}&M&\vec{u} \\\hline
x&\vec{u}^T&x
\end{array} \right],\] where $x \in \{ 0,1 \}$. If $A$ is the adjacency matrix of a permutation, then since the overlap graph of any permutation is Eulerian, we must have that $v_i+\sum_{j=1}^{n-1} M(i,j) + u_i \equiv 0 \pmod 2$, so it follows that $\vec{u} = \vec{v} + M \cdot [1\ 1\ 1\ \dots\ 1]^T$. By our formula proven earlier, we have that if $A$ is the complete adjacency matrix (equal to $M$ with additional rows and columns added at the beginning and end) and it corresponds to a permutation $\pi$ with precedence matrix $P$, then $P = S(Z(A)+B_{n+1})$, where the functions $S$ and $Z$ and the matrix $B_{n+1}$ are as defined in the lemma. Using the formula above, we can compute the element at position $(i,j)$ of the matrix $P$ as follows: \[ P(i,j) = \sum_{r=1}^i \sum_{c=1}^j (Z(A)+B_{n+1})(r,c).\]

By definition of $Z$, $Z(A)(1,1) = 1$, $Z(A)(1,i) = Z(A)(i,1) = 1$ for all $2 \le i \le n$, and $Z(A)(i,j) = A(i,j)$ for all $2\le i,j \le n$. Now, let $A'$ be the matrix with all zeroes in the top row and left column, and the entries of $A$ in the remaining positions. Then, \[ B = B_{n+1} + \begin{pmatrix}1&0&0&\dots&0 \\
0&0&0&\dots&0 \\
\vdots&\vdots&\vdots&\ddots&\vdots \\
0&0&0&\dots&0 \\
0&0&0&\dots&0\end{pmatrix} = \begin{pmatrix}0&0&0&\dots&0&0 \\
0&1&1&\dots&0&0 \\
0&0&1&\dots&0&0 \\
\vdots&\vdots&\vdots&\ddots&\vdots&\vdots \\
0&0&0&\dots&1&1 \\
0&0&0&\dots&0&1\end{pmatrix}.\] It follows that \[ P(i,j) = \sum_{r=1}^i \sum_{c=1}^j B_{n+2}(r,c) \oplus \sum_{r=1}^i \sum_{c=1}^j A'(r,c) = \sum_{r=1}^i \sum_{c=1}^j B_{n+2}(r,c) \oplus \sum_{r=2}^i \sum_{c=2}^j A(r,c).\]
For $2 \le i,j \le n+1$, we have that $\sum_{r=2}^i \sum_{c=2}^j A(r,c) = \sum_{r=2}^i v_r \oplus \sum_{c=2}^j v_c + \sum_{r=3}^i \sum_{c=3}^j M(r-1,c-1)$ (where the second sum is taken to be $0$ if $i=2$ or $j=2$). Hence, for $2 \le i,j \le n+1$, we conclude that \[ P(i,j) = \left( \sum_{r=1}^i \sum_{c=1}^j B_{n+2}(r,c) \oplus \sum_{r=3}^i \sum_{c=3}^j M(r-1,c-1) \right) \oplus\left(\sum_{r=2}^i v_r \oplus \sum_{c=2}^j v_c\right). \] 

Note that the first sum can be computed directly given the matrix $M$ in polynomial time ($O(n^4)$ by computing each term directly, or $O(n^2)$ by applying dynamic programming to the partial sums). If the precedence matrix $P$ is in $T_n$, then the row sums $\sum_{j=2}^n P(i,j)$ for each $i$ (not taken modulo $2$) form a permutation of $(0,1,2,\dots,n-1)$, so there is some row with sum $0$. 

Thus, there must be some $i$ such that if $\vec w$ is the $n$-element vector over $\mathbb{F}_2^n$ with \[ w_j = \left( \sum_{r=1}^i \sum_{c=1}^j B_{n+2}(r,c) \oplus \sum_{r=3}^i \sum_{c=3}^j M(r-1,c-1) \right),\] then we must have that $\left(\sum_{r=2}^i v_r + \sum_{c=2}^j v_c\right) = w_j$ for each $1 \le j \le n$ (over $\mathbb F_2$). As the vector $\vec v$ can be uniquely determined from this system of $n$ equations, there are at most $n$ possible values of this vector, namely the vectors $\vec w$ determined by each row $i$ of the matrix $M$. 

Each entry in one of the $\vec w$ vectors corresponding to a row of the matrix can be computed using $O(n^2)$ operations, and there are $n$ entries in each vector and $n$ vectors, so all $n$ of the vectors $\vec w$ corresponding to the rows can be computed in $O(n^4)$, yielding the possible values for $\vec v$ in polynomial time. 

Given a vector $\vec v$, we can compute the vector $\vec u$ corresponding to the rightmost column using the formula $\vec u = \vec v + M \cdot [1,1,1,\dots,1]^T$, derived above. 
As each entry in $A$ except for the single number $x = A(1,n) = A(n,1)$ is determined and $x \in \{ 0,1 \}$, there are only two possible adjacency matrices $A$ for each $\vec v$. Therefore, there are $2n$ possible adjacency matrices $A$ that could correspond to permutations with move graph $G$ (with adjacency matrix $M$). Then, using the formula $P = S(Z(A)+B_{n+1})$, we can compute the precedence matrix associated with each adjacency matrix $A$ in polynomial time: since there are $O(n^2)$ terms in the sum used to compute each entry in the matrix $P$ and $O(n^2)$ entries, this computation requires at most $O(n^4)$ time and so is polynomial. Finally, for each precedence matrix $P$ the row sums can be computed and placed in increasing order in $O(n^2)$ time, thus allowing us to decide if $P \in T_n$ in polynomial time. As this process only needs to be performed for $2n$ matrices, we can determine if there is a possible adjacency matrix $A$ with labeled move graph $G$ corresponding to a permutation in $O(n^5)$ time, so we conclude that the problem \textbf{LABELED CDS MOVE GRAPH} is in \textbf{P}.
\end{proof}
%-------------------------------------------
%After presenting the permutation results related to graphs and matrices, we can now look at these two mathematical objects under which conditions are they sortable under their respective \textbf{cds} operation. Notice that now we are not restricting ourselves to talk about overlap graphs nor even Eulerian graphs. The results of the upcoming section will be applicable to a more generalized graph.
%-------------------------------------------
%-------------------------------------------
\section{Graph and Matrix Sortability}\label{sec4}
%-------------------------------------------
%-------------------------------------------
\subsection{Sortability under \textbf{gcds}}
%-------------------------------------------
%The definition of a parity cut to a simple graph, and thus not just to an Eulerian graph, is defined as follows:
A parity cut of any simple graph (and thus not just an Eulerian graph), is defined as follows:
%-------------------------------------------
\begin{defn}
Let $G = (V,E)$ be a simple graph. A generalized parity cut is a set $V_1 \subseteq V$ such that $\forall v \in V, \delta_{V_1}(v)$ is even.
\end{defn}
%-------------------------------------------
One can tell if a subset of the vertex set is a generalized parity cut from the lemma that appears below.
%-------------------------------------------
\begin{lem}
Let $G = (V,E)$ and $G' = (V,E') = \textbf{gcds}_{\{p,q\}}(G)$ for two vertices $p,q \in V$. If we consider a subset $V_1$ of $V$, then $V_1$ is a generalized parity cut of $G$ if, and only if, there exists a generalized parity cut $V_1'$ of $G'$ such that for all vertices $v \neq p,q \in V$, $v \in V_1$ if and only if $v \in V_1'$.
\end{lem}
%-------------------------------------------
\begin{proof}
Consider a generalized parity cut of $G$ with set $V_1$. We will show that for any vertex $u \in V$, when \textbf{gcds} is performed on two vertices $p,q$, the parity of $\delta_{V_1}(u)$ does not change, and that therefore this is also a generalized parity cut for $G'$. 

We can again consider the sets $A = \{v \in V_1| v \neq q, \text{ v is adjacent only to $p$}\}, B = \{v \in V_1| v \neq p,  \text{ v is adjacent only to $q$}\}$, and $C = \{v \in V_1| \text{ v is adjacent to both $p$ and $q$}\}$, and suppose $u$ is adjacent to $x$ vertices in $A$, $y$ vertices in $B$, and $z$ vertices in $C$. We now have three cases:
\begin{enumerate}
	\item[\emph{Case 1.}] $p,q \in V_1$ and are adjacent to each other. %Since $p$ and $q$ are both in $V_1$ and adjacent to each other, we know that - change suggested by Liljana
Then $|N(p)| \setminus \{q\} = |A \bigcup C| = |A| + |C|$ is odd, and similarly $|B|+|C|$ is odd. Therefore, if $|C|$ is odd, $|A|+|B|$ is even, and if $|C|$ is even, $|A|+|B|$ is also even. 
    \begin{enumerate}
    \item $u$ is adjacent to both $p$ and $q$. Then when the \textbf{gcds} operation is performed, edges between $u$ and vertices of $A$ and $B$ will switch, and the edges to $p$ and $q$ will be removed which means $u' \in V(G')$ will have $\delta_{V_1}(u') = \delta_{V_1}(u)+(|A|-x)-x+(|B|-y)-y-2 = \delta_{V_1}(u)+|A|+|B|-2x-2y-2$ edges, so the change in $\delta_{V_1}(u)$ is even. 
    \item $u$ is adjacent to exactly one of $p$ and $q$. Without loss of generality, suppose $u$ is adjacent only to $p$. Then, when the \textbf{gcds} operation is performed, edges between $u$ and vertices of $B$ and $C$ will switch and the edge to $p$ will be removed, meaning $u' \in V(G')$ will have $\delta_{V_1}(u') = \delta_{V_1}(u)+(|B|-y)-y+(|C|-z)-z-1 = \delta_{V_1}(u)+|B|+|C|-2y-2z-1$ edges. This is again even, so either way the change in $\delta_{V_1}(u)$ is even.
    \end{enumerate}
	\item[\emph{Case 2.}] Exactly one of $p,q$ is in $V_1$: This is a similar argument to our argument in case $1$. 
Without loss of generality, suppose $p$ is in $V_1$. In this case, we know that $|B|+|C|$ is odd, and $|A|+|C|$ is even. Therefore, if $|C|$ is odd, $|A|+|B|$ is odd, and if $|C|$ is even, $|A|+|B|$ is also odd. 
    \begin{enumerate}
    \item If $u$ is adjacent to both $p$ and $q$, when the \textbf{gcds} operation is performed, edges between $u$ and vertices of $A$ and $B$ will switch, and the edge to $p$ will be removed meaning $u' \in V(G')$ will have $\delta_{V_1}(u') = \delta_{V_1}(u)+(|A|-x)-x+(|B|-y)-y-1 = \delta_{V_1}(u)+|A|+|B|-2x-2y-1$ edges, so the change in $\delta_{V_1}(u)$ is even. 
    \item Otherwise, if $u$ is adjacent to only $p$, when the \textbf{gcds} operation is performed, edges between $u$ and vertices of $B$ and $C$ will switch and the edge to $p$ will be removed, meaning $u' \in V(G')$ will have $\delta_{V_1}(u') = \delta_{V_1}(u)+(|B|-y)-y+(|C|-z)-z-1 = \delta_{V_1}(u)+|B|+|C|-2x-2y-1$ edges, which is also even. 
    \item Finally, if $u$ is adjacent to only $q$, when the \textbf{gcds} operation is performed, edges between $u$ and vertices of $A$ and $C$ will switch, meaning $u' \in V(G')$ will have $\delta_{V_1}(u') = \delta_{V_1}(u)+(|A|-x)-x+(|C|-z)-z = \delta_{V_1}(u)+|A|+|C|-2x-2z$ edges, which is also even.
    \end{enumerate}
	\item[\emph{Case 3.}] $p,q \not\in V_1$: In this case, we know that $|B|+|C|$ and $|A|+|C|$ are both even, so $|A|+|B|$ must be even. 
    \begin{enumerate}
    \item If $u$ is adjacent to both $p$ and $q$, when the \textbf{gcds} operation is performed, edges between $u$ and vertices of $A$ and $B$ will switch, so $u' \in V(G')$ will have $\delta_{V_1}(u') = \delta_{V_1}(u)+(|A|-x)-x+(|B|-y)-y = \delta_{V_1}(u)+|A|+|B|-2x-2y$ edges, which is even. 
    \item Otherwise, if $u$ is adjacent to only one of the vertices, without loss of generality $p$, when the \textbf{gcds} operation is performed, edges between $u$ and vertices of $B$ and $C$ will switch, which means $u' \in V(G')$ will have $\delta_{V_1}(u') = \delta_{V_1}(u)+(|B|-y)-y+(|C|-z)-z = \delta_{V_1}(u)+|B|+|C|-2x-2y$ edges, which is also even.
    \end{enumerate}
    \end{enumerate}

    In any case, for any vertex $u \neq p,q \in V$, performing \textbf{gcds} on $G$ doesn't change the parity of $\delta_{V_1}(u)$, so if $V_1$ is a generalized parity cut of $G$, it is also a generalized parity cut of $G'$.
    
Now consider a generalized parity cut, $V_1'$ of $G'$. Since $p$ and $q$ are isolated in $G'$, they can be moved in or out of $V_1'$ to create $V_1$, which will still be a generalized parity cut. If we consider $V_1'$ excluding $p$ and $q$, there are 3 cases:
  \begin{enumerate}
  	\item The parities of $\delta_{V_1'}(p)$ and $\delta_{V_1'}(q)$ are both even: In this case, since applying \textbf{gcds} doesn't change the parity of $\delta_{V_1}(u)$ for any vertex $u$, $V_1'$ is a generalized parity cut of $G$ also.
    \item Exactly one of $\delta_{V_1'}(p)$ and $\delta_{V_1'}(q)$ is odd: without loss of generality suppose only $\delta_{V_1'}(p)$ is odd. Then, let $V_1$ be $V_1' \bigcup \{q\}$. Since $p$ is adjacent to $q$, this means that $\delta_{V_1}(p)$ will now be even, so again, by the preservation of parity, $V_1$ will be a generalized parity cut of $G$. 
  	\item $\delta_{V_1'}(p)$ and $\delta_{V_1'}(q)$ are both odd: Similarly, in this case we can let $V_1$ be $V_1' \bigcup \{p,q\}$ and since $p$ and $q$ are adjacent to each other, this means that $\delta_{V_1}(p)$ and $\delta_{V_1}(q)$ will both now be even, so again, by the preservation of parity, $V_1$ will be a generalized parity cut of $G$.
  \end{enumerate}	
\end{proof}
%-------------------------------------------
We now have a condition for having a generalized parity cut with only one root.
%-------------------------------------------
\begin{thm}
A two-rooted graph $G = (V,E)$ has a generalized parity cut containing exactly one of the two roots and another cut containing exactly the other root if, and only if, its image under \textbf{gcds} does too.
\end{thm}
%-------------------------------------------
\begin{proof}
This follows directly from Lemma \ref{lem:cent_pres_ker}. \end{proof}
%-------------------------------------------
Thus, we have a necessary, and sufficient, condition for a two-rooted graph to be sortable under \textbf{gcds}.
%-------------------------------------------
\begin{cor}
A two-rooted graph $G = (V,E)$ is \textbf{gcds} sortable if, and only if, there exists a generalized parity cut containing exactly one of the two roots and another cut containing exactly the other root.
\end{cor}
%-------------------------------------------
\begin{proof}
Consider any fixed point graph. There must be some non-root vertex $v$ adjacent to at least one root and no other non-root vertices. However, in this case there cannot be a generalized parity cut containing only this root, because then the number of edges from $v$ to this cut will be only 1. Therefore, such a cut does not exist for any fixed point graph. Then, by the above theorem, no unsortable graph can have such a cut.
However, the discrete graph does have both generalized parity cuts (simply the sets of the roots alone), so every sortable graph must also have such cuts.
\end{proof}
%-------------------------------------------
\begin{cor}
\textbf{gcds} Inevitability
\end{cor}
%-------------------------------------------
%In this section, we also have results on \textbf{gcds} that relate to matrices. One of these results is provided in the theorem below.
%-------------------------------------------
\begin{thm}
Let $M$ be the adjacency matrix of a graph $G$ and $\text{cent}_{r,c}(M)$ be $M$ with the first and last rows and columns removed. The number of \textbf{gcds} operations it takes to reach any fixed point of $G$ is always $\frac{1}{2}$ rank($\text{cent}_{r,c}(M)$).
\end{thm}
%-------------------------------------------
\begin{proof}
First, %by [reference], 
the rank of a skew symmetric matrix will always be even. Since $M$ is symmetric, $\text{cent}_{r,c}(M)$ must be also, and since we are working mod 2, it must also be skew symmetric, so the rank of $\text{cent}_{r,c}(M)$ will be even. Now, consider the kernel of $\text{cent}_{r,c}(M)$. Clearly, if there are any edges between non-root vertices $p$ and $q$, $\text{cent}_{r,c}(M_{(p-1)(q-1)}) = \text{cent}_{r,c}(M_{(q-1)(p-1)}) = 1$ and $\vec{e_{(p-1)}}, \vec{e_{(q-1)}}$ will not be in $\ker(\text{cent}_{r,c}(M))$. Then, if we consider $M' = \textbf{mcds}(M) = MI_{pq}M + M = (MI_{pq}+I)M$, as shown above, the nullity of $MI_{pq}+I$ will be $2$, and similarly, since the first and last columns of $MI_{pq}+I$ will be $\vec{e_{1}}$ and $\vec{e_{n}}$, the nullity of $\text{cent}_{r,c}(MI_{pq}+I)$ will be 2. Therefore, the rank of $M$ will decrease by 2 at every step.
Since $\text{cent}_{r,c}(M)$ is effectively the adjacency matrix of the subgraph induced on the set of all non-root vertices of $G$, when we reach a fixed point, no non-root vertices will be adjacent, so $\text{cent}_{r,c}(M)$ will be the zero matrix and have rank 0. Therefore, the number of steps to get from $G$ to any fixed point must be $\frac{1}{2}$ rank($\text{cent}_{r,c}(M)$). 
\end{proof}
%-------------------------------------------
The above shows the connection between graph theory and linear algebra when talking about \textbf{gcds} on simple graphs. Next, we examine the case of sortability under \textbf{cds} on matrices, which we refer to as \textbf{mcds}. %Thus, it is natural to now talk about sortability under \textbf{cds} on matrices, this referred to in the beginning of the present paper as \textbf{mcds}.
%-------------------------------------------
\subsection{Sortability under \textbf{mcds}}
%-------------------------------------------
%We begin to present our results on the topic of \textbf{mcds} sortability by giving a property that depends on the entries of the matrix being analyzed.
%-------------------------------------------
\begin{lem}\label{lem:mcds_zero_col}
Let $\vec{M}'_q$ be the $q$th column of $M' = \textbf{mcds}(M)$. If $M(p,q) = 1$, then $\vec{M}'_q = \vec{0}$. %Rephrasing suggested by editor #13.
\end{lem}
%-------------------------------------------
\begin{proof}
By the definition of \textbf{mcds}, $M'(i,j) = M(i,j) + M(i,p)M(q,j) + M(i,q)M(p,j)$. If $M(p,q) = 1$, then for all $0 \leq k \leq n$, $M'(k,q) = M(k,q) + M(k,p)M(q,q) + M(k,q)M(p,q) = M(k,q)+M(k,q) = 0$.
\end{proof}
%-------------------------------------------
%The following lemma gives a relation between the kernel of a matrix and the kernel of the matrix after applying \textbf{mcds} to it.
%-------------------------------------------
\begin{lem}\label{lem:mcds_pres_ker}
Given matrix $M$, % its image under the \textbf{mcds} operation, 
$\text{ker}(M) \subseteq \text{ker}(M')$, where $M'=\textbf{mcds}(M)$.
\end{lem}
%-------------------------------------------
\begin{proof}
For any $\vec{x} \in \text{ker}(M), M'\vec{x} = (M+MI_{pq}M)\vec{x} = (I+MI_{pq})M\vec{x} = (I+MI_{pq})\vec{0} = \vec{0}$, so $\vec{x} \in \text{ker}(M')$.
\end{proof}
%-------------------------------------------
From the above, we have necessary and sufficient conditions of the columns of a matrix and that matrix after some \textbf{mcds} operation was applied to it.
%-------------------------------------------
\begin{thm}\label{thm:mcds_pres_gpcr}
An $n \times n$ matrix $M$ has an element $\vec{x}$ such that $x_1 = 0$ and $x_n = 1$ and an element $\vec{y}$ such that $y_1 = 1$ and $y_n = 0$ in $\text{ker}(M)$ if, and only if, its image under \textbf{mcds} does too.
\end{thm}
%-------------------------------------------
\begin{proof}
$\Rightarrow$ By Lemma $\ref{lem:mcds_pres_ker}$, if $M$ has such an $\vec{x}$ in its kernel, $M' = \textbf{mcds}_{\{p,q\}}(M)$ must too.

$\Leftarrow $ Consider an element $\vec{x} \in \ker(M')$. It must either be in the kernel of $M$ or it must be the case that $M\vec{x}$ is in the kernel of $I+MI_{pq}$. Now, $MI_{pq}$ has the form $\begin{bmatrix} \vec{0} & \cdots & \vec{q} & \cdots & \vec{p} & \cdots & \vec{0} \end{bmatrix}$, where $\vec{p}$ is the $p^{th}$ column of M and the $q^{th}$ column of $MI_{pq}$, and $\vec{q}$ is the $q^{th}$ column of M and the $p^{th}$ column of $MI_{pq}$. 

Now, if $M(p,q) = 1$ and $M(q,p) = 1$, $I+MI_{pq}$ must have rank $n-2$, because $n-2$ columns will just be $\vec{e_i}, i \in \{1, ..., n\} \setminus {p,q}$, and $(I+MI_{pq})(p,q) = (I+MI_{pq})(q,p) = 0$ because $(MI_{pq})(p,p) = M(p,q) = 1$ and $(MI_{pq})(q,q) = M(q,p) = 1$. Now by Sylvester's theorem, $\rank(M') \geq \rank(M)+\rank(I+MI_{pq}) - n = \rank(M) + (n-2) - n = \rank(M) -2$, which means that by the Rank-Nullity Theorem, $\nul(M') = n - \rank(M') \leq n-(\rank(M) - 2) = \nul(M)+2$. Also, because $M(p,q) = 1$ and $M(q,p) = 1$, $\vec{p}$ and $\vec{q}$ cannot be $\vec{0}$, which means that $\hat{e_p}$ and $\hat{e_q}$ were not elements of $\text{ker}(M)$. However, we know from Lemma $\ref{lem:mcds_zero_col}$ that the $p^{th}$ and $q^{th}$ columns of $M'$ will both be $\vec{0}$, so $\hat{e_p}$ and $\hat{e_q}$ will both be in $\ker(M')$. This accounts for the 2 additional vectors in the basis of $\text{ker}(M')$, so the only vectors added to the basis are $\hat{e_p}$ and $\hat{e_q}$, neither of which fit the condition that their first and last elements sum to 1. Therefore, the desired $\vec{x}$ will be in the kernel of $M'$ only if it was in the kernel of $M$.

Similarly, if exactly one of $M(p,q)$ and $M(q,p)$ is $0$, $I+MI_{pq}$ must have rank $n-1$ and can have at most 1 additional vector in the basis of its kernel. We know that one of $M(p,q)$ and $M(q,p)$ is $1$, and therefore either $\vec{p}$ or $\vec{q}$ cannot be $\vec{0}$, which means that either $\hat{e_p}$ or $\hat{e_q}$ were not elements of $\text{ker}(M)$. However, by Lemma $\ref{lem:mcds_zero_col}$ if $M(p,q) = 1$, the $q^{th}$ column of $M'$ will be $\vec{0}$ and if $M(q,p) = 1$, the $p^{th}$ column of $M'$ will be $\vec{0}$ so exactly one of $\hat{e_p}$ or $\hat{e_q}$ will be in $\ker(M')$, which accounts for the additional vector in the basis of $\text{ker}(M')$. Thus, the only vector added to the basis will be exactly one of $\hat{e_p}$ and $\hat{e_q}$, neither of which fits the condition that their first and last elements sum to 1, so the desired $\vec{x}$ will be in the kernel of $M'$ only if it was in the kernel of $M$.

Finally, if both $M(p,q)$ and $M(q,p)$ are $0$, $I+MI_{pq}$ must have dimension $n$ and $M'$ will have the same kernel as $M$.
\end{proof}
%-------------------------------------------
From the above, a \textbf{mcds} sortability criterion was found.
%-------------------------------------------
\begin{cor}\label{cor:mcds_sort_cond}
A matrix is \textbf{mcds}-sortable if, and only if, its kernel contains an element $\vec{x}$ such that $x_1 = 0$ and $x_n = 1$ and an element $\vec{y}$ such that $y_1 = 1$ and $y_n = 0$.
\end{cor}
%-------------------------------------------
\begin{proof}
For any fixed point matrix $M$, $1 < i,j < n$, we must have that $M(i,j) = 0$. If $\ker(M)$ contains an element $\vec{x}$ such that $x_1 = 0$ and $x_n = 1$, $M(i,n)$ must also be 0 for all $1 \leq i \leq n$, and similarly if $\ker(M)$ contains an element $\vec{y}$ such that $y_1 = 1$ and $y_n = 0$, $M(i,1)$ must also be 0 for all $1 \leq i \leq n$. Therefore the only such fixed point matrix must be the zero matrix, so by Theorem $\ref{thm:mcds_pres_gpcr}$ every fixed point of a matrix with this property must be the zero matrix and therefore such a matrix must be \textbf{mcds}-sortable. Similarly no fixed point of a matrix without this property can be the zero matrix, so a matrix is \textbf{mcds}-sortable if and only it has this property.
\end{proof}
%-------------------------------------------
A second \textbf{mcds} sortability criterion is presented below.
%-------------------------------------------
\begin{cor}\label{cor:mcds_sort_cond2}
Given a matrix $A \in M_n$, let $M = \text{cent}_c(A)$ and let $\vec{a_1}$ and $\vec{a_n}$ be the first and last columns of $A$, respectively. Then, $A$ is \textbf{mcds}-sortable if, and only if, there exist vectors $\vec{u},\vec{u'} \in \mathbb{F}_2^{n-2}$ such that $M\vec{u} = \vec{a_1}$ and $M\vec{u'} = \vec{a_n}$. 
\end{cor}
%-------------------------------------------
\begin{proof}
We will first prove that if $A$ is \textbf{mcds}-sortable, then such vectors $\vec{u}$ and $\vec{u'}$ exist. By Corollary \ref{cor:mcds_sort_cond}, we have that if $A$ is \textbf{mcds}-sortable, then $\ker(A)$ contains some vectors $\vec{v}$ and $\vec{w}$ such that $v_1 = 1$, $v_n = 0$, $w_1 = 0$, $w_n = 1$.

Let $\vec{u} \in \mathbb{F}_2^{n-2}$ be the column vector containing the elements $v_2$ through $v_{n-1}$, and define $\vec{u'}$ similarly for $\vec{w}$. Then, by definition of the kernel, we have that $M \vec u +\vec  a_1 = 0$ and $M \vec u' + \vec a_n = 0$. Since vector addition is its own inverse in $\mathbb{F}_2^{n-2}$, we have that $M\vec u = \vec a_1$ and $M \vec u' = \vec a_n$, as desired.

Conversely, if no such vector $\vec{u}$ exists, then there is no solution $\vec u$ to the equation $M \vec u = \vec a_1$ and hence no solution to $M \vec u + \vec a_1 = \vec 0$. Therefore, there can be no vector $\vec v \in \ker(A)$ such that $v_1 = 0$ and $v_n = 1$, so by Corollary \ref{cor:mcds_sort_cond}, $A$ is not \textbf{mcds}-sortable. The argument for when no such $\vec{u'}$ exists is similar, so both directions hold.
\end{proof}

%-------------------------------------------
\section{Counting \textbf{gcds}-sortable Graphs}
%-------------------------------------------
Using the sortability criterion proven in Corollary \ref{cor:mcds_sort_cond}, we prove several lemmas regarding the structure of general and Eulerian \textbf{gcds}-sortable graphs. Afterwards, we derive formulas counting the numbers of both general and Eulerian \textbf{gcds}-sortable graphs on $n$ vertices.

Let $M_n$ be the set of all $n \times n$ matrices $A$ with coefficients in $\mathbb{F}_2$ such that all diagonal elements of $A$ are zero and $A = A^T$. Note that an $n \times n$ matrix is the adjacency matrix of a simple undirected graph on $n$ vertices if, and only if, it is in $M_n$.

For brevity, we will say that an $n \times n$ matrix $A$ is \textit{Eulerian} if it is the adjacency matrix of an Eulerian graph.
%-------------------------------------------
\begin{defn}
For a vector $\vec{v} \in \mathbb{F}_2^n$, we define the \text{one's complement vector} of $\vec{v}$, denoted $\vec{v}^C$, to be the vector \[ \vec{v}^C = \vec{v}+ [1\ 1\ \dots\ 1]^T,\] where the vector addition is taken over $\mathbb{F}_2^n$.
\end{defn}
%-------------------------------------------
Using the definition above, we have the following lemma:
%-------------------------------------------
\begin{lemma}\label{lem:gcdscount0}
Let $A \in M_n$ be the adjacency matrix of an Eulerian graph $G$, let $M = \text{cent}_c(A)$, and for $1 \le i \le n$ let $\vec{a_i}$ be the $i$th column of $A$. Then, for any vector $\vec{u}$ such that $M\vec{u} = \vec{a_1}$, we have that $M\vec{u}^C = \vec{a_n}$.
\end{lemma}
%-------------------------------------------
\begin{proof}
By definition, for each $\vec{u} \in \mathbb{F}_2^{n-2}$, we have $\vec{u}^C = [1\ 1\ \dots\ 1]^T + \vec{u}$. If $M \vec{u} = \vec{a_1}$, then we have that $M \vec u = \vec{a_1}$ and so $M \vec{u} + \vec{a_1} = \vec{0}$. Since $G$ is Eulerian, we have that $$\vec{a_1}+(\vec{a_2}+\dots+\vec{a_{n-1}})+\vec{a_n} = \vec{0},$$ so $$M \left( [1\ 1\ \dots\ 1]^T \right)  + \vec{a_1}+\vec{a_n} = \vec{0}.$$ Therefore, $M (\vec{u}+\vec{u}^C)+\vec{a_1}+\vec{a_n} = \vec 0$, so by distributivity and commutativity we have that $M \vec{u} + \vec{a_1} + M\vec{u}^C+\vec{a_n} = 0$. Since $M\vec{u}+\vec{a_1} = 0$, it follows that $M\vec{u}^C + \vec{a_n} = \vec{0}$, and since addition is its own inverse over $\mathbb{F}_2^n$ we have that $M \vec{u}^C = \vec{a_n}$.
\end{proof}
%-------------------------------------------
A property regarding the dot product of specific vectors is shown below.
%-------------------------------------------
\begin{lemma}\label{lem:gcdscount1}
If $A$ is a matrix in $M_n$ and $\vec{u}, \vec{v} \in \mathbb{F}_2^n$, then we have $(A \vec{u}) \bullet \vec{v} = (A \vec{v}) \bullet \vec{u}$ and $(A \vec u) \bullet \vec{u} = 0$.
\end{lemma}
%-------------------------------------------
\begin{proof}
Expanding the dot product, we have that \[(A \vec{u}) \bullet \vec{v} = \sum_{i=1}^n (A \vec{u})_i \vec{v}_i = \sum_{i=1}^n \left( \sum_{j=1}^n A(i,j)\vec{u}_j \right) \vec{v}_i = \sum_{i=1}^n \sum_{j=1}^n A(i,j) \vec{u}_j \vec{v}_i. \] Since $A = A^T$, we have that \[ \sum_{i=1}^n \sum_{j=1}^n A(i,j) \vec{u}_j \vec{v}_i = \sum_{i=1}^n \sum_{j=1}^n A(j,i) \vec{u}_i \vec{v}_j = \sum_{i=1}^n \sum_{j=1}^n A(i,j) \vec{v}_j \vec{u}_i = (A \vec{v}) \bullet \vec{u}, \] as desired. Additionally, we have that 
\begin{align*}
(A \vec u) \bullet \vec u &= \sum_{i=1}^n 
\sum_{j=1}^n A(i,j) \vec{u}_i \vec{u}_j \\&= \sum_{i=2}^n \sum_{j=1}^{i-1} A(i,j) \vec{u}_i \vec{u}_j + \sum_{i=1}^n A(i,i)\vec{u}_i^2 + \sum_{i=2}^n \sum_{j=1}^{i-1} A(j,i) \vec{u}_i \vec{u}_j \\&= 2 \sum_{i=2}^n \sum_{j=1}^{i-1} A(i,j) \vec{u}_i \vec{u}_j + 0 = 0,\end{align*} where the addition is taken over $\mathbb F_2$.
\end{proof}
%-------------------------------------------
%A function that will later help give information on \textbf{mcds}-sortable matrices is now defined.
%-------------------------------------------
\begin{defn}
For each positive integer $n$, let $F : M_n \times \mathbb{F}_2^n \times \mathbb{F}_2^n \to M_{n+2}$ be the function defined as follows. For $A \in M_n$ and $\vec{u_1}, \vec{u_2} \in \mathbb{F}_2^n$, we define \[ F(A, \vec{u_1}, \vec{u_2}) = \left[\begin{array}{c|c|c}
0&\left(A\vec{u_1}\right)^T& \left(A \vec{u_1} \right)^T \bullet \vec{u_2} \\\hline
A\vec{u_1}&A&A\vec{u_2} \\\hline
\left(A \vec{u_1} \right)^T \bullet \vec{u_2} &\left(A\vec{u_2}\right)^T&0 \end{array} \right]. \]
\end{defn}
%-------------------------------------------
%-------------------------------------------
\begin{lemma}\label{lem:gcdscount2}
For any \textbf{mcds}-sortable $n \times n$ matrix $A \in M_n$, there exist vectors $\vec{u_1}, \vec{u_2} \in \mathbb{F}_2^{n-2}$ such that \[A = F(\fcent(A), \vec{u_1}, \vec{u_2}).\]  
\end{lemma}
%-------------------------------------------
\begin{proof}
For $1 \le i \le n$, we let $\vec{a_i}$ denote the $i$th column of the matrix $A$. Note that since $A$ is the adjacency matrix of an undirected graph, we have $A = A^T$, so $\vec{a_i}$ is also the $i$th row of $A$. We will also let $A(i,j)$ denote the element of $A$ in row $i$ and column $j$. Finally, we will let $M = \text{cent}_c(A)$ and $C = \fcent(A)$.

By Corollary \ref{cor:mcds_sort_cond2}, since $A$ is \textbf{mcds}-sortable there exist vectors $\vec{u_1}$ and $\vec{u_2}$ in $\mathbb{F}_2^{n-2}$ satisfying $M\vec{u_1} = \vec{a_1}$ and $M \vec{u_2}= \vec{a_n}$. We claim that $A = F(C, \vec{u_1}, \vec{u_2})$ for these $\vec{u_1}$ and $\vec{u_2}$. In particular, we will show that each section of the block matrix given in the definition of $F(C, \vec{u_1}, \vec{u_2})$ is equal to the corresponding submatrix of $A$. 

First, since $A \in M_n$, all of the diagonal elements of $A$ are $0$, so we have that the top-left and bottom-right blocks of $F(C, \vec{u_1}, \vec{u_2})$ match the corresponding elements of $A$. Also, we have by definition that $C$ is the $(n-2) \times (n-2)$ submatrix of $A$ containing the elements $A(i,j)$ for $2 \le i,j \le n-1$, so the central block of $F(C, \vec{u_1}, \vec{u_2})$ is also correct.

Since the vector in $\mathbb{F}_2^{n-2}$ containing the elements in positions $2$ through $n-1$ of the $n$-element vector $M\vec{u_1}$ is $C\vec{u_1}$ (as $C$ comprises rows $2$ through $n-1$ of $M$), we have that since $\vec{a_1} = M\vec{u_1}$, the vector containing the elements in positions $2$ through $n-1$ of $\vec{a_1}$ is $C\vec{u_1}$. Therefore, the block in row $2$, column $1$ of the block matrix $F(C, \vec{u_1}, \vec{u_2})$ is correct. Since $A = A^T$, the block in row $1$, column $2$ is also correct.

By similar reasoning, we have that since $\vec{a_n} = M\vec{u_2}$, the vector containing the elements in positions $2$ through $n-1$ of $\vec{a_n}$ is $C\vec{u_2}$. Therefore, the block in row $2$, column $3$ of the block matrix $F(C, \vec{u_1}, \vec{u_2})$ is correct, and so since $A = A^T$ the block in row $3$, column $2$ is also correct.

It remains to prove that the blocks in row $1$, column $3$ and row $3$, column $1$ are correct. We proved earlier that the block in row $1$, column $2$ is correct, so $\left[ A(1,2) \ A(1,3) \ \dots\ A(1,n-1) \right] = (C \vec{u_1})^T$. By definition of $\text{cent}_c(A) = M$, this row vector $(C\vec{u_1})^T$ is also the first row of the middle submatrix $M$ of $A$. Since we have $M \vec{u_2} = \vec{a_n}$ from the fact that $A$ is \textbf{mcds}-sortable, it follows that $(C\vec{u_1})^T \bullet \vec{u_2} = a_{1n}$. Therefore, the block in row $1$, column $3$ is correct, and since $A$ is symmetric the block in row $3$, column $1$ is correct as well. 

Thus, all of the blocks have been verified, and so we have shown that for some vectors $\vec{u_1}, \vec{u_2} \in \mathbb{F}_2^{n-2}$, we have $A = F(\fcent(A), \vec{u_1}, \vec{u_2})$. 
\end{proof}
%-------------------------------------------
From the above lemma, we have this corollary.
%-------------------------------------------
\begin{cor}\label{cor:gcdscor2}
For any \textbf{mcds}-sortable Eulerian $n \times n$ matrix $A \in M_n$, there exists a vector $\vec{u} \in \mathbb{F}_2^{n-2}$ such that $$A = F(\fcent(A), \vec{u}, \vec{u}^C).$$ 
\end{cor}
%-------------------------------------------
\begin{proof}
As before, let $M = \text{cent}_c(A)$ and let $\vec{a_i}$ be the $i$th column vector of $A$ for $1 \le i \le n$. By Lemma \ref{lem:gcdscount1}, since $A$ is \textbf{mcds}-sortable there are vectors $\vec{u}, \vec{u_2} \in \mathbb{F}_2^{n-2}$ such that $A = F(\fcent(A), \vec{u}, \vec{u_2})$, $M \vec{u} = \vec{a_1}$, and $M \vec{u_2} = \vec{a_n}$. But by Lemma \ref{lem:gcdscount0}, since $A$ is Eulerian and $M \vec{u} = \vec{a_1}$, we have $M \vec{u}^C = \vec{a_n}$. Since the only restriction placed on $\vec{u_2}$ in the proof of Lemma \ref{lem:gcdscount0} was that $M\vec{u_2} = \vec{a_n}$, we can replace $\vec{u_2}$ with $\vec{u}^C$ to get $A = F(\fcent(A), \vec{u}, \vec{u}^C)$, as desired.
\end{proof}
%-------------------------------------------
This function $F$ outputs an \textbf{mcds}-sortable graph, as shown in the upcoming lemma.
%-------------------------------------------
\begin{lemma}\label{lem:gcdscount3}
For any matrix $A \in M_n$ and any vectors $\vec{u_1}, \vec{u_2} \in \mathbb{F}_2^n$, the matrix $F(A, \vec{u_1}, \vec{u_2})$ is an element of $M_{n+2}$ and is \textbf{mcds}-sortable.
\end{lemma}
%-------------------------------------------
\begin{proof}
Let $M = F(A, \vec{u_1}, \vec{u_2})$. 

We begin by proving that $M \in M_{n+2}$; to do this, we must show that each diagonal entry of $M$ is a zero and that $M = M^T$. 

By definition of $F$, we have $M(1,1) = 0$, $M(i+1,i+1) = A(i,i)$ for $1 \le i \le n$, and $M(n+2,n+2) = 0$. Since $A \in M_n$, we have $A(i,i) = 0$ for each $1 \le i \le n$, so it follows that all of the diagonal entries of $M$ are zeroes.

Next, we will show that $M = M^T$. To do this, we will show that for each block submatrix $B(i,j)$ appearing in the block form of $M$ in the definition of $F$, we have that $B(i,j)^T = B(j,i)$ (for each $1 \le i,j \le 3$). First, since $B(1,1) = B(3,3) = \begin{bmatrix} 0\end{bmatrix}$, this relation holds for $(i,j) = (1,1)$ and $(3,3)$. Since $A \in M_n$, we have $A = A^T$ by definition, so since $B(2,2) = A$ it also holds for $(i,j) = (2,2)$. By definition of $F$, we have $B(2,1) = A \vec{u_1}$, $B(1,2) = (A \vec{u_1})^T$, $B(2,3) = A \vec{u_2}$, and $B(3,2) = (A \vec{u_2})^T$, so we additionally have $B(i,j)^T = B(j,i)$ for $(i,j) = (1,2),(2,1),(2,3),$ and $(3,2)$. Finally, both $B(1,1)$ and $B(3,3)$ are equal to the single element matrix $\begin{bmatrix} (A \vec{u_1})^T \bullet \vec{u_2} \end{bmatrix}$, so the cases $(i,j) = (1,1)$ and $(3,3)$ hold as well. Therefore, we have proven that $M = M^T$, so $M \in M_{n+2}$. 

Next, we will show that $M$ is \textbf{mcds}-sortable. By Corollary \ref{cor:mcds_sort_cond2}, it suffices to find vectors $\vec{u}$ and $\vec{u'} \in \mathbb{F}_2^n$ such that $\text{cent}_c(M) \vec{u} = \vec{m_1}$ and $\text{cent}_c(M) \vec{u'} = \vec{m_{n+2}}$. Let $P = \text{cent}_c(M)$. We claim that the vectors $\vec{u} = \vec{u_1}$ and $\vec{u'} = \vec{u_2}$ satisfy the required conditions; we must show that $P\vec{u_1} = \vec{m_1}$ and $P\vec{u_2} = \vec{m_{n+2}}$. 

By definition of $F$, we have that $P, \vec{m_1}$, and $\vec{m_{n+2}}$ can be expressed in block form as follows: \[ P = \left[ \begin{array}{c} \left(A\vec{u_1}\right)^T \\\hline A \\\hline \left(A \vec{u_2}\right)^T \end{array} \right] , \vec{m_1} = \left[ \begin{array}{c} 0 \\\hline A\vec{u_1} \\\hline \left( A \vec{u_1} \right)^T \bullet \vec{u_2} \end{array}\right], \text{ and } \vec{m_{n+2}} = \left[ \begin{array}{c} \left( A \vec{u_1} \right)^T \bullet \vec{u_2} \\\hline A\vec{u_2} \\\hline 0 \end{array}\right]. \]

We will prove that $P\vec{u_1} = \vec{m_1}$ and $P\vec{u_2} = \vec{m_{n+2}}$ by showing that for each component index $i$ with $1 \le i \le n+2$, $\left(P\vec{u_1} \right)_i = \left( \vec{m_1}\right)_i$ and $\left(P\vec{u_2} \right)_i = \left( \vec{m_{n+2}}\right)_i$.

First, for $2 \le i \le n+1$, we have that $(P\vec{u_1})_i = \vec{a_{i-1}} \bullet \vec{u_1} = (A \vec{u_1})_{i-1} = (\vec{m_1})_i$. Similarly, $(P \vec{u_2})_i = \vec{a_{i-1}} \bullet \vec{u_2} = (A \vec{u_2})_{i-1} = (\vec{m_{n+2}})_i$. 

Next, we consider the case in which $i = 1$. In this case, applying Lemma \ref{lem:gcdscount1} we have that $(P\vec{u_1})_1 = (A \vec{u_1})^T \bullet  \vec{u_1} = 0 = (\vec{m_1})_1$ and $(P \vec{u_2})_1 = (A \vec{u_1})^T \bullet \vec{u_2} =(\vec{m_{n+2}})_1 $, as needed.

Finally, we consider the case with $i = n+2$. We have by Lemma \ref{lem:gcdscount1} that $(P\vec{u_1})_{n+2} = (A \vec{u_2})^T \vec{u_1} = (A \vec{u_2}) \bullet \vec{u_1} = (A \vec{u_1})^T \vec{u_2} = (\vec{m_1})_{n+2}$ and $(P \vec{u_2})_{n+2} = (A \vec{u_2})^T \bullet \vec{u_2} = 0 = (\vec{m_{n+2}})_{n+2}$.

Therefore, we have shown that each pair of corresponding entries are equal, so we have that $P\vec{u_1} = \vec{m_1}$ and $P\vec{u_2} = \vec{m_{n+2}}$, and the lemma is proven.
\end{proof}
%-------------------------------------------
Moreover, $F$ outputs an Eulerian graph under the following circumstances.
%-------------------------------------------
\begin{cor}\label{cor:gcdscor3}
For any matrix $A \in M_n$ and any vector $\vec{u} \in \mathbb{F}_2^n$, the matrix $F(A, \vec{u}, \vec{u}^C)$ is both \textbf{mcds}-sortable and Eulerian. 
\end{cor}
%-------------------------------------------
\begin{proof}
Let $M = F(A, \vec{u}, \vec{u}^C)$. Since $\vec{u}^C \in \mathbb{F}_2^n$, it follows directly from Lemma \ref{lem:gcdscount3} that $M$ is \textbf{mcds}-sortable; hence, it suffices to prove that $M$ is Eulerian. 

$M$ represents the adjacency matrix of a graph $G$, and is Eulerian if, and only if, $G$ is Eulerian. Each position $j$ in row $i$ such that $M(i,j) = 1$ corresponds to a vertex $j$ such that there is an edge between vertices $i$ and $j$. Thus, for all vertices in $G$ to have even degree we must have that the sum of the elements in each row is even, or equivalently is $0$ when taken over $\mathbb F_2$. 

We begin by showing that the sum of the first row is zero over $\mathbb F_2$. By definition of $M = F(A, \vec{u}, \vec{u}^C)$, we have: \begin{align*}
\sum_{i=1}^{n+2} M(1,i) &= 0 + \sum_{i=1}^n (A\vec{u})_i + (A\vec{u})^T \bullet  \vec{u}^C \\
&= \sum_{i=1}^n \left( \sum_{j=1}^n A(i,j) u_j \right) + \sum_{i=1}^n (A\vec{u})_i \left(u^C\right)_i \\
&= \sum_{i=1}^n \sum_{j=1}^n A(i,j) u_j + \sum_{i=1}^n  \left( \sum_{j=1}^n A(i,j) {u}_j \right) \left({u^C}\right)_i \\
&= \sum_{i=1}^n \sum_{j=1}^n A(i,j) {u}_j  \left(1+\left({u^C}\right)_i\right) \\&= \sum_{i=1}^n \sum_{j=1}^n A(i,j) {u}_j {u}_i
\\ &= \sum_{i=2}^n \sum_{j=1}^{i-1} A(i,j) {u}_j {u}_i + \sum_{i=1}^n A(i,i) {u}_i^2 + \sum_{i=2}^n \sum_{j=1}^{i-1} A(j,i) {u}_i {u}_j.
\end{align*}
Since $A \in M_n$ the main-diagonal elements of $A$ are all zeroes and $A = A^T$. Therefore, we have $\displaystyle{\sum_{i=1}^n} A(i,i) {u}_i^2 = 0$ and $\displaystyle{\sum_{i=2}^n \sum_{j=1}^{i-1}} A(i,j) {u}_j {u}_i = \displaystyle{\sum_{i=2}^n \sum_{j=1}^{i-1}} A(j,i) {u}_i {u}_j$, so it follows that the sum of the first row is even. By similar reasoning, and the fact that \[ \left(A \vec{u} \right)^T \bullet \vec{u}^C = \sum_{i=1}^n \sum_{j=1}^n A(i,j) {u}_j ({u^C})_i = \sum_{i=1}^n \sum_{j=1}^n A(j,i) {u}_i ({u^C})_j = \sum_{i=1}^n \sum_{j=1}^n A(i,j)  ({u^C})_j {u}_i= \left(A \vec{u}^C \right)^T \bullet \vec{u},\] we also have that the sum of the last row is even. 

We will now show that the sum of each row $\vec{m_i}$ in the middle, for $2 \le i \le n+1$, is even. By definition of $F$, we have that for $1 \le i \le n$, \[ \sum_{j=1}^{n+2} M(i+1,j) = (A\vec{u})_i + \sum_{j=1}^n A(i,j) + (A\vec{u}^C)_i. \] By definition of $\vec{u}^C$, we have $\vec{u}+\vec{u}^C = [1\ 1\ \dots\ 1]^T$, so by commutativity and distributivity we have that \begin{align*} (A\vec{u})_i + \sum_{j=1}^n A(i,j) + (A\vec{u}^C)_i &= (A(\vec{u}+\vec{u}^C))_i + \sum_{j=1}^n A(i,j) \\ &= \left(A \left( [1\ 1\ \dots\ 1]^T \right)\right)_i + \sum_{j=1}^n A(i,j) \\ &=  \sum_{j=1}^n A(i,j) + \sum_{j=1}^n A(i,j),
\end{align*}
which is even. Hence, the sum of each row in the matrix $M = F(A, \vec{u}, \vec{u}^C)$ is even, so $F(A, \vec{u}, \vec{u}^C)$ is Eulerian and the corollary is proven.
\end{proof}
%-------------------------------------------
A necessary and sufficient condition related to the kernel of a matrix and the function $F$ applied to it is expressed below.
%-------------------------------------------
\begin{lemma}\label{lem:gcdscount4}
Let $A$ be a matrix in $M_n$, and let $\vec{u_1}, \vec{u_2}, \vec{v_1}, \vec{v_2}$ be vectors in $\mathbb{F}_2^n$. Then, we have that $F(A, \vec{u_1}, \vec{u_2}) = F(A, \vec{v_1}, \vec{v_2})$ if, and only if, $\vec{v_1}-\vec{u_1} \in \ker(A)$ and $\vec{v_2}-\vec{u_2} \in \ker(A)$.
\end{lemma}
%-------------------------------------------
\begin{proof}
If $F(A, \vec{u_1}, \vec{u_2}) = F(A, \vec{v_1}, \vec{v_2})$, then considering the middle $n$ rows of the two matrices produced via the definition of $F$ we must have that $A\vec{u_1} = A\vec{v_1}$ and $A\vec{u_2} = A\vec{v_2}$. Hence, we have that $A(\vec{v_1}-\vec{u_1}) = \vec{0}$ and $A(\vec{v_2}-\vec{u_2}) = \vec{0}$, so by definition we have $\vec{v_1}-\vec{u_1}, \vec{v_2}-\vec{u_2} \in \ker(A)$.

Conversely, if the vectors $\vec{v_1}-\vec{u_1}$ and $\vec{v_2}-\vec{u_2}$ are both in the kernel of $A$, let $\vec{w_1} = \vec{v_1}-\vec{u_1}$ and $\vec{w_1} = \vec{v_2}-\vec{u_2}$. We will show that the corresponding blocks in each location within the matrices $F(A, \vec{u_1}, \vec{u_2})$ and $F(A, \vec{v_1}, \vec{v_2})$ are equal. Since the blocks along the main diagonal only involve $A$, which is the same for both matrices, it suffices to prove equality for the remaining six pairs of non-diagonal blocks. Moreover, since the top-right and bottom-left blocks in the definition of $F$ are identical, and since $A\vec{u_1} = A\vec{v_1} \implies (A\vec{u_1})^T = (A\vec{v_1})^T$ and $A\vec{u_2} = A\vec{v_2} \implies (A\vec{u_2})^T = (A\vec{v_2})^T$, it suffices to show that the corresponding blocks in positions $(2,1)$, $(2,3)$, and $(1,3)$ of the two block matrices $F(A, \vec{u_1}, \vec{u_2})$ and $F(A, \vec{v_1}, \vec{v_2})$ are equal.

Since $A\vec{v_1} = A(\vec{u_1}+\vec{w_1}) = A\vec{u_1}+A\vec{w_1} = A\vec{u_1}+\vec{0} = A\vec{u_1}$, the corresponding blocks in row $2$, column $1$ are equal. By similar reasoning, we have $A \vec{v_2} = A\vec{u_1}$, so the blocks in row $2$, column $3$ are equal. Finally, since $A \vec{u_1} = A \vec{v_1}$, we have that \[ (A \vec{v_1})^T \bullet \vec{v_2} = (A \vec{u_1})^T \bullet \vec{v_2} = (A \vec{u_1})^T \bullet \left( \vec{u_2} + \vec{w_2} \right) = (A \vec{u_1})^T \bullet \vec{u_2} + (A \vec{u_1})^T \bullet \vec{w_2}. \]
As $A$ is symmetric and all of the main diagonal entries are zeroes, we may apply an argument similar to that used in the proof of Lemma \ref{lem:gcdscount2} to show that $(A \vec{u_1})^T \bullet \vec{w_2} = (A \vec{w})^T \bullet \vec{u_1}$, which is $0$ since $\vec{w} \in \ker(A)$. Thus, we are left with $(A \vec{v_1})^T \bullet \vec{v_2} =  (A \vec{u_1})^T \bullet \vec{u_2}$, so the blocks in row $1$, column $3$ are equal as well. Hence, we have shown that if $\vec{v_1}-\vec{u_1}, \vec{v_2}-\vec{u_2} \in \ker(A)$ then $F(A, \vec{u_1}, \vec{u_2}) = F(A, \vec{v_1}, \vec{v_2})$.
\end{proof}
%-------------------------------------------
A consequence of the above lemma is now presented.
%-------------------------------------------
\begin{cor}\label{cor:diff_in_ker}
Let $A$ be a matrix in $M_n$, and let $\vec{u}, \vec{v} \in \mathbb{F}_2^n$. Then, we have that $F(A, \vec{u}, \vec{u}^C) = F(A, \vec{v}, \vec{v}^C)$ if, and only if, $\vec{v}-\vec{u} \in \ker(A)$.
\end{cor}
%-------------------------------------------
\begin{proof}
By Lemma \ref{lem:gcdscount2}, it suffices to show that $\vec{v}^C-\vec{u}^C \in \ker(A)$ if, and only if, $\vec{v}-\vec{u} \in \ker(A)$. By definition, we have that \[ \vec{v}^C-\vec{u}^C = \left( \vec{v} + [1\ 1\ \dots\ 1]^T \right) - \left( \vec{u} + [1\ 1\ \dots\ 1]^T \right) = \vec{v} - \vec{u}, \] so it follows immediately that $\vec{v}-\vec{u} \in \ker(A) \iff \vec{v}^C-\vec{u}^C \in \ker(A)$. 
\end{proof}
%-------------------------------------------
The properties that have been mentioned in this section have led us to the following theorem which tells us the number of two-rooted \textbf{gcds}-sortable graphs $G$ on $n+2$ vertices under a specific circumstance. It is described in more detail below. 
%-------------------------------------------
\begin{thm}\label{thm:rankforms}
Let $A$ be the adjacency matrix of an undirected graph on $n$ vertices with no self-loops and at most one edge between any two vertices. Then, we have that the number of two-rooted \textbf{gcds}-sortable graphs $G$ on $n+2$ vertices such that $A$ is the adjacency matrix of the subgraph of non-rooted vertices of $G$ is $4^{\text{\normalfont rank}(A)}$, and the number of such graphs that are Eulerian is $2^{\text{\normalfont rank}(A)}$, where $A$ is taken as a matrix over the field $\mathbb F_2$.
\end{thm}
%-------------------------------------------
\begin{proof}
We define the relation $\sim$ on $\mathbb{F}_2^n \times \mathbb{F}_2^n$ as follows: $(\vec{u_1}, \vec{u_2}) \sim (\vec{v_1}, \vec{v_2})$ if, and only if, $F(A, \vec{u_1}, \vec{u_2}) = F(A, \vec{v_1}, \vec{v_2})$. Since $\sim$ is a relation defined in terms of equality, we have that $\sim$ is reflexive, symmetric, and transitive, so $\sim$ is an equivalence relation. Let $[(\vec u_1, \vec u_2 )]$ be the equivalence class containing $(\vec u_1, \vec u_2 )$ under the equivalence relation $\sim$. 

By Lemma \ref{lem:gcdscount4}, we have that $(\vec{u_1}, \vec{u_2}) \sim (\vec{v_1}, \vec{v_2})$ if, and only if, $\vec{v_1} - \vec{u_1}$ and $\vec{v_2}-\vec{u_2}$ are both in the kernel of $A$. Thus, we have that \[ [(\vec{u_1}, \vec{u_2})] = \{ (\vec{u_1}+\vec{w_1}, \vec{u_2}+\vec{w_2}) \ | \ \vec{w_1}, \vec{w_2} \in \ker(A) \}, \] and so $|[(\vec{u_1}, \vec{u_2})]| = |\ker(A)|^2$ for any pair $(\vec{u_1}, \vec{u_2}) \in \mathbb{F}_2^n \times \mathbb{F}_2^n$.  

Thus, the number of distinct equivalence classes under the relation $\sim$ is given by $\left|\ \left( \mathbb{F}_2^n \times \mathbb{F}_2^n \right)\backslash\sim\ \right| = \frac{4^n}{|\ker(A)|^2}$. Since $\ker(A)$ is a subspace of the vector space $\mathbb{F}_2^n$, we have that $\ker(A)$ is isomorphic to the vector space $\mathbb{F}_2^{\dim(\ker(A))}$, and so $|\ker(A)| = 2^{\dim(\ker(A))}$. Therefore, since $\dim(\ker(A)) + \text{rank}(A) = n$ by the Rank-Nullity Theorem, we have that the number of distinct equivalence classes is \[ \frac{4^n}{|\ker(A)|^2} =\frac{4^n}{\left(2^{\dim(\ker(A))}\right)^2} = \frac{4^n}{4^{\dim(\ker(A))}} = 4^{\text{rank}(A)}.\] 

Therefore, there are exactly $4^{\text{rank}(A)}$ distinct $(n +2) \times (n+2)$ matrices $M$ such that $M = F(A, \vec{u_1}, \vec{u_2})$ for some $\vec{u_1}, \vec{u_2} \in \mathbb{F}_2^n$. By Lemma \ref{lem:gcdscount3}, all of these matrices are valid adjacency matrices of undirected simple graphs, and all are \textbf{mcds}-sortable. Also, by Lemma \ref{lem:gcdscount2}, any $(n+2) \times (n+2)$ \textbf{mcds}-sortable matrix $M$ whose central $n \times n$ submatrix is $A$ can be written as $ F(A, \vec{u_1}, \vec{u_2})$ for some $\vec{u_1}$ and $\vec{u_2}$, so there are exactly $4^{\text{rank}(A)}$ \textbf{mcds}-sortable $(n+2) \times (n+2)$ matrices whose central submatrix is $A$. Each of these matrices corresponds to a distinct two-rooted graph with $G$ as its non-rooted subgraph, so there are also $4^{\text{rank}(A)}$ such \textbf{gcds}-sortable graphs. 

To count the number of Eulerian \textbf{gcds}-sortable graphs with central submatrix $A$, we apply a similar procedure, but using single vectors in $\mathbb{F}_2^n$ instead of ordered pairs of vectors. As before, we define an equivalence relation $\sim_E$ on $\mathbb{F}_2^n$ such that $\vec{u} \sim_E \vec{v}$ if, and only if, $F(A, \vec{u}, \vec{u}^C) = F(A, \vec{v}, \vec{v}^C)$, and we consider the equivalence classes $[\vec{u}]$ under $\sim_E$. By Corollary \ref{cor:diff_in_ker}, we have that $[\vec{u}] = \{ \vec{u}+\vec{w} \ | \ \vec{w} \in \ker(A) \}$. Thus, we have $|[\vec{u}]| = |\ker(A)|$. By an analogous argument to the one above, the number of equivalence classes under $\sim_K$ is $$\frac{2^k}{|\ker(A)|} = 2^{\text{rank}(A)}.$$ It follows by Corollaries \ref{cor:gcdscor2} and \ref{cor:gcdscor3} that there is exactly one $(n+2) \times (n+2)$ \textbf{mcds}-sortable Eulerian matrix whose central submatrix is $A$ for each equivalence class, using a similar argument as before, so there are $2^{\text{rank}(A)}$ \textbf{mcds}-sortable Eulerian $(n+2) \times (n+2)$ matrices whose central submatrix is $A$, each corresponding to a \textbf{gcds}-sortable graph with non-rooted $n$-vertex subgraph $G$.
\end{proof}
%-------------------------------------------
To count the total number of \textbf{gcds}-sortable two-rooted graphs on $n$-vertices, we apply Theorem \ref{thm:rankforms} along with the following combinatorial result proven in \cite{macw}.
%-------------------------------------------
\begin{thm}[MacWilliams, 1969, Theorem 3]\label{thm:macw}
Let $N_0(t,r)$ be the number of symmetric $t \times t$ matrices $A$ with entries in $\mathbb{F}(2^n)$ such that all of the main diagonal entries of $A$ are $0$ and $\text{rank}(A) = r$. Then, if $r$ is odd, we have $N_0(t,r) = 0$, and if $r=2s$ we have that \[ N_0(t,r) = \prod_{i=1}^s \frac{2^{2i-2}}{2^{2i}-1} \cdot \prod_{i=0}^{2s-1} (2^{t-i}-1). \]
\end{thm}
%-------------------------------------------
\begin{proof}This result was proven by MacWilliams in \cite{macw}, Theorem 3.
\end{proof}
%-------------------------------------------
From such theorem, we get the following corollary.
%-------------------------------------------
\begin{cor}
The number of \textbf{gcds}-sortable two-rooted graphs on $n$ vertices is \[ \sum_{s=0}^{\lfloor n/2 \rfloor-1} 2^{s(s+3)}\left( \dfrac{\prod_{i=0}^{2s-1} \left(2^{n-2-i}-1\right)}{\prod_{i=1}^s \left(2^{2i}-1\right)} \right). \] Moreover, the number of \textbf{gcds}-sortable two-rooted Eulerian graphs on $n$ vertices is \[ \sum_{s=0}^{\lfloor n/2 \rfloor-1} 2^{\frac{s(s+3)}{2}}\left( \dfrac{\prod_{i=0}^{2s-1} \left(2^{n-2-i}-1\right)}{\prod_{i=1}^s \left(2^{2i}-1\right)} \right). \]
\end{cor}
%-------------------------------------------
\begin{proof}
By Theorem \ref{thm:rankforms}, we have that for each matrix $A \in M_{n-2}$, there are exactly $4^{\text{rank}(A)}$ two-rooted \textbf{gcds}-sortable graphs $G$ such that $A$ is the adjacency matrix of the non-rooted subgraph of $G$, $2^{\text{rank}(A)}$ of which are Eulerian. Therefore, to count the \textbf{gcds}-sortable graphs on $n$ vertices, it suffices to count the number of matrices $A \in M_{n-2}$ with rank $r$ over the field $F_2$ for each $r$ with $0 \le r \le n-2$. 

To do this, we apply Theorem \ref{thm:macw} (proven as Theorem $3$ in \cite{macw}) with $t = n-2$, for each value of $r$ from $0$ to $n-2$. Since there are no solutions in the case where $r$ is odd, the total number of \textbf{gcds}-sortable graphs on $n$ vertices is given by
\begin{align*}
\sum_{s=0}^{\lfloor\frac{n-2}{2} \rfloor} 4^{2s} N_0(n-2, 2s) &= \sum_{s=0}^{\lfloor\frac{n-2}{2} \rfloor}4^{2s} \left( \prod_{i=1}^s \frac{2^{2i-2}}{2^{2i}-1} \cdot \prod_{i=0}^{2s-1} (2^{n-2-i}-1) \right) \\
&= \sum_{s=0}^{\lfloor n / 2\rfloor - 1 }2^{4s} \left( \prod_{i=1}^s \frac{2^{2i-2}}{2^{2i}-1} \cdot \prod_{i=0}^{2s-1} (2^{n-2-i}-1) \right) \\
&= \sum_{s=0}^{\lfloor n / 2\rfloor - 1 }2^{4s} \cdot \prod_{i=1}^s 2^{2i-2} \cdot \left( \frac{\prod_{i=0}^{2s-1} (2^{n-2-i}-1)}{\prod_{i=1}^s (2^{2i}-1)}  \right)\\
&= \sum_{s=0}^{\lfloor n / 2\rfloor - 1 }2^{4s} \cdot 2^{s(s-1)} \cdot \left( \frac{\prod_{i=0}^{2s-1} (2^{n-2-i}-1)}{\prod_{i=1}^s (2^{2i}-1)}  \right) \\
&= \sum_{s=0}^{\lfloor n / 2\rfloor - 1 }2^{s(s+3)} \cdot \left( \frac{\prod_{i=0}^{2s-1} (2^{n-2-i}-1)}{\prod_{i=1}^s (2^{2i}-1)}  \right),
\end{align*} as desired. The derivation of the corresponding formula for Eulerian graphs is similar, but with the initial $4^{2s}$ term replaced with $2^{2s}$.
\end{proof}
%-------------------------------------------
%\subsection{Data on \textbf{gcds}-Sortable Graphs}

\section{The proportion of graphs that are \textbf{gcds} sortable}

Here is a table of the numbers of \textbf{gcds}-sortable graphs and \textbf{gcds}-sortable Eulerian graphs on $n$-vertices, for $3 \le n \le 10$, computed using the formula. 
%The counts for the \textbf{gcds}-sortable graphs for $n \le 8$ were initially computed naively by performing individual \textbf{mcds} operations, and the initial numbers match those produced by the formula.

\begin{center}\begin{tabular}{|c|c|c|c|c|}\hline 
\# vertices& \# total graphs& \# \textbf{gcds}-sortable &ratio&\# \textbf{gcds}-sortable Eulerian  \\\hline
3&8&1&.125&1\\\hline
4&64&17&.266&5\\\hline
5&1,024&113&.110&29\\\hline
6&32,768&7,729&.236&365\\\hline
7&2,097,152&224,689&.107&7,565\\\hline
8&268,435,456&61,562,033&.229&259,533\\\hline
9&68,719,476,736&7,309,130,417&.106&16,766,541\\\hline
10&35,184,372,088,832&8,013,328,398,001&.228&1,695,913,805\\\hline

\end{tabular}\\
    \vspace{2.5mm}
	TABLE 1. \textbf{gcds}-sortable graphs on $n$ vertices, $3 \le n \le 10$ \\
\end{center}

As can be observed in the table above, the proportions of \textbf{gcds}-sortable graphs on even and odd numbers of vertices appear to converge. A proof of this is given in the following lemma:
%-------------------------------------------
\begin{lem}\label{lem:conv}
Let $r_n$ be the proportion of two-rooted graphs on $n$ vertices that are \textbf{gcds}-sortable. By our formula for the number of \textbf{gcds}-sortable graphs on $n$ vertices and the fact that there are $2^{\binom n2}$ total graphs, we have that \[ r_n = \frac{1}{2^{\frac{n(n-1)}{2}}} \sum_{s=0}^{\lfloor n/2 \rfloor-1} 2^{s(s+3)}\left( \dfrac{\prod_{i=0}^{2s-1} \left(2^{n-2-i}-1\right)}{\prod_{i=1}^s \left(2^{2i}-1\right)} \right).\]

Then, we claim that the sequences $\{ r_{2n} | n \in \mathbb N\}$ and $\{ r_{2n+1} | n \in \mathbb N\}$ of the proportions of \textbf{gcds}-sortable graphs on odd and even numbers of vertices converge to finite, positive constants.
\end{lem}
%-------------------------------------------
\begin{proof}
We first consider the sequence $\{ r_{2n} | n \in \mathbb N\}$, corresponding to the case where the number of vertices is even. We begin by rewriting the formula in terms of $r_{2n}$, and we define $x_n = r_{2n}$ for convenience: \[ x_n = r_{2n} = \sum_{s=0}^{n-1} \frac{2^{s(s+3)}}{2^{n(2n-1)}} \left( \frac{\prod_{i=0}^{2s-1} \left( 2^{2n-2-i}-1\right)}{\prod_{i=1}^s \left( 2^{2i}-1 \right)}\right). \]
For each positive integer $n$ and each integer $s$ with $0 \le s < n$, let \[ T(n, s) = \frac{2^{s(s+3)}}{2^{n(2n-1)}} \left( \frac{\prod_{i=0}^{2s-1} \left( 2^{2n-2-i}-1\right)}{\prod_{i=1}^s \left( 2^{2i}-1 \right)}\right). \]

In order to prove that the sequence $\{ x_n \}$ converges, we will prove the following two lemmas about the terms $T(n,s)$.
	%-------------------------------------------
	\begin{lemma}\label{lem:conv1}
	Let $k = \sqrt 2$. Then for each integer $n \ge 10$ and each integer $s$ with $\frac{n}{3} \le s \le n$, we have that \[ 1-k^{-n} < \frac{T(n+1,s+1)}{T(n,s)} < 1+k^{-n}.\]
	\end{lemma}
    %-------------------------------------------
	\begin{proof}
	We begin by simplifying the expression $\frac{T(n+1,s+1)}{T(n,s)}$; we have \begin{align*}
 \frac{T(n+1,s+1)}{T(n,s)} &= \left( \frac{2^{(s+1)(s+4)}}{2^{(n+1)(2n+1)}} \left( \frac{\prod_{i=0}^{2s+1} \left( 2^{2n-i}-1\right)}{\prod_{i=1}^{s+1} \left( 2^{2i}-1 \right)}\right) \right) \cdot \left( \frac{2^{n(2n-1)}}{2^{s(s+3)}} \left( \frac{\prod_{i=1}^s \left( 2^{2i}-1 \right)}{\prod_{i=0}^{2s-1} \left( 2^{2n-2-i}-1\right)}\right)\right) \\
  &= \left( \frac{2^{s^2+5s+4}}{2^{2n^2+3n+1}} \left( \frac{\prod_{i=1}^s \left( 2^{2i}-1 \right)}{\prod_{i=1}^{s+1} \left( 2^{2i}-1 \right)}\right) \right) \cdot \left( \frac{2^{2n^2-n}}{2^{s^2+3s}} \left( \frac{\prod_{i=0}^{2s+1} \left( 2^{2n-i}-1\right)}{\prod_{i=0}^{2s-1} \left( 2^{2n-2-i}-1\right)}\right)\right) \\
  &= \frac{2^{s^2+5s+4}}{2^{s^2+3s}} \cdot \frac{2^{2n^2-n}}{2^{2n^2+3n+1}} \cdot \left( \frac{\prod_{i=1}^s \left( 2^{2i}-1 \right)}{\prod_{i=1}^{s+1} \left( 2^{2i}-1 \right)}\right)  \cdot \left( \frac{\prod_{i=0}^{2s+1} \left( 2^{2n-i}-1\right)}{\prod_{i=0}^{2s-1} \left( 2^{2n-2-i}-1\right)}\right) \\
  &= 2^{2s+4} \cdot \frac{1}{2^{4n+1}} \cdot \frac{1}{2^{2s+2}-1} \cdot \left( \frac{\prod_{i=2n-2s-1}^{2n} \left( 2^{i}-1\right)}{\prod_{i=2n-2s-1}^{2n-2} \left( 2^{i}-1\right)}\right) \\
  &= \frac{2^{2s+4}}{2^{2s+2}-1} \cdot \frac{1}{2^{4n+1}} \cdot \left( 2^{2n}-1\right) \left( 2^{2n-1}-1 \right) \\
  &= \frac{2^{2s+2}}{2^{2s+2}-1} \cdot \frac{(2^{4n-1}-3 \cdot 2^{2n-1}+1)}{2^{4n-1}} \\
  &= \left( 1 + \frac{1}{2^{2s+2}-1}\right) \left( 1 - \frac{3}{2^{2n}} + \frac{1}{2^{4n-1}} \right).
 \end{align*}
Let $k = \sqrt 2$. Then, for $n \ge 10$, $s \ge \frac n3$, we have $$\frac{T(n+1,s+1)}{T(n,s)} \ge 1-\frac{3}{2^{2n}} \ge 1 - 2^{-n} > 1-k^{-n}$$ and $$\frac{T(n+1,s+1)}{T(n,s)} \le 1 + \frac{1}{2^{2s+2}-1} \le 1+\frac{1}{2^{2s}} = 1+\frac{1}{2^{2n/3}} = 1+k^{-4n/3} < 1+k^{-n},$$ as desired.
	\end{proof}
    %-------------------------------------------
	\begin{lemma}\label{lem:conv2}
	Let $k = \sqrt 2$. Then for each integer $n \ge 10$ and each nonnegative integer $s \le \frac{2n}{3}$, we have \[ T(n, s) < k^{-n}.\]
	\end{lemma}
    %-------------------------------------------
    \begin{proof}
    Note that since $2^{2n-2-i}-1 \le 2^{2n-2-i}$ for all $0 \le i < n$ and $2^{2i}-1 \ge 2^{2i-1}$ for all $i \ge 1$, we have that \[T(n, s) =  \frac{2^{s(s+3)}}{2^{n(2n-1)}} \left( \frac{\prod_{i=0}^{2s-1} \left( 2^{2n-2-i}-1\right)}{\prod_{i=1}^s \left( 2^{2i}-1 \right)}\right) \le \frac{2^{s(s+3)}}{2^{n(2n-1)}} \left( \frac{\prod_{i=0}^{2s-1} \left( 2^{2n-2-i}\right)}{\prod_{i=1}^s \left( 2^{2i-1} \right)}\right). \] Therefore, we have that \begin{align*}\log_2 T(n,s) &\le s(s+3)-n(2n-1) + \sum_{i=0}^{2s-1} (2n-2-i) - \sum_{i=1}^s (2i-1) \\
&= s^2+3s-2n^2+n+4ns-4s-\left( \sum_{i=0}^{2s-1} i\right)  - 2\left( \sum_{i=1}^s i \right) + s\\
&= s^2+3s-2n^2+n+4ns-4s-\frac{2s(2s-1)}{2} - 2\frac{s(s+1)}{2} + s \\
&= s^2+3s-2n^2+n+4ns-4s-2s^2+s-s^2-s+s \\
&= -2n^2+4ns-2s^2+n \\
&= -2(n-s)^2+n
\end{align*}
Using our assumption that $s \le \frac{2n}{3}$, we have $n-s \ge \frac{n}{3}$, so $-2(n-s)^2+n \le -2n^2/9+n$. For $n \ge 10$, we have $$-2n^2/9+n < -2n+n = -n.$$ Therefore, it follows that $T(n,s) < 2^{-n} < \left( \sqrt 2 \right)^{-n}$ for $n \ge 10$, $s \le \frac{2n}{3}$, so the value $k = \sqrt 2$ satisfies the condition for our desired constant.
	\end{proof}
	%-------------------------------------------
Let $k = \sqrt 2$. By Lemmas \ref{lem:conv1} and \ref{lem:conv2}, for each integer $n$ with $n \ge 10$ we have $1-k^{-n} < \frac{T(n+1,s+1)}{T(n,s)} < 1+k^{-n}$ for $s \ge n/2$, $T(n+1,s) < k^{-n}$ for $s \le n/2$, and $T(n,s) < k^{-n}$ for $s \le n/2$.

Multiplying the inequality $1-k^{-n} < \frac{T(n+1,s+1)}{T(n,s)} < 1+k^{-n}$ through by $T(n,s)$ and then subtracting $T(n,s)$ from each term, we have that $-T(n,s) k^{-n} < T(n+1,s+1)-T(n,s) < T(n,s)k^{-n}$. Since $T(n,s) \le x_n \le 1$ (as the proportion of \textbf{gcds}-sortable graphs cannot be greater than $1$), we must have that $$|T(n+1,s+1)-T(n,s)| < k^{-n}.$$

We expand the difference $|x_{n+1}-x_n|$ using the Triangle Inequality:
\begin{align*}
|x_{n+1}-x_n|&=\left| \sum_{s=0}^n T(n+1,s) - \sum_{s=0}^{n-1} T(n,s)\right| \\
&\le \left| \sum_{s=\lfloor \frac{n}{2}\rfloor+1}^{n-1} T(n+1,s+1)-T(n,s) \right| + \left| \sum_{s=0}^{\lfloor \frac{n}{2}\rfloor+1} T(n+1,s) \right| + \left| \sum_{s=0}^{\lfloor \frac{n}{2}\rfloor} T(n,s) \right| \\
&\le \sum_{s=\lfloor \frac{n}{2}\rfloor+1}^{n-1} \left|T(n+1,s+1)-T(n,s)\right| + \sum_{s=0}^{\lfloor \frac{n}{2}\rfloor+1}\left| T(n+1,s) \right| + \sum_{s=0}^{\lfloor \frac{n}{2}\rfloor} \left| T(n,s) \right| \\
&\le (n-\lfloor n/2 \rfloor -1)k^{-n} + \left( \lfloor n/2 \rfloor +2 \right) k^{-n} +  \left( \lfloor n/2 \rfloor +1 \right) k^{-n}\\
&= \left( n+\lfloor n/2 \rfloor + 2\right) k^{-n}\\
&\le 3nk^{-n}. \\
\end{align*}
Let $d = \frac{k-1}{2}$, and let $c = \frac{k}{d+1}$. Note that since $k > 1$, we have $d+1 = \frac{k+1}{2} < k$, so $c > 1$. Now let $T = \frac{3}{d} = \frac{6}{k-1}$. Since $k$ is a constant not depending on $n$, both $T$ and $c$ are positive constants not depending on $n$ (and $c > 1$). Then, since $d > 0$, we have $$3nk^{-n} = \frac{3}{d}k^{-n} (nd) \le \frac{3}{d} k^{-n} \sum_{j=0}^n \binom nj d^j = \frac{3}{d} k^{-n}(d+1)^n =  \frac{3}{d} \cdot \left( \frac{k}{d+1} \right)^{-n} = T \cdot c^{-n}. $$
Thus, we have shown that $|x_{n+1}-x_n| < Tc^{-n}$ for each integer $n \ge 10$ and constants $T > 0, c > 1$.

Let $a_1 = x_1$, and for each integer $n \ge 2$, $a_n = x_n-x_{n-1}$. Then, for each $n>0$ we have that $x_n = \sum_{i=1}^n a_i$. Let $b_n = Tc^{-(n-1)}$ for each $n > 0$. Then, $|a_n| \le b_n$ for all $n \ge 10$. Since the geometric series $\sum_{n=1}^\infty b_n$ converges, we have that the series $\sum_{n=1}^\infty a_n$ converges as well. Therefore, the sequence $\{x_n\ | \ n \in \mathbb N \} = \{r_{2n}\ | \ n \in \mathbb N \}$ converges.
% "by the direct comparison test" was removed by editor #13

Let $L = \displaystyle{\lim_{n \to \infty}} x_n$. Then, $$L = x_{100} + \sum_{n=100}^{\infty} (x_{n+1}-x_{n}) \ge x_{100} - \sum_{n=100}^{\infty} |x_{n+1}-x_n| \ge x_{100} - \sum_{n=100}^{\infty} T \cdot c^{-n} \ge x_{100} - c^{-100} \frac{cT}{c-1}.$$

With $k = \sqrt 2$, we have $c = \frac{2\sqrt 2}{\sqrt 2 + 1} = 4-2\sqrt 2 > 1.1$ and $T = \frac{6}{\sqrt 2-1} < \frac{6}{0.3} = 20$, and also $c < 2$. Thus, we have $\frac{cT}{c-1} < \frac{2 \cdot 20}{0.1} = 400$. Additionally, we have that $c^{-100} < 1.1^{-100} < 10^{-4}$, so $c^{-100}\frac{cT}{c-1} < 0.04$. By computing the value of $x_{100}$ through inserting $n = 100$ into the formula for $x_n$. We can verify that $x_{100} > 0.2$ and so $L \ge 0.2-0.04 \ge 0.16 > 0$. Therefore, the limiting value of the sequence, which exists by our convergence result, must be strictly positive. 

The proof that the sequence $\{x_n\ | \ n \in \mathbb N \} = \{r_{2n+1}\ | \ n \in \mathbb N \}$ converges to a positive limit is similar. % (editor 13 suggested removing the rest) essentially identical, except with the $2^{n(2n-1)}$ term in the denominator of $T(n,s)$ replaced with $2^{n(2n+1)}$ and the product in the numerator changed from $\prod_{i=0}^{2s-1} \left( 2^{2n-2-i}-1 \right)$ to $\prod_{i=0}^{2s-1} \left( 2^{2n-1-i}-1 \right)$.
\end{proof}
%-------------------------------------------
%-------------------------------------------
%-------------------------------------------

%-------------------------------------------
%-------------------------------------------
\appendix
\section{Equivalence of \textbf{gcds} definitions}
%-------------------------------------------
%-------------------------------------------
We now provided an equivalence among a definition we provided in the present paper with a definition presented in another reference.
%-------------------------------------------
\begin{lem}\label{lem:gcds_eq_defn}
The definition of \textbf{gcds} given as Definition \ref{defn:gcds} is equivalent to the definition given in \cite{p4}. 
\end{lem}
%-------------------------------------------
\begin{proof}
We proceed by casework, considering whether edge $\{u,v\}$ will be present in the graph $\textbf{gcds}_{\{p,q\}}(G)$ for each case in the definition from \cite{p4} governing the adjacency relationships between the vertices $u,v,p,$ and $q$.

Our vertices $p$ and $q$ correspond to the vertices $x$ and $y$ in the  \cite{p4} definition used to apply \textbf{gcds}, and our vertices $u$ and $v$ correspond to the arbitrary vertices $p$ and $q$ in the  \cite{p4} that are affected by the \textbf{gcds} operation.\\

\begin{enumerate}
\item Case A: At least one of the vertices $u$ and $v$ is adjacent to neither $p$ nor $q$. 

In this case, we will either have $f_p(u)=f_q(u)=0$ or $f_p(v)=f_q(v) = 0$. In both situations, it follows that $f_p(u)f_q(v)+f_q(u)f_p(v) = 0$, so the edge $\{u,v\}$ is in $E'$ if, and only if, $\{u,v\} \in E$. At least one of $u$ and $v$ will not appear in the master list $M(p,q)$ in this case (whichever one was adjacent to neither $p$ nor $q$), so this condition matches the specification given in the  \cite{p4} definition. \\

\item Case B: Both $u$ and $v$ appear in the master list $M(p,q)$. 

\begin{enumerate}
\item Subcase B1: $u$ and $v$ do not appear in the same column of the master list. 

The  \cite{p4} definition specifies that $\{u,v\} \in E'$ if, and only if, $\{u,v\} \not \in E$. This subcase corresponds to the case when there is no vertex in the set $\{p,q\}$ that $u$ and $v$ are both adjacent to, so if $f_p(u) = 1$, then $f_q(u) = 0$, $f_q(v) = 1$, and $f_p(v) = 0$. Alternatively, if $f_p(u) = 0$, then $f_q(u) = 1$, $f_q(v) = 0$, and $f_p(v) = 1$. In either case, we have that $f_p(u)f_q(v)+f_q(u)f_p(v) = 1$, so our definition specifies that $\{u,v\} \in E'$ if and only $\{u,v\} \in E$, as desired.\\

\item Subcase B2: $u$ and $v$ appear in the same column of the master list, and the sum of the number of times $u$ appears and the number of times $v$ appears is even. 

First, we note that the  \cite{p4} definition lets $\{u,v\}$ be an edge in $E'$ if, and only if, $\{u,v\} \in E$. We now consider the possible values of $f_p(u),f_q(u),f_p(v),$ and $f_q(v)$. 

Since $u$ and $v$ both appear in the same column of the master list, they are both adjacent to $w$ for a specific vertex $w \in \{ p,q \}$. Furthermore, since the number of adjacencies is even, either none or both of $u$ and $v$ are adjacent to the other vertex in $\{ p,q \}$. 

This leaves a total of $3$ possibilities. If both are adjacent to the other vertex for either choice of $w$, we have that $f_p(u)=f_p(v)=f_q(u)=f_q(v) = 1$, so $f_p(u)f_q(v)+f_p(v)f_q(u) \not= 1$ and we have $uv \in E'$ if, and only if, $uv \in E$. Otherwise, both $u$ and $v$ are adjacent to $w$ but neither is adjacent to the other vertex. If $w = p$, then $f_p(u) = f_p(v) = 1$ and $f_q(u) = f_q(v) = 0$, so we have $f_p(u)f_q(v)+f_q(u)f_p(v) \not= 1$; similarly, if $w = q$, then $f_p(u)=f_p(v) = 0$ and $f_q(u)=f_q(v)=1$, so $f_p(u)f_q(v)+f_q(u)f_p(v) \not= 1$. In either case, we have $\{u,v\} \in E'$ if, and only if, $\{u,v\} \in E$, so the definitions are always equivalent in subcase B2.\\ 

\item Subcase B3: $u$ and $v$ appear in the same column of the master list, and the sum of the number of times $u$ appears and the number of times $v$ appears is odd. 

The  \cite{p4} definition lets $\{u,v\} \in E'$ if, and only if, $\{u,v\} \not \in E$. Note that since $u$ and $v$ are both adjacent to some $w \in \{p,q \}$ (since they appear in the same column of the master list), at least two of $f_p(u),f_q(u),f_p(v),$ and $f_q(v)$ are equal to $1$. Since the sum of the numbers of times $p$ and $q$ appear in the master list is odd, we have that an odd number of $f_p(u),f_q(u),f_p(v),$ and $f_q(v)$ are equal to $1$. Hence, exactly three of the quantities $f_p(u),f_q(u),f_p(v),$ and $f_q(v)$ are equal to $1$, and the fourth is equal to zero. Thus, one of the two products in the expression $f_p(u)f_q(v)+f_q(u)f_p(v)$ will contain two ones, and the other term will contain one $1$ and one $0$. Therefore, we must have $f_p(u)f_q(v)+f_q(u)f_p(v)=1$, so our \textbf{gcds} definition lets $\{u,v\} \in E'$ iff $\{u,v\} \not \in E$ as well for subcase B3.\\

\end{enumerate}
\end{enumerate}

Finally, if $u$ or $v$ is an element of $\{ p,q \}$, we consider the case in which $u = p$ without loss of generality. Then, we have $f_p(u) = 0$ and $f_q(u) = 1$, so we have that $f_p(u)f_q(v)+f_q(u)f_p(v)=1$ if, and only if, $f_p(v) = 1$ - in other words, if $v$ and $u$ are adjacent. In this case (when $\{u,v\} \in E$), we will have $\{u,v\} \in E'$ if, and only if, $\{u,v\} \not \in E$, so $\{u,v\} \not \in E'$. If $\{u,v\} \not \in E$, we will have $f_p(u)f_q(v)+f_q(u)f_p(v)\not =1$, so $\{u,v\} \in E'$ if, and only if, $\{u,v\} \in E$ and so $\{u,v\} \not \in E'$. Therefore, either way we have that edge $\{u,v\}$ is not in $E'$. By similar reasoning, we conclude that there is no edge in $E'$ whose endpoints include either vertex $p$ or vertex $q$, so $p$ and $q$ are isolated vertices in $G'$. This aligns with the specification given in the  \cite{p4} definition. 

Thus, our definition corresponds to the definition of  \cite{p4} in each case, so the two definitions are equivalent.

\end{proof}
%-------------------------------------------
\begin{comment}
%-------------------------------------------
%\documentclass{article}
\documentclass[12pt]{article}
%-------------------------------------------
\usepackage[margin=1in]{geometry}
%\usepackage[margin=1in]{geometry}
\usepackage[utf8]{inputenc}
\usepackage[english]{babel}
\usepackage{amsmath,amsfonts,amsthm,verbatim, graphicx,tikz}
\usepackage{amssymb, sectsty, mathrsfs, mathtools, centernot}
%-------------------------------------------
\usetikzlibrary{shapes.geometric,decorations.pathreplacing,angles,quotes}
\end{comment}
%-------------------------------------------

\begin{thebibliography}{99} 
	\bibitem%[Adamyk, 2013]
	{p3} 
	K.L.M. Adamyk, E. Holmes, G.R. Mayfield, D.J. Moritz, M. Scheepers, B.E. Tenner, and H.C. Wauck, 
    \newblock \emph{Sorting Permutations: Games, Genomes, and Cycles},
    \newblock  \textbf{Discrete Mathematics, Algorithms and Applications} \normalfont \textbf{9:5} \normalfont (2017), 1750063 (31 pp) %, arXiv:1410.2353v3 [math.CO]
        
     \bibitem{BH} R. Brijder and H.J. Hoogeboom,
     \newblock \emph{The Group Structure of Pivot and Loop Complementation on Graphs and Set Systems},
     \newblock \textbf{European Journal of Combinatorics } \textbf{32} (2011),  1353 - 1367
   
     \bibitem{BHH} R. Brijder, T. Harju and H.J. Hoogeboom, 
     \newblock \emph{Pivots, Determinants, and Perfect Matchings of Graphs}
     \newblock \textbf{Theoretical Computer Science} \textbf{454} (2012),  64 - 71.  
        
    \bibitem{christie} D.A. Christie,     
    \newblock \emph{Sorting Permutations by Block Interchanges},
    \newblock  \textbf{Information Processing Letters} \normalfont \textbf{60} \normalfont (1996), 165 - 169 %, arXiv:1410.2353v3 [math.CO]
       
        
     \bibitem{ehppr}   A. Ehrenfeucht, T.Harju, I. Petre, D.M. Prescott and G. Rozenberg, 
     \newblock \emph{Computation in Living Cells: Gene Assembly in Ciliates}, 
     \newblock \textbf{Springer-Verlag} 2004.
        
        
    \bibitem{Geelen} J.F. Geelen, 
    \newblock \emph{A generalization of Tutte's characterization of totally unimodular matrices},
    \newblock \textbf{Journal of Combinatorial Theory, Series B} \textbf{70} (1997), 101 - 117.
        
%	\bibitem%[Axler, 2015] %Not cited
%	{axler} S. Axler.
%    \newblock \emph{Linear Algebra Done Right, 3rd edition.}, 
%    \newblock \textbf{Springer Undergraduate Texts in Mathematics}, 2015.
    
%    \bibitem%[Gaetz, 2016] %Not cited
%    {p5} M. Gaetz, B. Molokach, M. Scheepers and M.A. Shanks,		
%	\newblock \emph{Quantifying \textbf{cds} Sortability of Permutations Using Strategic Pile Size},
%	\newblock arXiv1811.11937 
    
%    \bibitem%[Hartman, 2006] %Not cited
%    {p7} Hartman T. and Verbin E.,
%    \newblock \emph{Matrix Tightness: A Linear-Algebraic Framework for Sorting By Transpositions}
%    \newblock{In: Crestani F., Ferragina P., Sanderson M. (eds) String Processing and Information Retrieval. SPIRE 2006. Lecture Notes in Computer Science, vol 4209. Springer, Berlin, Heidelberg, (2006), 279 - 290.}
    
    \bibitem%[Jansen, 2014]
    {p4} C.L. Jansen, M. Scheepers, S.L. Simon, and E. Tatum
    \newblock \emph{Permutation Sorting and a Game on Graphs}, 
    \newblock \textbf{arXiv}:1411.5429v1 [math.CO]
    
    \bibitem%[Li, 2015]
    {p2} H.Q. Li, J. Ramsey, M. Scheepers, H.E. Schilling, and C. Stanford
    \newblock \emph{Parity Cuts  In  2-Rooted  Graphs:  A Trichotomy Theorem} 
    \newblock (in preparation) 
    
%	\bibitem[Zermelo, 1913]{p6} E. Zermelo 
%    \newblock \emph{\"Uber eine Anwendung der Mengenlehre auf die Theorie des Schachspiels} 				
%    \newblock{Proceedings of the Fifth International Congress of Mathematics (1913), 501-504}
    
    \bibitem%[MacWilliams, 1969]{
    {macw} J.~MacWilliams
    \newblock \emph{Orthogonal matrices over finite fields}
    \newblock \textbf{The American Mathematical Monthly} \normalfont \textbf{76} \normalfont  (1969), no. 2, 152--164
    
    \bibitem{per} D.M. Prescott, A. Ehrenfeucht and G. Rozenberg, 
    \newblock \emph{Template guided recombination for IES elimination and unscrambling of genes in stichotrichous ciliates}, 
    \newblock \textbf{Journal of Theoretical Biology} \textbf{222} (2003), 3232 - 330
    
    
    \bibitem{oeis} N. J. A. Sloane, 
    \newblock \emph{The On-line Encyclopedia of Integer Sequences}, 
    \newblock published electronically at {\tt http://www.oeis.org/.}


%	\bibitem%[Hefferon, 2017] %Not cited
%	{heff} J. Hefferon.
%	\newblock \emph{Linear Algebra, 3rd edition.}, 2017. http://joshua.smcvt.edu/linearalgebra/

%	\bibitem%[Ruohonen, 2013] %Not cited
%	{ruoh} K. Ruohonen. 
%	\newblock \emph{Graph Theory}, 2013. 

%Rashmika Tentative Source List:
% Ruohonen
% https://en.wikipedia.org/wiki/Cycle_basis#cite_note-5
% http://students.imsa.edu/~tliu/Math/GraphTheoryIII.pdf
% https://en.wikipedia.org/wiki/Rank%E2%80%93nullity_theorem
% https://en.wikipedia.org/wiki/Sylvester_domain
% https://www3.nd.edu/~dgalvin1/40210/40210_F12/CGT_early.pdf
% %https://books.google.com/books?id=cyX32q8ZP5cC&lpg=PA12&dq=rank%20of%20skew-symmetric%20matrix&pg=PA12#v=onepage&q=rank%20of%20skew-symmetric%20matrix&f=false
\end{thebibliography}
\end{document}